\documentclass[reqno,11pt]{amsart}
\usepackage{amsfonts,amssymb,amsmath,amsthm,gensymb,color,epic,eepic,float,fullpage,combelow}
\usepackage[mathscr]{euscript}
\usepackage[T1]{fontenc}
\usepackage{mathtools}
\let\savedegree\degree
\let\degree\relax
\usepackage{mathabx}
\let\degree\savedegree
\usepackage{comment}
\usepackage{stmaryrd}
\usepackage{hyperref}
\usepackage{dsfont}
\usepackage{enumitem}
\usepackage{cite}
\usepackage{combelow}
\usepackage{youngtab}

\numberwithin{equation}{section}
\allowdisplaybreaks
\usepackage{tikz}
\usetikzlibrary{graphs,patterns,decorations.markings,arrows,matrix}
\usetikzlibrary{calc,decorations.pathmorphing,decorations.markings,
decorations.pathreplacing,patterns,shapes,arrows}
\tikzstyle{block} = [rectangle,draw,text width=10em,text centered,rounded corners,minimum height=4em]
\tikzstyle{line} = [draw, -latex']
\tikzset{arrow/.style={postaction={decorate,thick,decoration={markings,mark = at position #1 with {\arrow{>}}}}},arrow/.default=1}
\tikzset{arrow/.style={postaction={decorate,thick,decoration={markings,mark = at position #1 with {\arrow{>}}}}},arrow/.default=0.5}
\tikzset{invarrow/.style={postaction={decorate,thick,decoration={markings,mark = at position #1 with {\arrow{<}}}}},invarrow/.default=0.5}
\tikzset{gcol/.style={black!35!green}}
\newcommand{\rcheck}[2]{
  \begin{scope}[shift={(#1,#2)}]
    \draw (0,0) -- (1,1);
    \draw (1,0) -- (0,1);
  \end{scope}
}
\newcommand{\rcheckarrows}[2]{
  \begin{scope}[shift={(#1,#2)}]
    \draw[invarrow=0.75] (0,0) -- (1,1);
    \draw[invarrow=0.75] (1,0) -- (0,1);
  \end{scope}
}

\newcommand{\rcheckarrowleft}[2]{
  \begin{scope}[shift={(#1,#2)}]
    \draw (0,0) -- (1,1);
    \draw[invarrow=0.75] (1,0) -- (0,1);
  \end{scope}
}

\newtheorem{defn}{Definition}
\newtheorem{lem}{Lemma}
\newtheorem{thm}{Theorem}
\newtheorem{conj}{Conjecture}
\newtheorem{cor}{Corollary}
\newtheorem{rmk}{Remark}
\newtheorem{prop}{Proposition}
\newtheorem{ex}[thm]{Example}

\def\leq{\leqslant}
\def\geq{\geqslant}
\def\i{\imath}

\newcommand{\bra}[1]{\left\langle #1\right|}
\newcommand{\ket}[1]{\left|#1\right\rangle}

\newcommand{\ba}{\[\begin{aligned}~}
\newcommand{\ea}{\end{aligned}\]}

\def\id{\normalfont{\text{id}}}
\def\Tr{\normalfont{\text{Tr}}}

\newcommand{\BC}{\mathbb{C}}
\newcommand{\BN}{\mathbb{N}}

\newcommand{\BZ}{\mathbb{Z}}

\newcommand{\CA}{\mathcal{A}}
\newcommand{\CB}{\mathcal{B}}
\newcommand{\CS}{\mathcal{S}}
\newcommand{\CV}{\mathcal{V}}

\newcommand{\bd}{\boldsymbol{d}}
\newcommand{\bs}{\boldsymbol{\varsigma}}
\newcommand{\bde}{\boldsymbol{\delta}}

\newcommand{\zz}{\BZ^I}

\newcommand{\vdeg}{\text{vdeg }}
\newcommand{\hdeg}{\text{hdeg }}

\newcommand{\fsl}{\mathfrak{sl}}
\newcommand{\fgl}{\mathfrak{gl}}

\newcommand{\su}{U_t(\dot{\fsl}_n)}
\newcommand{\uu}{U_t(\dot{\fgl}_n)}
\newcommand{\uui}{U_t(\dot{\fgl}_1)}
\newcommand{\suu}{U_t(\dot{\fgl}_{n|m})}

\newcommand{\UU}{U_{q,t}(\ddot{\fgl}_n)}

\title{Shuffle algebras, lattice paths and quantum toroidal $\mathfrak{gl}_{n|m}$}
\author{Alexandr Garbali and Andrei Negu\cb{t}}

\address{Alexandr~Garbali, School of Mathematics and Statistics, University of Melbourne, Australia}
\email{alexandr.garbali@unimelb.edu.au}
\address{Andrei~Negu\cb t, \'{E}cole Polytechnique F\'ed\'erale de Lausanne (EPFL), Lausanne, Switzerland
\newline \text{} \quad Simion Stoilow Institute of Mathematics (IMAR), Bucharest, Romania}
\email{andrei.negut@gmail.com}

\begin{document}
\begin{abstract}
We describe and compute various families of commuting elements of the matrix shuffle algebra (\!\!\cite{Ng_tale}) of type $\fgl_{n|m}$, which is expected to be isomorphic to quantum toroidal $\fgl_{n|m}$ (\!\!\cite{BM, Ng_reduced}). Our formulas are given in terms of partial traces of products of $R$-matrices of the quantum affine algebra $\suu$, and have a lattice path interpretation following \cite{GG}. Our calculations are based on the machinery of the quantum toroidal algebras and a new anti-homomorphism between matrix shuffle algebras.
\end{abstract}
%=====================================================================
\maketitle

\section{Introduction}
Shuffle algebras associated with quantum toroidal algebras have been an active topic of research in the past years. They appear in the construction of quantum toroidal algebras and in their geometric representations \cite{FT-shuffle,FiT,SV-Hall,Ng_Laumon}, calculations of integrals of motion \cite{FO,FHSSY,FJM_integrals,FJM_ssym} and the Bethe ansatz \cite{LV,FJMM_BA}.

The quantum toroidal algebra of type $\fgl_n$ has two (in a sense orthogonal) realizations as double shuffle algebras: the first of these, denoted by $\mathcal{S}$, is more standard and traces its origin to \cite{FO} and \cite{E}; we recall it in Subsection \ref{sub:classic shuffle}. The second realization, denoted by $\mathcal{A}$, was introduced in \cite{Ng_tale} and has been less studied; it is a vector space of rational functions with matrix valued coefficients living in $\text{End}(V\otimes \cdots \otimes V)$ with $V\cong\mathbb C^n$. For this reason we refer to $\mathcal{A}$ as the {\it matrix} shuffle algebra and by this we distinguish it from $\mathcal{S}$ \footnote{For $n=1$ the two algebras $\mathcal{S}$ and $\mathcal{A}$ coincide.}. Such shuffle algebra presentations of quantum toroidal algebras offer a number of technical advantages. Instead of having abstract generators, relations and multiplication rules, in the shuffle presentation one has explicit rational functions and an explicit multiplication, called the {\it shuffle multiplication}. A notable example of the advantage offered by the shuffle algebra presentation is given by a description of commutative subalgebras of toroidal algebras \cite{FHSSY,FJM_integrals,FT_Bethe}. This is best illustrated in the case of the quantum toroidal $\fgl_1$ in \cite{FHSSY}, where the authors provided explicit formulas for several families of elements of the commuting shuffle algebra and proved their commutativity with respect to the shuffle multiplication. 

There exists an alternative approach to the commutative subalgebra of the $\fgl_1$ shuffle algebra due to \cite{GZJ,GG}. In this approach one computes partition functions of ensembles of coloured paths on a finite portion of the square lattice with specific boundary conditions and local Boltzmann weights given by the matrix elements of the $R$-matrix of $\suu$. These partition functions produce families of rational functions which belong to the space of $\fgl_1$ commuting shuffle algebra. By tuning the boundary conditions appropriately one can compute different families of shuffle algebra elements, including those of \cite{FHSSY}. In \cite{GG} the boundary conditions are chosen to be ``conic'', which means that the lattice partition functions can be written as a partial trace of products of $R$-matrices. Such trace formulas can be manipulated to reproduce the definition of the shuffle product, thus giving a new way of proving shuffle commutativity.

In this paper we generalize the results of \cite{GG} to the case of the commuting subalgebra of the matrix shuffle algebra of type $\fgl_{n|m}$. Our generalization relies on Conjecture \ref{conj:super}, which states that the quantum toroidal $\fgl_{n|m}$ is isomorphic to the double matrix shuffle algebra. This Conjecture was proved in the case $m=0$ in \cite{Ng_tale}. A key element in our approach is the construction of explicit anti-homomorphisms (of independent interest) between the matrix shuffle algebra of type $\fgl_{n'|m'}$ and the matrix shuffle algebra of type $\fgl_{n|m}$.

\subsection{Conic partition function}
Consider a cone pointing upwards and the square lattice drawn on it as follows. Draw $N$ directed parallel lines which wrap around the tip of the cone and then self-intersect:
%%%%% ------------------ 
%%%%% ------------------ begin cone 
%%%%% ------------------ 
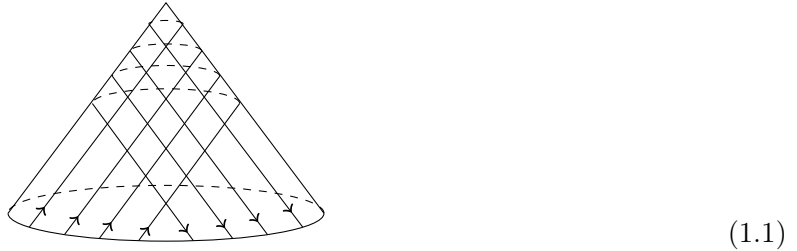
\begin{align}\label{eq:cone}
% \begin{center}
\begin{tikzpicture}[scale=0.7,transform shape]
\newcommand{\hht}{4}  % Set hight
\newcommand{\wt}{3}  % Set width/2 
\newcommand{\hwt}{0.5\wt}  % Set width/2/2 
\pgfmathsetmacro{\hratio}{\hht / \wt}  % half ratio
\pgfmathsetmacro{\ihratio}{\hwt / \hht}  % inverse half ratio
%% begin cone 
  \draw[dashed] (2*\wt,0) arc [start angle=0, end angle=180, x radius=\wt cm, y radius=\hwt cm];
    \draw (0,0) arc [start angle=-180, end angle=0, x radius=\wt cm, y radius=\hwt cm];
     \draw (0,0) -- (\wt,\hht) -- (2*\wt,0);
%% end cone      
%% points on cone
%% begin lattice
\foreach \angle in {30,50,65,80}{
  \pgfmathsetmacro{\xa}{\wt+\wt*cos(-\angle)}  % Compute x coordinate
  \pgfmathsetmacro{\ya}{\hwt*sin(-\angle)}  % Compute y coordinate
  \pgfmathsetmacro{\xb}{\wt+\wt*cos(\angle-180)}  % Compute x coordinate
  \pgfmathsetmacro{\yb}{\hwt*sin(\angle-180)}  % Compute y coordinate
  %\node[below] at (\xa,\ya) {$0$}; 
  %\node[below] at (\xb,\yb) {$0$}; 
  % begin find grid points on edge 
  \pgfmathsetmacro{\xxa}{0.5*\xa+\ihratio*\ya-0.09}
  \pgfmathsetmacro{\yya}{0.5*\hratio*\xa+0.5*\ya}
  \pgfmathsetmacro{\xxb}{0.5*\xb+\wt-\ihratio*\yb+0.09}
  \pgfmathsetmacro{\yyb}{0.5*\hratio*\xa + 0.5*\ya}
  % end find grid points on edge 
  % \fill (\xxa,\yya) circle (1pt); % Plot grid points on edge 
  % \fill (\xxb,\yya) circle (1pt);  % Plot grid points on edge
\draw[invarrow=0.1] (\xa,\ya) -- (\xxa,\yya) ; % draw grid lines
  \draw[arrow=0.1] (\xb,\yb) -- (\xxb,\yyb) ; % draw grid lines
% \node at (\xa,\ya) {$\bullet$}; 
% \node at (\xb,\yb) {$\bullet$};   
  % Plot trace lines
  \pgfmathsetmacro{\xxr}{0.5*\xxa-0.5*\xxb}
  \pgfmathsetmacro{\yyr}{-0.2*\xxr}
  \draw[dashed] (\xxb,\yyb) arc [start angle=180, end angle=0, x radius=\xxr cm, y radius=\yyr cm];
  }
%% end lattice
\end{tikzpicture}
% \end{center}
\end{align}
%%%%% ------------------ 
%%%%% ------------------ end cone
%%%%% ------------------ 
The points where the lattice intersects the ``base'' of the cone will be called {\it boundary points}. On this square lattice we draw coloured paths which are either closed or begin and end on two boundary points. The former will sometimes be called loops and the latter boundary paths. The colours of paths will be labelled by the set $\mathcal{I}=\{1\ldots n+m\}$ where the labels $i$ have an associated $\mathbb Z_2$ grading given by $(-1)^{\delta_{i>n}}$\footnote{We adopt this grading in the introduction for simplicity; a more general grading will be assumed in the main text.} with $\delta_{\text{True}}=1$ and $\delta_{\text{False}}=0$. This grading leads to the interpretation of ``bosonic'' and ``fermionic'' paths. Drawing a coloured path on this lattice should follow the directions indicated by the arrows in \eqref{eq:cone}.  In the following example we have a configuration with one red loop, one green loop, one green boundary path and one blue boundary path: 
%%%%% ------------------ 
%%%%% ------------------ begin cone 
%%%%% ------------------ 
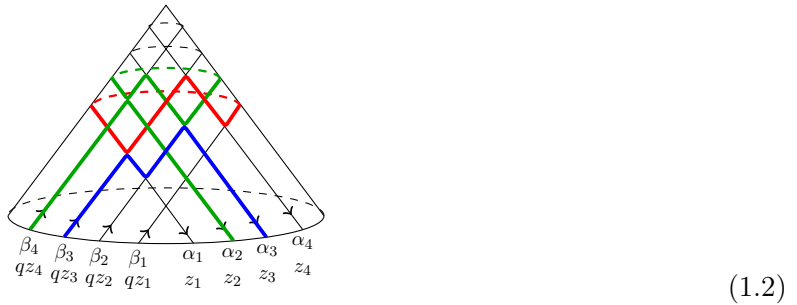
\begin{align}\label{eq:cone_paths}
% \begin{center}
\begin{tikzpicture}[scale=0.7,transform shape]
\newcommand{\hht}{4}  % Set hight
\newcommand{\wt}{3}  % Set width/2 
\newcommand{\hwt}{0.5\wt}  % Set width/2/2 
\pgfmathsetmacro{\hratio}{\hht / \wt}  % half ratio
\pgfmathsetmacro{\ihratio}{\hwt / \hht}  % inverse half ratio
%% begin cone 
  \draw[dashed] (2*\wt,0) arc [start angle=0, end angle=180, x radius=\wt cm, y radius=\hwt cm];
    \draw (0,0) arc [start angle=-180, end angle=0, x radius=\wt cm, y radius=\hwt cm];
     \draw (0,0) -- (\wt,\hht) -- (2*\wt,0);
%% end cone      
%% points on cone
%% begin lattice
\foreach \angle/\lab in {30/4,50/3,65/2,80/1}{
  \pgfmathsetmacro{\xa}{\wt+\wt*cos(-\angle)}  % Compute x coordinate
  \pgfmathsetmacro{\ya}{\hwt*sin(-\angle)}  % Compute y coordinate
  \pgfmathsetmacro{\xb}{\wt+\wt*cos(\angle-180)}  % Compute x coordinate
  \pgfmathsetmacro{\yb}{\hwt*sin(\angle-180)}  % Compute y coordinate
  \node[below] at (\xa,\ya) {$\alpha_\lab$}; 
  \node[below] at (\xb,\yb) {$\beta_\lab$}; 
  \node[below] at (\xa,\ya-0.5) {$z_\lab$}; 
  \node[below] at (\xb,\yb-0.5) {$q z_\lab$}; 
  % begin compute grid points on edge 
  \pgfmathsetmacro{\xxa}{0.5*\xa+\ihratio*\ya-0.09}
  \pgfmathsetmacro{\yya}{0.5*\hratio*\xa+0.5*\ya}
  \pgfmathsetmacro{\xxb}{0.5*\xb+\wt-\ihratio*\yb+0.09}
  \pgfmathsetmacro{\yyb}{0.5*\hratio*\xa + 0.5*\ya}
  % end compute  grid points on edge 
  % \fill (\xxa,\yya) circle (1pt); % Plot grid points on edge 
  % \fill (\xxb,\yya) circle (1pt);  % Plot grid points on edge
\draw[invarrow=0.1] (\xa,\ya) -- (\xxa,\yya) ; % draw grid lines
  \draw[arrow=0.1] (\xb,\yb) -- (\xxb,\yyb) ; % draw grid lines
  % Plot trace lines
  \pgfmathsetmacro{\xxr}{0.5*\xxa-0.5*\xxb}
  \pgfmathsetmacro{\yyr}{-0.2*\xxr}
  \draw[dashed] (\xxb,\yyb) arc [start angle=180, end angle=0, x radius=\xxr cm, y radius=\yyr cm];
  }
%% end lattice
%% ------------------------
\pgfmathsetmacro{\xa}{\wt+\wt*cos(-80)}
\pgfmathsetmacro{\ya}{\hwt*sin(-80)}
\pgfmathsetmacro{\xb}{\wt+\wt*cos(80-180)}  % Compute x coordinate
  \pgfmathsetmacro{\yb}{\hwt*sin(80-180)}  % Compute y coordinate 
  \pgfmathsetmacro{\xxa}{0.5*\xa+\ihratio*\ya-0.09}
  \pgfmathsetmacro{\yya}{0.5*\hratio*\xa+0.5*\ya}
\pgfmathsetmacro{\xxb}{0.5*\xb+\wt-\ihratio*\yb+0.09}
  \pgfmathsetmacro{\yyb}{0.5*\hratio*\xa + 0.5*\ya}
  \pgfmathsetmacro{\xxr}{0.5*\xxa-0.5*\xxb}  
  \pgfmathsetmacro{\yyr}{-0.2*\xxr}
\draw[dashed, red, line width=0.3mm] (\xxb,\yyb) arc [start angle=180, end angle=0, x radius=\xxr cm, y radius=\yyr cm];  
%% start path 1
\pgfmathsetmacro{\xa}{\wt+\wt*cos(-50)}
\pgfmathsetmacro{\ya}{\hwt*sin(-50)}
\pgfmathsetmacro{\xxa}{0.5*\xa+\ihratio*\ya-0.09}
  \pgfmathsetmacro{\yya}{0.5*\hratio*\xa+0.5*\ya}
\draw[gcol, line width=0.5mm, rounded corners=0.4mm] (\xa-0.9,\ya+3) -- (\xa-1.57,\ya+2.13) --(\xxa+0.3,\yya-0.4)--(\xxa-0.05,\yya-0.87) --(\xxa-0.36,\yya-0.45);
\pgfmathsetmacro{\xa}{\wt+\wt*cos(-65)}
\pgfmathsetmacro{\ya}{\hwt*sin(-65)}
\pgfmathsetmacro{\xb}{\wt+\wt*cos(65-180)}  % Compute x coordinate
  \pgfmathsetmacro{\yb}{\hwt*sin(65-180)}  % Compute y coordinate
  \pgfmathsetmacro{\xxa}{0.5*\xa+\ihratio*\ya-0.09}
  \pgfmathsetmacro{\yya}{0.5*\hratio*\xa+0.5*\ya}
\pgfmathsetmacro{\xxb}{0.5*\xb+\wt-\ihratio*\yb+0.09}
  \pgfmathsetmacro{\yyb}{0.5*\hratio*\xa + 0.5*\ya}
  \pgfmathsetmacro{\xxr}{0.5*\xxa-0.5*\xxb}
  \pgfmathsetmacro{\yyr}{-0.2*\xxr}
\draw[dashed, gcol, line width=0.3mm] (\xxb,\yyb) arc [start angle=180, end angle=0, x radius=\xxr cm, y radius=\yyr cm];
%% path 2
\draw[red, line width=0.5mm, rounded corners=0.4mm] (1.57,2.1) -- (2.26,1.19) -- (3.38,2.66)-- (4.13,1.7) -- (4.4,2.1) ;
%% path 3
\draw[gcol, line width=0.5mm, rounded corners=0.4mm](0.42,-0.26) -- (2.27,2.19) -- (4.28,-0.47);
%% path 4
\draw[blue, line width=0.5mm, rounded corners=0.4mm] (1.07,-0.4) -- (2.26,1.16) -- (2.62,0.72) -- (3.35,1.7) -- (4.9,-0.4);
\end{tikzpicture}
% \end{center}
\end{align}
%%%%% ------------------ 
%%%%% ------------------ end cone
%%%%% ------------------
What may appear as an intersection of two green lines in the above picture should be interpreted as a touching of a green loop and a green boundary path. We can take the labeling set of the colours to be $\mathcal{I}=\{1,2,3,4\}$ with the identification $\{$no path, red, green, blue$\} =\{1,2,3,4\}$. The occupation of the boundary points by paths is recorded by $\alpha=(\alpha_1\ldots \alpha_4)$ and $\beta=(\beta_1\ldots \beta_4)$, so that we have  $\alpha=(1,3,4,1)$ and $\beta=(1,1,4,3)$. The parameters $z_1\ldots  z_4$ and $q z_1\ldots q z_4$ are called the {\it spectral parameters}, they are ``carried'' by the directed lines of the lattice which are directly above them. The winding of a lattice line around the tip of the cone changes the value of the corresponding spectral parameter by a factor of $q$. These parameters will enter the Boltzmann weights as discussed below.

By specifying $n,m$ and $N$ we can draw all possible path configurations. These configurations can be grouped according to specific occupations of the boundary points given by $\alpha$ and $\beta$. Each configuration $C$ carries a weight $W_C$ which is computed by multiplying the values of all local Boltzmann weights\footnote{These Boltzmann weights correspond to the coefficients of the $R$-matrix of the algebra $U_{t}(\widehat{\fgl}_{n|m})$ computed in \cite{BazS}.}: 
\begin{equation}\label{tikz:coloredvertices_intro}
\begin{tabular}{c@{\hskip 0.7cm}c@{\hskip 0.7cm}c@{\hskip 0.7cm}c@{\hskip 0.7cm}c}
% -R matrix vertex
\begin{tikzpicture}[scale=0.6,baseline=-2pt]
\draw[invarrow=0.75] (-1,0) --(1,0);
\draw[invarrow=0.75,] (0,-1) --(0,1);
\node[right] at (1,0) {$\scriptstyle x$};
\node[above] at (0,1) {$\scriptstyle y$};
\end{tikzpicture}:
% vertices
&
\begin{tikzpicture}[scale=0.6,baseline=-2pt]
\draw[red, ultra thick, rounded corners=0.7mm] (-1,0) node[left,black]{$\scriptstyle i$}  -- (0,0) -- (0,1) node[above,black]{$\scriptstyle i$};
\draw[red, ultra thick, rounded corners=0.7mm] (0,-1)node[below,black]{$\scriptstyle i$}  -- (0,0) --(1,0)node[right,black]{$\scriptstyle i$};
\end{tikzpicture}
&
\begin{tikzpicture}[scale=0.6,baseline=-2pt]
\draw[red, ultra thick, rounded corners=0.7mm] (-1,0) node[left,black]{$\scriptstyle i$}  -- (0,0)--(0,1) node[above,black]{$\scriptstyle i$};
\draw[gcol, ultra thick, rounded corners=0.7mm] (0,-1) node[below,black]{$\scriptstyle j$} --(0,0)-- (1,0)node[right,black]{$\scriptstyle j$};
\end{tikzpicture}
&
\begin{tikzpicture}[scale=0.6,baseline=-2pt]
\draw[gcol, ultra thick] (-1,0) node[left,black]{$\scriptstyle j$}  -- (1,0) node[right,black]{$\scriptstyle j$};
\draw[red, ultra thick] (0,-1)node[below,black]{$\scriptstyle i$}  --(0,1) node[above,black]{$\scriptstyle i$};
\end{tikzpicture}
% weights
\\[3em]
&$
 \begin{cases}
  \dfrac{t^{-1/2}-t^{1/2}x/y}{1-x/y}  & i\leq n\\
  \dfrac{t^{-1/2}x/y-t^{1/2}}{1-x/y}  & i> n
\end{cases}
$
&$
 \begin{cases}
  \dfrac{(t^{-1/2}-t^{1/2})x/y}{1-x/y}  & i<j\\
  \dfrac{(t^{-1/2}-t^{1/2})}{1-x/y}  & i>j
\end{cases}
$
&$1$
\end{tabular}
\end{equation}
where red ``$i$'' and green ``$j$'' can be replaced by any pair of distinct colours. We attach the label $1\in \mathcal{I}$ to the edges which have no path. The vertex with all labels equal to $i$ is interpreted as two paths of the same colour $i$ touching each other but not intersecting. The weight of this vertex depends on whether the label $i$ is bosonic or fermionic which is determined by the gradation. The parameters $x$ and $y$ are the spectral parameters. These parameters should be replaced by the appropriate $z_a$ and $q z_b$ carried by the lattice lines.

For a fixed $N$, $\mathcal{I}$ and $\alpha,\beta \in \mathcal I^N$, the collection of all global configurations on the $N$ by $N$ lattice \eqref{eq:cone} with such bosonic and fermionic paths is denoted by $\Omega_{\alpha,\beta}$, where  the two indices $\alpha$ and $\beta$  specify the boundary conditions at the base of the cone as explained under \eqref{eq:cone_paths}. In addition to the local weights \eqref{tikz:coloredvertices_intro} each configuration $C$ is multiplied by a factor which accounts for the closed loops content of the configuration, see \eqref{eq:Z}. The partition function of such lattice paths is defined by:
\begin{align}\label{eq:Z}
  Z_{\alpha,\beta} = \sum_{C\in \Omega_{\alpha,\beta}}  
  \prod_{i\in \mathcal{I}}   u_i^{m_i(C)}
    \times W_C
\end{align}
where $m_i(C)$ denotes the total number of loops of colour $i$ in $C$. Therefore, the new variables $u_i$ count bosonic loops for $i\leq n$ and fermionic loops for $i>n$.

Let $V\cong \mathbb C^{n+m}$ have a standard orthonormal basis with vectors labeled by $\mathcal{I}$. For the basis vectors of $V^{\otimes N}$ we use $\alpha\in \mathcal{I}^N$ as labels. Then $Z_{\alpha,\beta}$ with different $\alpha$ and $\beta$ can be combined to form a matrix $Z_N$:
\begin{align}
    \label{eq:Z_N-intro}
    Z_N = \sum_{\alpha,\beta\in \mathcal{I}^N} Z_{\alpha,\beta} \ket{\alpha}\bra{\beta}
\end{align}
Furthermore, the various $Z_N$ can be collected into a generating function:
\begin{align}
    \label{eq:Z(v)-intro}
    Z(v)= \sum_{N= 0}^\infty  Z_{N} v^N
\end{align}
One of the main results of this paper is an expression for $Z(v)$ in terms of a ``shuffle exponential''\footnote{In the main text we will write $\exp$ for $\exp_*$.}: 
\begin{align*}
\exp_*(A):= 1 + A + \frac{1}{2!}  A*A + \frac{1}{3!}  A*A*A +\cdots 
\end{align*}
where $*$ is the shuffle product of the matrix shuffle algebra associated to the quantum toroidal $\fgl_{n|m}$, see \eqref{eqn:shuffle product}.
\begin{thm}\label{thm:Z_intro}
For any $n,m \geq 0$, the generating function $Z(v)$ can be expressed as:
    \begin{align}
    \label{eq:Zexp}
    Z(v) = 
    \exp_* \left(\sum_{k=1}
    ^{\infty}\frac{v^k}{k}
    \sum_{i=1}^{n+m}  \left(y_{i+1}^k - t^{\epsilon_{i} k} y_{i}^k \right)S_k^{(i)}
    \right)
    \end{align}
where $y_i:=t^{-\epsilon_i/2}\epsilon_i u_i$, $u_{n+m+1}:=q u_1$, $\epsilon_i=(-1)^{\delta_{i>n}}$,  and $S_k^{(1)}\ldots S_k^{(n+m)} \in \emph{End}(V^{\otimes k})(z_1\ldots z_k)$ is a collection of matrix-valued rational functions which can be computed recursively: 
\begin{equation}
\label{eqn:power comm introduction}
(z_1+\dots+z_k) S_k^{(i)} = \left[S_{k-1}^{(i)},z_1 S_{1}^{(i)} \right],
\qquad
S_{1}^{(i)} = \frac{-1}{1-q}\normalfont{\text{diag}}\{\overset{i}{\overbrace{1\ldots 1}},q\ldots q \}
\end{equation}
where $[a,b]=a*b-b*a$. Theorem \ref{thm:Z_intro} is contingent on Conjecture \ref{conj:super}, which was proved for $m=0$ (\!\!\cite{Ng_tale}).
\end{thm}
The matrix elements of $S_k^{(i)}$ themselves can also be written as conic partition functions of the form \eqref{eq:Z} where the summation runs over all configurations with boundary paths coloured by $\{1\ldots n+m+2\}\setminus\{i+1,i+2\}$ and where the loops are allowed to be only of colours $i+1$ (bosonic) and $i+2$ (fermionic) and the $u$-factors of \eqref{eq:Z} must be replaced with $m_{i+2}(C)(-1)^{m_{i+2}(C)}$ (see \eqref{eq:S-trace intro}). 

For all $i,j,k,l$ we have the equality
$$
S_k^{(i)}*S_l^{(j)}=S_l^{(j)}*S_k^{(i)}
$$
Thus formula \eqref{eq:Zexp} is unambiguous, and moreover the matrices $S_k^{(i)}$ can be viewed as elements (in fact generators) of a commuting matrix shuffle algebra $\mathcal{B}^+$ that we recall in Subsection \ref{sub:slope 0}. The matrices $Z_N$ therefore also represent commuting elements of $\mathcal{B}^+$. 

Theorem \ref{thm:Z_intro} is a generalization of the shuffle exponential formula of \cite{GG} for the $\fgl_1$ shuffle algebra. Indeed, in \cite{GG} the definition of the partition function prevents having boundary paths. Choosing the boundary conditions at the base of the cylinder \eqref{eq:cone} to be free of paths corresponds to taking the vacuum-vacuum expectation value of $Z_N$:
\begin{align}
    \bra{1^N} Z_N \ket{1^N}
    \label{eq:Z-gl1}
\end{align}
By computing this expectation value of $Z(v)$ in \eqref{eq:Zexp} one can reproduce the exponential formula given in Theorem 1.1 of \cite{GG}.

Let us make a remark on a relation of such shuffle exponential formulas and integrable models. In the work \cite{GG} the authors showed that their shuffle exponential generating function of \eqref{eq:Z-gl1} with $n=0$ is equal to the (mixed) Cauchy kernel for the Macdonald functions \cite{FHSSY}. In the context of the Bethe ansatz of \cite{FJMM_BA} this mixed Cauchy kernel is interpreted as the off-shell Bethe wave function of the ``zero twist'' transfer matrix\footnote{The same exponential formula of \cite{GG} at $m=n$ actually produces the full off-shell Bethe wave function of \cite{FJMM_BA} with the general twist parameter $\tilde p$. To see this one must take the formula (6.8) in \cite{GG}, set in it $z_0=0$ and $z_i=\tilde p w_i$ and then compare the resulting expression with the off-shell Bethe vector from Theorem 5.5 in \cite{FJMM_BA}.}. Let us consider the $n=0$ case of \eqref{eq:Zexp} expressed in terms of generators $P_k^{(i)}\in \mathcal{B}^+$ which where introduced in \cite{Ng_tale}. These generators have a simple relation with $S_k^{(i)}$ (see Proposition \ref{prop:comm}):
\begin{align}
    P_k^{(i)} = q^{k\delta_{i=1}} S_k^{(i-1)}-S_k^{(i)}
\end{align}
where $S_k^{(0)}:=S_k^{(n+m)}$.
\begin{rmk}
    The expression in \eqref{eq:Zexp} at $n=0$ (all $\epsilon_i=-1$), $v=t$ and written using the generators $P_k^{(i)}$ reads:
    \begin{align}
    \label{eq:Zexp-nsm}
    Z(t) = 
    \exp_* \left(\sum_{k=1}
    ^{\infty}\frac{1}{k}\sum_{i,j=1}^{m} \frac{1-q^{\delta_{i=j}k}t^{k}}{1-q^k}q^{k\delta_{i>j}} P_k^{(i)}y_j^k\right)
    \end{align}
    We compare this formula with the reproducing kernel of the non-symmetric Macdonald polynomials from \cite{Mimachi-Noumi}
    \begin{align}\label{eq:NSM-kernel}
        E(x,y;q,t) &= \prod_{1\leq j<i\leq m}
    \frac{(q t x_i y_j;q)_\infty}{(q x_i y_j;q)_{\infty}}
    \prod_{1\leq i\leq m}
    \frac{(q t x_i y_i;q)_\infty}{(x_i y_i;q)_{\infty}}
    \prod_{1\leq i<j\leq m}
    \frac{(t x_i y_j;q)_\infty}{(x_i y_j;q)_{\infty}}\nonumber \\
    &= \exp \left(\sum_{k=1}
    ^{\infty}\frac{1}{k}\sum_{i,j=1}^{m} \frac{1-q^{\delta_{i=j}k}t^{k}}{1-q^k}q^{k\delta_{i>j}} x_i^k y_j^k \right) 
    \end{align}
    and find that the generator in \eqref{eq:Zexp-nsm}, of the fermionic partition functions \eqref{eq:cone_paths}, matches with the reproducing kernel of the non-symmetric Macdonald polynomials \eqref{eq:NSM-kernel} after the identification of $x^k_i$ with $P_k^{(i)}$.
\end{rmk}
Further investigation of connections between matrix shuffle algebras and Macdonald polynomials and their generalizations is beyond the scope of this paper.

\subsection{An anti-isomorphism of shuffle algebras}
Let us explain the underlying algebraic reasons for the formula \eqref{eq:Zexp}. We draw the directed lattice in \eqref{eq:cone} in a planar form and identify every crossing of the lattice lines with a matrix $\check R_{a,b}(z_a/(q z_b))\in\text{End}(V^{\otimes 2N})$, defined in Section \ref{sub:super}\footnote{We find it convenient to work with $\check R= R P$, with $P$ being the permutation matrix in $V\otimes V$.}. Then the partition function $Z_{\alpha,\beta}$ is given by\footnote{In order to recover \eqref{eq:cone} one needs to rotate the picture \eqref{eq:Z-trace} by $90$ degrees counter-clockwise and lay it down on the cone such that the point marked with the dot in \eqref{eq:Z-trace} matches with the tip of the cone. 
}:
\begin{align}\label{eq:Z-trace}
Z_{\alpha,\beta}
=
\begin{tikzpicture}[scale=0.5,baseline=(current  bounding  box.center)]
% crosses
\rcheck{3}{4}
\rcheck{2}{3}
\rcheck{4}{3}
\rcheck{3}{2}
\rcheck{1}{2}
\rcheck{0}{1}
\rcheck{2}{1}
\rcheck{1}{0}
\rcheck{5}{2}
\rcheck{4}{1}
\rcheck{6}{1}
\rcheck{5}{0}
\rcheck{3}{0}
\rcheck{2}{-1}
\rcheck{4}{-1}
\rcheck{3}{-2}
% vertical lines
\draw[arrow=0.166] (0,5) -- (0,2);
\draw[arrow=0.25] (1,5) -- (1,3);
\draw[arrow=0.5] (2,5) -- (2,4);
\draw[arrow=0.3] (3,5) -- (4,4);
\draw[arrow=0.834] (0,1) -- (0,-2);
\draw[arrow=0.75] (1,0) -- (1,-2);
\draw[arrow=0.5] (2,-1) -- (2,-2);
\draw[invarrow=0.3] (3,-2) -- (4,-1);
% trace
\draw [rounded corners=5pt] (7,2) -- (7.5,2) -- (7.5,1) -- (7,1);
\draw [rounded corners=5pt] (6,3) -- (8,3) -- (8,0) -- (6,0);
\draw [rounded corners=5pt] (5,4) -- (8.5,4) -- (8.5,-1) -- (5,-1);
\draw [rounded corners=5pt] (4,5) -- (9,5) -- (9,-2) -- (4,-2);
% indices
\node[below] at (0,-2) {$\scriptstyle{\alpha_1}$};
\node[below] at (1,-2) {$\scriptstyle{\cdots}$};
\node[below] at (2,-2) {$\scriptstyle{\cdots}$};
\node[below] at (3,-2) {$\scriptstyle{\alpha_N}$};
\node[above] at (0,5) {$\scriptstyle{\beta_1}$};
\node[above] at (1,5) {$\scriptstyle{\cdots}$};
\node[above] at (2,5) {$\scriptstyle{\cdots}$};
\node[above] at (3,5) {$\scriptstyle{\beta_N}$};
% paramters
\node[below] at (0,-2-0.6) {$\scriptstyle{z_1}$};
\node[below] at (1,-2-0.6) {$\scriptstyle{\cdots}$};
\node[below] at (2,-2-0.6) {$\scriptstyle{\cdots}$};
\node[below] at (3,-2-0.6) {$\scriptstyle{z_N}$};
\node[above] at (0,5+0.6) {$\scriptstyle{q z_1}$};
\node[above] at (1,5+0.6) {$\scriptstyle{\cdots}$};
\node[above] at (2,5+0.6) {$\scriptstyle{\cdots}$};
\node[above] at (3,5+0.6) {$\scriptstyle{q z_N}$};
\node at (7,1.5) {$\scriptstyle{.}$};
\end{tikzpicture}
\end{align}
where the lines joining the bottom edges with the top edges of the right half of the $N$ by $N$ lattice are interpreted as the twisted trace in the vector spaces positioned at $N+1\ldots 2N$ within the tensor product $V^{\otimes 2N}$. The twist is given by the matrices $\hat u$:
\begin{align}\label{eq:u-hat introduction}
    \hat u := \sum_{i=1}^{n+m}u_i E_{ii}
\end{align}
which are implied in \eqref{eq:Z-trace}. The graphical representation of $Z_N$ is the same as \eqref{eq:Z-trace} but with the indices $\alpha$ and $\beta$ removed. 
Algebraically $Z_N$ reads:
\begin{align}\label{eq:Z-R introduction}
   Z_N = \Tr_{N+1\ldots 2N} 
   \left[
   \prod_{j=1}^N\prod_{i=1}^N \check{R}_{N+j-i}\left(\frac{z_{N-i+1}}{q z_j}\right)\cdot
\text{id}^{\otimes N}\otimes \hat u^{\otimes N}
\right]
\end{align}
Consider any $X\in \text{End}(V^{\otimes N}) (z_1 \ldots z_N)$. In Section \ref{sec:Psi} we show that the assignment:
\begin{align}\label{eq:Psi-intro}
\Psi\left [ X\right]:=  
 \Tr_{N+1\ldots 2N} 
   \left[
   \prod_{j=1}^N\prod_{i=1}^N \check{R}_{N+j-i}\left(\frac{z_{N-i+1}}{q z_j}\right)\cdot
\text{id}^{\otimes N}\otimes  X
\right]
\end{align}
realizes an anti-isomorphism of the matrix shuffle algebras $\CA' \xrightarrow{\sim} \CA^+$ given  by $X \mapsto \Psi[X]$. Here $\CA^+$ is a matrix shuffle algebra defined in Section \ref{sub:symmetric tensors} and $\CA'$ is a matrix shuffle algebra which is related to the dual of $\CA^+$ by a simple transformation (see Section \ref{sec:A-prime}).  The diagonal matrix $\hat u^{\otimes N}$ in \eqref{eq:Z-R introduction} is a generator of tensors composed of diagonal matrix units $E_{ii}$:
\begin{align}\label{eq:u-sum introduction}
    \hat u^{\otimes N} = \sum_{\lambda\in \mathcal{I}^N}  \bigotimes_{i\in \lambda} u_i E_{ii}
\end{align}
(see Lemma \ref{lem:H shuffle product}). Let $\mu=(\mu_1\ldots \mu_{n+m})$ be a composition of non-negative integers s.t. $\mu_1+\cdots +\mu_{n+m}=N$. We define elements $X_\mu$ as follows:
\begin{align}\label{eq:X-lambda}
X_\mu=\text{coefficient of }u_{\mu_1}\cdots u_{\mu_{n+m}} \text{ in }\hat u^{\otimes N}
\end{align}
In the case of the shuffle algebra $\mathcal{A}'$ of quantum toroidal $\fgl_n$ it is easy to show that such $X_\mu$ are commuting $X_\mu\in \mathcal B'$. In Section \ref{sec:Commutative} we explain that this should also hold in the case of quantum toroidal $\fgl_{n|m}$. 
Since $X_\mu\in\mathcal{B}' \subset \mathcal{A}'$, then by the anti-isomorphism property of $\Psi$ we will have that $\Psi[X_\mu]\in \mathcal{B}^+$. Therefore the generating function $Z(v)$ in \eqref{eq:Zexp} represents a particular example of the application of the anti-isomorphism $\Psi$ to the elements \eqref{eq:X-lambda} of the commuting subalgebra $\mathcal{B}'$ of $\mathcal{A}'$.  
 
\subsection{Formulas for elements of commuting matrix shuffle algebra of the quantum toroidal $\fgl_{n|m}$}
The generating function \eqref{eq:Zexp} gives rise to many formulas for various families of elements of the commuting shuffle algebra $\mathcal{B}^+$. All of these formulas have the form of a partial twisted trace of products of $\check R$-matrices similar to \eqref{eq:Z-R introduction}. The example discussed above gives such trace formulas for a family of commuting elements labelled by non-negative integer compositions $\mu=(\mu_1\ldots \mu_{n+m})$, namely 
\[
\Psi\left[\sum_{\substack{\lambda\in \mathcal I^N\\ m(\lambda)=\mu}}\bigotimes_{i=1}^N E_{\lambda_i \lambda_i}\right]
\]
where $m(\lambda)$ is the vector of multiplicities $m_i(\lambda)$ of $i$ in $\lambda$. In this example the tensors $\Psi[X]$, $X$ and the product of $\check R$'s are all endomorphisms of a tensor product of $N$ copies of the same vector space $V$. This is a special case of a more general map:
\begin{align}\label{eq:Psi-tilde intro}
       \widetilde\Psi[X]= \Tr_{N+1\ldots 2N}\left[\text{projection}\left(\prod_{j=1}^N\prod_{i=1}^N \check{R}''_{N+j-i}\left(\frac{z_{N-i+1}}{q z_j}\right)\right) \cdot \id^{\otimes N}\otimes X\right]\in \text{End}(V^{\otimes N})
\end{align}
where $X\in \text{End}(V'^{\otimes N})$ with $V'\cong \mathbb C^{n'+m'}$ for some $n',m'\geq 0$. The $R$-matrix  $\check R''$ is associated with $U_t(\dot{\fgl}_{n''|m''})$ with $n''\geq n$ and $m''\geq m$ and acts on $V''\otimes V''$, with $V'' \cong \mathbb C^{n''+m''}$ and we identify $V$ and $V'$ as coordinate subspaces of $V''$. The projection in \eqref{eq:Psi-tilde intro} is such that:
\begin{align}
       \text{projection}\left(\prod_{j=1}^N\prod_{i=1}^N \check{R}''_{N+j-i}\left(\frac{z_{N-i+1}}{q z_j}\right)\right): \text{End}(V^{\otimes N})\otimes \text{End}(V'^{\otimes N}) \to 
       \text{End}(V^{\otimes N})\otimes \text{End}(V'^{\otimes N})
\end{align}
The trace in \eqref{eq:Psi-tilde intro} is taken over the $N$ tensor factors of the vector spaces $V'$. This set up allows one to compute a variety of elements of the commuting subalgebra $\CB^+$ that are exponentially generated by $S_k^{(i)}$. In Section \ref{sec:Psi-tilde computation} we discuss such examples and in addition we demonstrate how to use $\widetilde \Psi$ to compute the elements $S_N^{(i)}$. The result of this computation is the following formula for the matrix elements $S_{\alpha,\beta}^{(i)}$ of $S_N^{(i)}$:
\begin{align}\label{eq:S-trace intro}
S_{\alpha,\beta}^{(i)}
=
\sum_{\gamma\in \{i+1,i+2\}^N}
\frac{(-1)^{m_{i+2}(\gamma)}m_{i+2}(\gamma)}{t^{N/2}-t^{-N/2}}
\begin{tikzpicture}[scale=0.5,baseline=(current  bounding  box.center)]
% crosses
\rcheck{3}{4}
\rcheck{2}{3}
\rcheck{4}{3}
\rcheck{3}{2}
\rcheck{1}{2}
\rcheck{0}{1}
\rcheck{2}{1}
\rcheck{1}{0}
\rcheck{5}{2}
\rcheck{4}{1}
\rcheck{6}{1}
\rcheck{5}{0}
\rcheck{3}{0}
\rcheck{2}{-1}
\rcheck{4}{-1}
\rcheck{3}{-2}
% vertical lines
\draw[arrow=0.166] (0,5) -- (0,2);
\draw[arrow=0.25] (1,5) -- (1,3);
\draw[arrow=0.5] (2,5) -- (2,4);
\draw[arrow=0.3] (3,5) -- (4,4);
\draw[arrow=0.834] (0,1) -- (0,-2);
\draw[arrow=0.75] (1,0) -- (1,-2);
\draw[arrow=0.5] (2,-1) -- (2,-2);
\draw[invarrow=0.3] (3,-2) -- (4,-1);
% indices
\node[below] at (0,-2) {$\scriptstyle{\iota(\alpha_1)}$};
\node[below] at (1.1,-2) {$\scriptstyle{\cdots}$};
\node[below] at (1.9,-2) {$\scriptstyle{\cdots}$};
\node[below] at (3,-2) {$\scriptstyle{\iota(\alpha_N)}$};
\node[above] at (0,5) {$\scriptstyle{\iota(\beta_1)}$};
\node[above] at (1,5) {$\scriptstyle{\cdots}$};
\node[above] at (1.9,5) {$\scriptstyle{\cdots}$};
\node[above] at (3,5) {$\scriptstyle{\iota(\beta_N)}$};
\node[below] at (4.4,-2) {$\scriptstyle{\gamma_1}$};
\node[below] at (5.2,-1) {$\scriptstyle{\cdots}$};
\node[below] at (6.2,0) {$\scriptstyle{\cdots}$};
\node[below] at (7.2,1) {$\scriptstyle{\gamma_N}$};
\node[above] at (4.4,5) {$\scriptstyle{\gamma_1}$};
\node[above] at (5.2,4) {$\scriptstyle{\cdots}$};
\node[above] at (6.2,3) {$\scriptstyle{\cdots}$};
\node[above] at (7.2,2) {$\scriptstyle{\gamma_N}$};
% paramters
\node[below] at (0,-2-1) {$\scriptstyle{z_1}$};
\node[below] at (1,-2-1) {$\scriptstyle{\cdots}$};
\node[below] at (2,-2-1) {$\scriptstyle{\cdots}$};
\node[below] at (3,-2-1) {$\scriptstyle{z_N}$};
\node[above] at (0,5+1) {$\scriptstyle{q z_1}$};
\node[above] at (1,5+1) {$\scriptstyle{\cdots}$};
\node[above] at (2,5+1) {$\scriptstyle{\cdots}$};
\node[above] at (3,5+1) {$\scriptstyle{q z_N}$};
\end{tikzpicture}
\end{align}
where the crosses are interpreted as the $\check R$-matrix of $U_t(\dot{\fgl}_{n+1|m+1})$, the space $V''$ has the labeling set $\mathcal{I}''=\{1\ldots n+m+2\}$. The gradation on $V''$ is given by $\epsilon''_j=\epsilon_j$ for $j\leq i$, $\epsilon''_{i+1}=1$, $\epsilon''_{i+2}=-1$ and $\epsilon''_{j}=\epsilon_{j-2}$ for $j>i+2$. The map $\iota: \mathcal{I}\to \mathcal{I}''$ is given by $\iota(j)=j$ for $j\leq i$ and $\iota(j)=j+2$ for $j>i$. Note that the summation over $\gamma$ in \eqref{eq:S-trace intro} can be interpreted as the trace  over $N$ two-dimensional subspaces of $V''$. In the language of lattice paths of \eqref{eq:cone_paths} the sum in \eqref{eq:S-trace intro} corresponds the partition function whose boundary paths are labelled by $\{1\ldots n+m+2\}\setminus \{i+1,i+2\}$ and the loops are labelled by $\{i+1,i+2\}$ with $i+1$ being bosonic and $i+2$ fermionic.

\subsection{Overview of the paper} In Section \ref{sec:shuffle}, we review the matrix shuffle algebra of type $\fgl_n$ and present its conjectural generalization to type $\fgl_{n|m}$. In Section \ref{sec:Psi}, we describe the anti-isomorphism $\Psi$ (and its relatives such as $\widetilde{\Psi}$) together with its graphical interpretation, which serves as the algebraic bedrock of our main results. In Section \ref{sec:Commutative}, we describe the commutative subalgebra $\CB^+$ of the matrix shuffle algebra $\CA^+$; inside $\CB^+$ we identify the elements $S_k^{(i)}$, which allows us to prove Theorem \ref{thm:Z_intro}.

%=====================================================================
%%%%%%%%%%%%
%%%%%%%%%%%%
%%%%%%%%%%%%
%%%%%%%%%%%%
%%%%%%%%%%%%    End of introduction
%%%%%%%%%%%%
%%%%%%%%%%%% 

\section{The matrix shuffle algebra of quantum toroidal $\fgl_{n|m}$}
\label{sec:shuffle}

In this section we recall the matrix shuffle algebra and its relation with the quantum toroidal $\fgl_n$ algebra \cite{Ng_tale}. We will then explain the generalization to the case of $\fgl_{n|m}$, of which many steps are conjectural (however, we outline the main challenges that need to be solved). After that we will introduce a diagrammatic representation of the elements of the shuffle algebra and the shuffle product which will be used in subsequent sections.

\subsection{The $R$-matrix}\label{sec:R-matrix}
Let $V = \BC^n$ and let $E_{ij}$ denote the $n \times n$ matrix with a single 1 at the intersection of row $i$ and column $j$, and zeroes everywhere else. 
The $R$-matrix $R(z/w) \in \text{End}(V_w \otimes V_z)$ of $\uu$ is given by:
\begin{align}\label{eq:R-matrix}
  R(z) := \sum_{1\leq i,j\leq n}
  \left(\frac{t^{-1/2}-z t^{1/2}}{1-z}\right)^{\delta_{i=j}}E_{ii}\otimes E_{jj}  
  +
  \sum_{1\leq i\neq j\leq n}
  \frac{(t^{-1/2}-t^{1/2})z^{\delta_{i<j}}}{1-z}E_{ij}\otimes E_{ji}  
\end{align}
Let $P \in \text{End}(V \otimes V)$ be the permutation matrix. The matrix $\check R$ is defined by:
\begin{align}\label{eq:R and R-check}
    \check{R}(z) := R(z) P
\end{align}
We define the matrix $\check R_i(z_{i+1}/z_i)\in \text{End}(V_{z_1}\otimes \cdots \otimes V_{z_N})$ by the action of $\check R(z_{i+1}/z_i)$ on $V_{z_i}\otimes V_{z_{i+1}}$ and by the identity on all other tensor factors. Similarly $P_i\in \text{End}(V_{z_1}\otimes \cdots \otimes V_{z_N})$ denotes the action by $P$ on $V_{z_i}\otimes V_{z_{i+1}}$ and by the identity on the other factors. For $1\leq i<j\leq N$ we define $\check R_{i,j}(z)$ as $\check R$ acting in the $i$-th and $j$-th tensor factors. Another important matrix is $\check R^\bullet$ defined by:
\begin{align}
    \label{eq:R-bullet}
    \check{R}^\bullet(z):=(D\otimes \id) \check{R} (z t^{-n}) (\id\otimes D^{-1})
\end{align}
where:
\begin{align}
    \label{eq:D}
    D := \normalfont{\text{diag}}\{1,t^{-1}\ldots t^{-n+1}\}
\end{align}
The matrix $\check R^\bullet_{i,j}\in \text{End}(V_{z_1}\otimes \cdots \otimes V_{z_N})$ is defined in the same way as $\check R_{i,j}$.

The matrix $\check R$ satisfies the Yang--Baxter equation, the unitarity relation and has a simple pole at $z=1$:
\begin{align}
    \label{eq:YB}
    \check R_{i}(z/y)
    \check R_{i+1}(z/x)
    \check R_{i}(y/x)
    &=
    \check R_{i+1}(y/x)
    \check R_{i}(z/x)
    \check R_{i+1}(z/y)
\\
\label{eq:unitarity}
    \check R(x/y)
    \check R(y/x)
    &=f(x/y) \text{id}\otimes \text{id},\\
\label{eq:R residue}
    \underset{x=y}{\text{Res}}\,\check R(x/y) &= (t^{1/2}-t^{-1/2})\id \otimes \id     
\end{align}
where ``Res $f(z)$'' denotes the residue of the 1-form $f(z)\frac {dz}z$, and
\begin{align}
    \label{eq:f}
    f(x)=\frac{(1-t x)(1-t^{-1}x)}{(1-x)^2}
\end{align}
Additionally, the matrices $\check R$ and $\check R^\bullet$ satisfy the {\it crossing unitarity relation}:
\begin{align}
    \label{eq:crossing_unitarity}
\left(\check{R}^\bullet(x) P
    \right)^{T_2}
    \left(P \check{R}(x^{-1})
    \right)^{T_2} =\id\otimes \id
\end{align}
where $T_2$ denotes the transposition in the second tensor factor.

\subsection{Symmetric tensors and wheel conditions}
\label{sub:symmetric tensors}

Consider the vector space:
\begin{equation}
\label{eqn:big space}
\CV = \bigoplus_{k=0}^{\infty} \text{End}_{\BC(q,t)}(V^{\otimes k}) (z_1 \ldots z_k)
\end{equation}
For a fixed  $k\geq 0$ we will sometimes write $(z)$ instead of $(z_1\ldots z_k)$, for simplicity. Similarly, if $\sigma$ is a permutation in the symmetric group $\mathfrak{S}_k$ then we will write $(z_\sigma)=(z_{\sigma(1)} \ldots z_{\sigma(k)})$. Fix $k\geq 0$, a tensor $X=X_{1\dots k}(z_1 \ldots z_k)$ will be called symmetric if for all $\sigma \in \mathfrak{S}_k$ we have:
\begin{equation}
\label{eqn:symmetric tensor}
\check R_\sigma(z) 
X_{1\ldots k}(z) = 
X_{1\ldots k}(z_{\sigma})
\check R_\sigma(z)
\end{equation}
where the matrix $\check R_{\sigma}(z) = \check R_{\sigma} (z_1\ldots z_k)$ is defined recursively by
\begin{align}
\label{eq:Rch-sigma}
&\check R_{\text{id}}(z) = 1,\qquad    
\check R_{\sigma}(z) = \check R_{i}(z_{\sigma'(i+1)}/z_{\sigma'(i)})
   \check R_{\sigma'}(z)
\end{align}
whenever $\sigma=s_i\sigma'$ is a reduced decomposition in $\mathfrak S_k$. We similarly define $\check R^\bullet_{\sigma}$ and $R_\sigma$; in particular, in terms of the latter matrices, formula \eqref{eqn:symmetric tensor} takes the form
$$
R_\sigma (z)
X_{1...k}(z) = 
X_{\sigma(1)\dots\sigma(k)}(z_{\sigma})
R_\sigma(z)
$$
Above and henceforth, we often write $X_{1\dots k}$ to indicate the fact that $X$ is a linear operator on $V^{\otimes k}$, and use the notation $X_{\sigma(1)\dots\sigma(k)}$ to denote the same linear operator conjugated by the permutation of the factors of $V^{\otimes k}$ corresponding to $\sigma \in \mathfrak S_k$. The subspace of symmetric tensors will be denoted by:
\begin{equation}
\label{eqn:symmetric space}
\CV_{\text{sym}} \subset \CV
\end{equation}
Furthermore, we say that a symmetric tensor satisfies the pole conditions if it is of the form:
\begin{equation}
\label{eqn:pole condition}
X_{1\dots k}(z_1\ldots z_k) = \frac {x_{1\dots k}(z_1\ldots z_k)}{\prod_{1\leq i \neq j \leq n} (z_i-z_jq)}
\end{equation}
for some Laurent polynomial:
$$
x_{1\dots k}(z_1 \dots z_k) \in \text{End}_{\BC(q,t)}(V^{\otimes k}) [z_1^{\pm 1} \dots z_k^{\pm 1}]
$$
The subspace of symmetric tensors which satisfy the pole conditions will be denoted by:
\begin{equation}
\label{eqn:pole space}
\CA^+_{\text{big}} \subset \CV_{\text{sym}}
\end{equation}

\begin{defn}
\label{def:wheel}
Consider any composition $\lambda$ of length $u$ and size $\lambda_1+\dots+\lambda_{u}=k$, and let $c_i = \lambda_1+\dots+\lambda_{i-1}+1$ for all $i \in \{1,\dots,u+1\}$. A tensor $X \in \CA^+_{\emph{big}} \cap \normalfont{\text{End}}_{\BC(q,t)}(V^{\otimes k}) (z_1 \ldots z_k)$ is said to satisfy the wheel conditions if, for every composition $\lambda$ as above, the iterated residue of $X$ at:
\begin{equation}
\label{eqn:residue variables}
\Big\{z_{\lambda_1+\dots+\lambda_{s-1}+1} = y_s, z_{\lambda_1+\dots+\lambda_{s-1}+2} = y_s q  \ldots  z_{\lambda_1+\dots+\lambda_{s-1}+\lambda_s} = y_s q^{\lambda_s-1} \Big\}_{1\leq s \leq u}
\end{equation}
is of the form:
\begin{equation}
\label{eqn:residue}
(t^{\frac 12} - t^{-\frac 12})^{k-u} \prod^{\text{unordered pairs } (s,d) \neq (t,e)}_{\text{with } 1 \leq s,t \leq u, 1 \leq d < \lambda_s, 1 \leq e < \lambda_t} f \left( \frac {y_s q^{d}}{y_t q^{e}} \right) \left[ \prod_{u \geq s \geq 1} \prod_{s \leq t \leq u} \prod_{1 \leq e < \lambda_t} R_{c_s,c_t+e} \left( \frac {y_s}{y_t q^{e}}\right) \right]  
\end{equation}
$$
 X^{(\lambda_1\ldots \lambda_u)}_{c_1 \ldots c_u}(y_{1}\ldots y_{u}) \left[ \prod_{1 \leq s \leq u} \prod_{u \geq t > s} \prod_{\lambda_t > e \geq 1} R_{c_t + e, c_s} \left( \frac {y_t q^{e}}{y_s q^{\lambda_s}} \right) \right] \left[ \prod_{s=1}^u P_{c_s} P_{c_s+1} \dots P_{c_{s+1}-2} \right]
$$
for some tensor $X^{(\lambda_1\ldots \lambda_u)} \in \emph{End}(V^{\otimes u})(y_1\ldots y_u)$. We write:
\begin{equation}
\label{eqn:wheel space}
\CA^+ \subset \CA^+_{\emph{big}}
\end{equation}
for the vector subspace of symmetric tensors which satisfy the pole and wheel conditions. 
\end{defn}

We remark that the pole conditions (which require that tensors $X$ have at most simple poles at $z_i = z_j q$) were precisely introduced in order for the residue at \eqref{eqn:residue variables} to be well-behaved. As shown in \cite[Proposition 4.11]{Ng_tale}, for any composition $\lambda=(\lambda_1\ldots \lambda_u)$ with $k = \lambda_1+\dots+\lambda_u$, the tensor $X^{\lambda}(y_{1}\ldots y_{u})$ that features in \eqref{eqn:residue} also satisfies the pole conditions in the variables $y_s$, and is symmetric in the sense of \eqref{eqn:symmetric tensor} with respect to all pairs of variables $y_s$ and $y_t$ for which $\lambda_s = \lambda_t$.

\begin{ex}\label{ex:wheel u=1}
    We consider the particular example of Definition \ref{def:wheel} with $u=1$ and $k=3$ 
    and compute the wheel condition for a tensor $X(z_1,z_2,z_3)\in \CA_{\text{big}}^+ \cap \normalfont{\text{End}}_{\BC(q,t)}(V^{\otimes 3}) (z_1,z_2,z_3)$. 
    We have $\lambda=(\lambda_1)=(3)$ and $(c_1,c_2)=(1,4)$, so the specialization in \eqref{eqn:residue variables} is given by:
    \begin{equation}
    \label{eqn:residue variables example first line}
    \left\{z_1 = y, z_2 =  q y,  z_3 = q^2 y\right\}
    \end{equation}
    where we set $y_1=y$. The first line in \eqref{eqn:residue} is given by
    \begin{equation*}
(t^{\frac 12} - t^{-\frac 12})^{2} \prod_{1 \leq e < d< 3} f \left(q^{d-e} \right) \left[ \prod_{1 \leq e < 3} R_{1,1+e} \left(q^{-e}\right) \right]  =
(t^{\frac 12} - t^{-\frac 12})^{2} 
f \left(q \right)
R_{1,2} \left(q^{-1}\right)
R_{1,3} \left(q^{-2}\right)
\end{equation*}
and the second line in \eqref{eqn:residue} is given by
\begin{equation*}
     X^{(3)}_{1}(y) \left[ \prod_{s=1}^u P_{c_s} P_{c_s+1} \dots P_{c_{s+1}-2} \right]  = X^{(3)}_{1}(y) P_1 P_2
\end{equation*}
Combining these two expressions, we conclude that the wheel condition with $\lambda=(3)$ for a tensor $X(z_1,z_2,z_3)$ requires that its residue be of the form:
\begin{equation}\label{eqn:residue variables example}
\underset{\left\{z_1 = y,z_2 =q y,z_{3} = q^{2} y\right\}}{\normalfont{\text{Res}}}X(z_1,z_2, z_3)=
(t^{\frac 12} - t^{-\frac 12})^{2} 
f \left(q \right)
R_{1,2} \left(q^{-1}\right)
R_{1,3} \left(q^{-2}\right)
 X^{(3)}_{1}(y) P_1 P_2
\end{equation}
for some $X^{(3)}(y) \in \emph{End}(V)(y)$. By commuting the permutation matrices $P_i$ to the left and combining them with the $R$-matrices we can rewrite this using the $\check R$ matrices \eqref{eq:R and R-check}:
\begin{equation}\label{eqn:residue variables example R-check}
\underset{\left\{z_1 = y,z_2 =q y,z_{3} = q^{2} y\right\}}{\normalfont{\text{Res}}}X(z_1,z_2, z_3)
=(t^{\frac 12} - t^{-\frac 12})^{2}
f \left(q \right)
\check R_1 \left(q^{-1}\right)
\check R_2 \left(q^{-2}\right)
 X^{(3)}_{3}(y)
\end{equation}

In a more general case $\lambda=(k)$ we have:
\begin{align}\label{eqn:residue (k)}
\underset{\left\{z_1 = y \ldots  z_{k} = q^{k-1} y\right\}}{\normalfont{\text{Res}}}X(z_1\ldots z_k)= (t^{\frac 12} - t^{-\frac 12})^{k-1} \prod_{j=1}^{k-2} f \left(q^j \right)^{k-1-j}
 \check R_{1} \left(q^{-1}\right) \cdots \check R_{k-1} \left(q^{-k+1}\right) X_k^{(k)}(y)
\end{align}
\end{ex}
\subsection{The shuffle product}
\label{sub:shuffle product}

We can make the vector space $\CV$ of \eqref{eqn:big space} into an associative algebra via the following formula for all $A = A_{1\ldots k}(z_1\ldots z_k)$, $B = B_{1\ldots l}(z_1\ldots z_l)$ in $\CV$:
\begin{equation}
\label{eqn:shuffle product}
A * B = \sum^{a_1<\ldots <a_k, \ b_1<\ldots <b_l}_{\{1\ldots k+l\} = \{a_1\ldots a_k\} \sqcup \{b_1\ldots b_l\}} \left[ \prod_{i=k}^{1} \prod_{j=1}^l \underbrace{R_{a_ib_j} \left( \frac {z_{a_i}}{z_{b_j}} \right)}_{\text{only if }a_i < b_j} \right] 
\end{equation}
$$
A_{a_1\ldots a_k}(z_{a_1}\ldots z_{a_k}) \left[ \prod_{i=1}^k \prod_{j=l}^{1} R_{b_ja_i}  \left( \frac {z_{b_j}}{z_{a_i} q} \right) \right] B_{b_1\ldots b_l}(z_{b_1}\ldots z_{b_l})  \left[  \prod_{i=k}^{1} \prod_{j=1}^l \underbrace{R_{a_ib_j} \left( \frac {z_{a_i}}{z_{b_j}} \right)}_{\text{only if }a_i > b_j} \right] 
$$
The unit of this multiplication is $1 \in \text{End}_{\BC(q,t)}(V^{\otimes 0})$. We will refer to $*$ as the shuffle product.

% \begin{prop}[\cite[Propositions 2.6, 4.6, 4.10]{Ng_tale}]
\begin{prop}[\texorpdfstring{\!\!\cite[Propositions 2.6, 4.6, 4.10]{Ng_tale}}{[Prop. 2.6, 4.6, 4.10]}]
\label{prop:preserve}
The subspaces $\CV_{\emph{sym}}$, $\CA^+_{\emph{big}}$ and $\CA^+$ are preserved by $*$.

\end{prop}

With the Proposition above in mind, we will call $\CA^+$ the shuffle algebra, and it will be the main object of our study. 

The shuffle product $*$ can be written in terms of the matrix $\check{R}$:
\begin{align}
    \label{eq:A*B-Rcheck}
    A*B=\sum_{\sigma\in \mathfrak S_{k+l}/\mathfrak S_l\times \mathfrak S_k}
    \frac{1}{f_\sigma(z)}\check R_\sigma^{-1}(z)
    \Gamma^+_{A,B}(z_\sigma) \check R_\sigma(z)
\end{align}
where $f_\sigma(z)$ is a normalizing factor:
\begin{align}\label{eq:f-sigma}
    f_\sigma(z):= \prod_{\substack{1\leq i< j\leq k+l\\ \sigma(i)>\sigma(j)}} f\left(\frac{z_{\sigma(j)}}{z_{\sigma(i)}}\right)
\end{align}
The tensor $\Gamma^+_{A,B}(z_1\ldots z_{k+l})$ is given by:
\begin{multline}
        \label{eq:Gamma_plus}
    \Gamma^+_{A,B}(z_1\ldots z_{k+l}) :=
    A_{l+1\ldots k+l}(z_{l+1}\ldots z_{k+l})
    \check R_{\omega^{l}}(q z_{l+1}\ldots q z_{l+k},z_{1}\ldots z_{l})\\
    B_{k+1\ldots k+l}(z_1\ldots z_{l}) 
    \check R_{\omega^k}(z_1\ldots z_{k+l})
\end{multline}
where $\omega^{j}\in \mathfrak S_{k+l}$ is the rotation by $j$ units: $\omega^j(12\ldots k+l)=(k+l-j+1\ldots k+l,1\ldots k+l-j)$.  
The version of the shuffle product in \eqref{eq:A*B-Rcheck} contains pairs of $\check R$-matrices which can be removed with the help of the unitarity relation \eqref{eq:unitarity}. When one such pair is removed we need to multiply the expression by a factor $f(z_i/z_j)$ for an appropriate pair of indices $i,j$. When all such pairs of $\check R$-matrices are removed the resulting product of factors $f(z_i/z_j)$ will cancel against $f_\sigma(z)$. 

\subsection{The double shuffle algebra}
\label{sub:double}

One of the deeper results concerning the shuffle algebra $\CA^+$ (which is one of the reasons why the wheel conditions in Definition \ref{def:wheel} are natural to impose) is the following.

\begin{prop}
\label{prop:gen}

(\!\!\cite[Corollary 5.30]{Ng_tale}) As a $\BC(q,t)$ algebra, $\CA^+$ is generated by the elements
\begin{equation}
\label{eqn:generators}
E_{ij} = E_{\bar{i}\bar{j}} z_1^{\left \lfloor\frac {i-1}n \right \rfloor - \left \lfloor\frac {j-1}n \right \rfloor} \in \emph{End}(V)[z_1^{\pm 1}]
\end{equation}
as $(i,j)$ runs over $\BZ^2/(n,n)\BZ$, where $\bar{i}$ denotes the element of $\{1\ldots n\}$ congruent to $i$ modulo $n$. 

\end{prop}

\begin{defn}
\label{def:negative shuffle}

Let $\CA^-$ denote the algebra defined just like $\CA^+$, but with the parameter $q^{-1}t^n$ instead of $q$ in Definition \ref{def:wheel} and \eqref{eqn:pole condition}, and with
$$
D_{b_j} R_{b_ja_i}  \left( \frac {z_{b_j}q}{z_{a_i} t^n} \right) D_{b_j}^{-1} \quad \text{instead of} \quad R_{b_ja_i} \left( \frac {z_{b_j}}{z_{a_i} q} \right)
$$
in \eqref{eqn:shuffle product}, where $D$ was defined in \eqref{eq:D}.

\end{defn}
We will reuse the symbol $*$ for the shuffle product of $\CA^-$. For two tensors $A = A_{1\ldots k}(z_1\ldots z_k)\in \CA^-$ and $B = B_{1\ldots l}(z_1\ldots z_l)\in \CA^-$ we have the following analogue of \eqref{eq:A*B-Rcheck}
\begin{align}
    \label{eq:A*B-Rcheck-dual}
    A*B=\sum_{\sigma\in \mathfrak S_{k+l}/\mathfrak S_l\times \mathfrak S_k}
    \frac{1}{f_\sigma(z)}
    \check R_\sigma^{-1}
    \Gamma^-_{A,B}(z_\sigma) \check R_\sigma
\end{align}
where:
\begin{multline}
        \label{eq:Gamma_minus}
    \Gamma^-_{A,B}(z_1\ldots z_{k+l}) :=
    A_{l+1\ldots k+l}(z_{l+1}\ldots z_{k+l})
    \check R^\bullet_{\omega^{l}}(z_{l+1}\ldots  z_{l+k},q z_{1}\ldots q z_{l})\\
    B_{k+1\ldots k+l}(z_1\ldots z_{l}) 
    \check R_{\omega^k}(z_1\ldots z_{k+l})
\end{multline}
It is easy to see that we have an isomorphism $\CA^+ \xrightarrow{\sim} \CA^-$ given by the assignment
$$
X_{1\dots k}(z_1\ldots z_k) \mapsto D_1 \dots D_k X_{1\dots k}(z_1\ldots z_k)  \Big|_{q \mapsto q^{-1}t^n}
$$
In what follows, we will encounter the quantum group of type $\widehat{A}_{n-1}$
\begin{equation}
\label{eqn:sl}
\su = \BC(t) \Big \langle x_i^+, x_i^-, \psi_i^{\pm 1},c^{\pm 1} \Big \rangle_{i \in \{1\ldots n\}} \Big/ \Big( \text{\text{relations (3.4)-(3.7) of \cite{Ng_tale}}} \Big)
\end{equation}
(note that its Cartan subalgebra has one more dimension than usual for the quantum group, but this is a trivial modification) and the quantum Heisenberg algebra
\begin{equation}
\label{eqn:gl 1}
\uui = \BC(t) \Big \langle p_{\pm k}, c^{\pm 1} \Big \rangle_{k \in \BN} \Big/ 
 \Big( [p_k,p_{-k}] = k \frac {c^k - c^{-k}}{t^k - t^{-k}} \text{ and all other commutators 0} \Big)
\end{equation}
We define the quantum affine group associated to $\fgl_n$ by
\begin{equation}
\label{eqn:gl}
\uu = \su \otimes \uui
\end{equation}

\begin{defn}
\label{def:double shuffle}

(see \cite[Section 6]{Ng_tale} for details) The vector space
\begin{equation}
\label{eqn:double a}
\CA = \CA^+ \otimes \uu \otimes \CA^-
\end{equation}
is made into an algebra by specifying how the three factors commute with each other.

\end{defn}

\subsection{The quantum toroidal algebra}
\label{sub:classic shuffle}

We will consider the quantum toroidal algebra
$$
\UU = \BC(q,t) \Big \langle x_{i,k}^+, x_{i,k}^-, \varphi_{i,k'}^+,  \varphi_{i,k'}^-, \psi_i^{\pm 1}, c^{\pm 1}, \bar{c}^{\pm 1} \Big \rangle_{i \in \{1\ldots n\}}^{k \in \BZ, k' >0 } \Big/ \Big( \text{relations (3.41)--(3.47) of \cite{Ng_tale}} \Big)
$$

\begin{thm}
\label{thm:iso}

(\!\!\cite[Theorem 1.5]{Ng_tale}) There exists an isomorphism 
\begin{equation}
\label{eqn:Phi}
\Phi : \UU \xrightarrow{\sim} \CA
\end{equation}
completely determined by sending 
$$
\varphi_{i,1}^\pm \in \UU \qquad \text{to} \qquad E_{ii} \in \CA^\pm \subset \CA
$$
and 
$$
x_{i,0}^\pm,\psi_i \in \UU \quad \text{to} \quad  x_i^\pm,\psi_i \in \uu \subset \CA
$$
for all $i \in \{1\ldots n\}$.

\end{thm}

The isomorphism $\Phi$ is very non-trivial (it uses an explicit description of both $\UU$ and $\CA$ in terms of slope subalgebras determined in \cite{Ng_pbw}) and this has a lot of potential, in the following sense: formulas that are quite easy on one side of the isomorphism are quite difficult on the other side, and vice versa. An example of this is given by Proposition \ref{prop:comm} in Section \ref{sec:Commutative}. Theorem \ref{thm:iso} was proved in \cite{Ng_tor_sh, Ng_pbw,Ng_tale} by an in-depth study of the algebras involved, in which an important part was played by the classic double shuffle algebra realization of $\UU$, which we now recall.

\begin{defn}
\label{def:classic shuffle}

(\!\!\cite{E}, inspired by \cite{FO}) Consider the vector space
\begin{equation}
\label{eqn:classic shuffle}
\CS^+ \subset \bigoplus_{\bd = (d_i \geq 0)_{i \in \{1,\dots,n\}}} \BC(q,t) (x_{i1},\dots,x_{id_i})_{i \in \{1,\dots,n\}}^{\emph{sym}} 
\end{equation}
consisting of rational functions which are

\begin{itemize}[leftmargin=*]

\item symmetric in $x_{i1}, x_{i2},\dots$ for each $i \in \{1,\dots,n\}$ separately

\item have at most poles of the form $x_{ia}t^{-1/2}-x_{i+1,b}t^{1/2}$ for any $i,a,b$, where we write $x_{n+1,b} =q^{-1} x_{1b}$.

\item their residue the pole above is divisible by $x_{ia}-x_{i+1,b'}$ and by $x_{i+1,b}-x_{ia'}$ for any $a \neq a'$ and $b \neq b'$

\end{itemize}

One makes $\CS^+$ into an algebra with respect to the shuffle multiplication
\begin{multline} 
R * R' = \emph{Sym} \left[R(x_{i1},\dots,x_{id_i})_{i \in \{1,\dots,n\}} R'(x_{i,d_i+1},\dots,x_{i,d_i+d_i'})_{i \in \{1,\dots,n\}} \right. \\ \left. \prod_{i,j \in \{1,\dots,n\}} \prod_{a \leq d_i, b > d_j} \begin{cases} \frac {x_{ia} t^{-1/2} - x_{jb}t^{1/2}}{x_{ia}-x_{jb}} & \text{if } i \equiv j \text{ mod }n \\ \frac {x_{ia}-x_{jb}q^{-\delta_{j1}}}{x_{ia} t^{-1/2} - x_{jb}t^{1/2}q^{-\delta_{j1}}} &\text{if } i+1 \equiv j \text{ mod }n \\ 1 & \text{otherwise} \end{cases}\right]
\end{multline}

\end{defn}

One lets $\CS^- = \CS^{+,\text{op}}$ and defines the double shuffle algebra as
$$
\CS = \CS^+ \otimes \uui^{\otimes n} \otimes \CS^-
$$
by appropriately specifying how the three tensor factors should commute with each other (see \cite{Ng_tor_sh} for details). Then the main result of \cite{Ng_tor_sh} is an isomorphism
\begin{equation}
\label{eqn:Upsilon}
\Upsilon: \UU \xrightarrow{\sim} \CS
\end{equation}
determined by sending $x_{i,k}^\pm$ to the Laurent polynomials $x_{i1}^k \in \CS^\pm$, for all $i \in \{1,\dots,n\}$ and $k \in \BZ$. Composing the isomorphisms \eqref{eqn:Phi} and \eqref{eqn:Upsilon} gives us the isomorphism $\CA \cong \CS$ between the two kinds of double shuffle algebras featured in the present paper.

\subsection{The super case}
\label{sub:super}

There exists an analogue of the discussion above for the Lie superalgebra $\fgl_{n|m}$ instead of $\fgl_n$, although much of it is conjectural (we will outline the complete picture and clearly state the missing parts, but stop short of implementing them). One considers $V = \BC^{n+m}$ and replaces the $R$-matrix \eqref{eq:R-matrix} by
\begin{align}\label{eq:R-matrix super}
  R(z) := \sum_{i,j=1}^{n+m}
  \left(\epsilon_i\frac{t^{-\epsilon_i /2}-z t^{\epsilon_i /2}}{1-z}\right)^{\delta_{i=j}}E_{ii}\otimes E_{jj}  
  +
  \sum_{\substack{i,j=1\\ i\neq j}}^{n+m}
  \frac{(t^{-1/2}-t^{1/2})z^{\delta_{i<j}}}{1-z}E_{ij}\otimes E_{ji}  
\end{align}
where:
\begin{align}\label{eq:gradation}
\epsilon_i = (-1)^{\delta_{i> n}}
\end{align}
is the gradation of the standard basis vectors $\ket{i}\in V$, $i=1\ldots n+m$\footnote{We choose the gradation of vectors to be $\{\overset{n}{\overbrace{1\ldots 1}},\overset{m}{\overbrace{-1\ldots -1}}\}$ for simplicity of presentation. More generally, one may take instead any permutation of $\{1\ldots 1,-1\ldots -1\}$.}.
The supersymmetric version of the matrix $\check R$ is given by:
\begin{align}\label{eq:R-check  super}
  \check R(z) := R(z)P
\end{align}
The $R$-matrix \eqref{eq:R-matrix super} has the same fundamental properties as \eqref{eq:R-matrix}: it satisfies the Yang-Baxter equation with spectral parameter \eqref{eq:YB}, unitarity \eqref{eq:unitarity} and crossing unitarity \eqref{eq:crossing_unitarity} in which the matrix $\check R^{\bullet}$ reads:
\begin{align}
    \label{eq:R-bullet-super}
    \check{R}^\bullet(z):=(D\otimes \id) \check{R} (z t^{m-n}) (\id\otimes D^{-1})
\end{align}
where:
\begin{align}
    \label{eq:D-super}
    D := \normalfont{\text{diag}}\{1,t^{-1}\ldots t^{-n+1},(-1)t^{-n+1},(-t)t^{-n+1}\ldots (-t^{m-1})t^{-n+1}\}
\end{align}
The $R$-matrix \eqref{eq:R-matrix super} has a simple pole at $z = 1$ with the same residue as before. This allows the discussion in Subsections \ref{sub:symmetric tensors}, \ref{sub:shuffle product} and \ref{sub:double} to run through, and determines shuffle algebras $\CA^+$ and $\CA^-$ that one can combine into a Drinfeld double
\begin{equation}
\label{eq:double super}
\CA = \CA^+ \otimes U_t(\dot{\fgl}_{n|m}) \otimes \CA^-
\end{equation}

\begin{conj}
\label{conj:super}

There is an isomorphism 
\begin{equation}
\label{eqn:conj super}
\CA \cong U_{q,t}(\ddot{\fgl}_{n|m})
\end{equation}
with the quantum toroidal superalgebra in the right-hand side defined as in \cite{BM}.

\end{conj}

\begin{rmk}
\label{rmk:conj super}

It may be the case that for \eqref{eqn:conj super} to hold, one needs to impose upon its right-hand side additional higher order relations, on top of the ones considered in \cite{BM}. In \cite{Ng_reduced} (in the particular case of the generalized conifold, which is a singular toric Calabi-Yau threefold) the second-named author introduced such higher order relations, with the express goal of making the quantum toroidal superalgebra isomorphic to the corresponding double shuffle algebra $\CS$. 

\end{rmk}

%%% ------ %%% ------ %%% ------ %%% ------ %%% ------ 
%%% ------ %%% ------ %%% ------ %%% ------ %%% ------ 
%%% ------ %%% ------ %%% ------ %%% ------ %%% ------ 

We propose the following approach for the proof of Conjecture \ref{conj:super}, though we point out that working out the details would be a very challenging task (as would be to develop analogues of \eqref{eqn:conj super} in types other than $A$, super or not).

\begin{itemize}

\item Develop the factorization of $U_{q,t}(\ddot{\fgl}_{n|m})$ into slope subalgebras (which we only know how to do using the $\CS$ interpretation of quantum toroidal algebras, see \cite{Ng_tor_sh} for the non-super case).

\item Work out the commutation relations inside $U_{q,t}(\ddot{\fgl}_{n|m})$ of elements from ``nearby'' slope subalgebras (by analogy with the case of quantum toroidal $\fgl_1$ that was worked out in \cite{BS, Schiffmann} and of quantum toroidal $\fgl_n$ that was worked out in \cite{Ng_pbw}).

\item Develop the factorization of $\CA$ into slope subalgebras, following \cite{Ng_tale}.

\item Work out the commutation relations inside $\CA$ of elements from ``nearby'' slope subalgebras.

\item Compare the explicit descriptions in bullets 2 and 4, which may require some combinatorial reindexing of generators and non-trivial checking that relations match up.
 
\end{itemize}

In the remainder of Section \ref{sec:shuffle} and in \ref{sec:Psi}, the objects $R(z)$, $\CA^+$ etc will refer to the notions defined with respect to the super Lie algebra $\fgl_{n|m}$, and our main Theorem \ref{thm:Z_intro} holds conditional on Conjecture \ref{conj:super}. This Conjecture was proved in \cite{Ng_tale} in the non-super case $m=0$, in which case all results contained in the present paper are completely established.

\subsection{Diagrammatic representation}
In our computations it is helpful to use a diagrammatic representation of various tensors.  In this representation the vector space $V_z$ is given by a line with the parameter $z$ attached to it. The matrix $\check R$ is represented by a crossing of two lines:
\begin{align}
    \label{eq:R-Boltzmann}
    \check R(z_{i+1}/z_i) = 
\begin{tikzpicture}[scale=0.6,baseline=-2pt]
\draw[invarrow=0.75] (-1,0) --(1,0);
\draw[invarrow=0.75,] (0,-1) --(0,1);
\node[right] at (1,0) {$\scriptstyle z_{i+1}$};
\node[above] at (0,1) {$\scriptstyle z_i$};
\end{tikzpicture}=
\sum_{a,b,c,d=1\ldots n+m}
\left[
\begin{tikzpicture}[scale=0.6,baseline=-2pt]
\draw[invarrow=0.75] (-1,0) --(1,0);
\draw[invarrow=0.75] (0,-1) --(0,1);
\node[right] at (1+0.5,0) {$\scriptstyle z_{i+1}$};
\node[above] at (0,1+0.5) {$\scriptstyle z_{i}$};
\node[below] at (0,-1) {$\scriptstyle c$};
\node[left] at (-1,0) {$\scriptstyle a$};
\node[right] at (1,0) {$\scriptstyle d$};
\node[above] at (0,1) {$\scriptstyle b$};
\end{tikzpicture}
\right]
 \ket{a,c}\bra{b,d}
\end{align}
where the rightmost expression provides an explicit correspondence between the matrix elements of $\check R$ and the vertices with specified boundary conditions. Thus, using \eqref{eq:R-Boltzmann} and \eqref{eq:R-check  super} one can compute the Boltzmann weights in \eqref{tikz:coloredvertices_intro} in the introduction. When we consider the action of $\check R$ on tensor products of multiple spaces $V$ we will rotate the crossing in \eqref{eq:R-Boltzmann} by 45$\degree$ counterclockwise and add straight lines to represent the action by the identity:
\begin{align}
\label{eq:R_i-graph}
\check R_i(z_{i+1}/z_i)=\quad
\begin{tikzpicture}[baseline=0.5cm]
\draw[invarrow] (0,0) -- (0,1);
\node at (1,0.5) {$\cdots$};
\draw[invarrow] (2,0) -- (2,1);
\draw[invarrow=0.75] (3,0) -- node[pos=0.75,right] {$\scriptstyle z_{i+1}$} (4,1);
\draw[invarrow=0.75] (4,0) -- node[pos=0.75,left] {$\scriptstyle z_i$} (3,1);
\draw[invarrow] (5,0) -- (5,1);
\node at (6,0.5) {$\cdots$};
\draw[invarrow] (7,0) -- (7,1);
\draw[decorate,decoration=brace] (2,-0.2) -- node[below] {$i-1$} (0,-0.2);
\draw[decorate,decoration=brace] (7,-0.2) -- node[below] {$N-i-1$} (5,-0.2);
\end{tikzpicture}
\end{align}
The arrows attached to the vector spaces help us to keep track of the ordering of operators: moving forward w.r.t. the orientation of a line is reading an product of operators right to left.   The matrix $\check R^\bullet$ is represented as follows:
\begin{align}
\label{eq:R_i-bullet-graph}
\check R_i^{\bullet}(z_{i+1}/z_i)=\quad
\begin{tikzpicture}[baseline=0.5cm]
\draw[invarrow] (0,0) -- (0,1);
\node at (1,0.5) {$\cdots$};
\draw[invarrow] (2,0) -- (2,1);
\draw[invarrow=0.75] (3,0) -- node[pos=0.75,right] {$\scriptstyle z_{i+1}$} (4,1);
\draw[invarrow=0.75] (4,0) -- node[pos=0.75,left] {$\scriptstyle z_i$} (3,1);
\draw[invarrow] (5,0) -- (5,1);
\node at (6,0.5) {$\cdots$};
\draw[invarrow] (7,0) -- (7,1);
\draw[decorate,decoration=brace] (2,-0.2) -- node[below] {$i-1$} (0,-0.2);
\draw[decorate,decoration=brace] (7,-0.2) -- node[below] {$N-i-1$} (5,-0.2);
\node at (3.5,0.5) {$\bullet$};
\end{tikzpicture}
\end{align}

Let us recall the Yang--Baxter equation \eqref{eq:YB}, the unitarity relation \eqref{eq:unitarity} and the residue condition at $z=1$ \eqref{eq:R residue}. Graphically, these identities are represented as:
\begin{align}
    \label{eq:YB graph}
\begin{tikzpicture}[scale=0.6,baseline=0.8cm]
\rcheckarrows{0}{2}
\rcheck{1}{1}
\rcheck{0}{0}
\node[above] at (0,3) {$\scriptstyle x$};
\node[above] at (1,3) {$\scriptstyle y$};
\node[above] at (2,3) {$\scriptstyle z$};
\draw[arrow=0.4] (2,3) -- (2,2);
\draw (2,1) -- (2,0);
\draw (0,2) -- (0,1);
\end{tikzpicture}
&=
\begin{tikzpicture}[scale=0.6,baseline=0.8cm]
\rcheckarrows{1}{2}
\rcheck{0}{1}
\rcheck{1}{0}
\node[above] at (0,3) {$\scriptstyle x$};
\node[above] at (1,3) {$\scriptstyle y$};
\node[above] at (2,3) {$\scriptstyle z$};
\draw[arrow=0.4] (0,3) -- (0,2);
\draw (2,2) -- (2,1);
\draw (0,0) -- (0,1);
\end{tikzpicture}
\\
\label{eq:unitarity graph}
\begin{tikzpicture}[scale=0.8,baseline=1cm]
\rcheck{0}{0}
\rcheckarrows{0}{1}
\node[above] at (1,2) {$\scriptstyle y$};
\node[above] at (0,2) {$\scriptstyle x$};
\end{tikzpicture}
    &=f(x/y)
\begin{tikzpicture}[scale=0.8,baseline=0.5cm]
\draw[invarrow=0.5] (3,0) -- (3,1) node[above] {$\scriptstyle x$};
\draw[invarrow=0.5] (4,0) -- (4,1) node[above] {$\scriptstyle y$} ;
\end{tikzpicture}\\
\label{eq:R residue graph}
    \underset{x=y}{\text{Res}}\, \check R(x/y)&=
         (t^{1/2}-t^{-1/2}) \quad 
\begin{tikzpicture}[scale=0.8,baseline=0.5cm]
\draw[invarrow=0.5] (3,0) -- (3,1) node[above] {$\scriptstyle y$};
\draw[invarrow=0.5] (4,0) -- (4,1) node[above] {$\scriptstyle y$} ;
\end{tikzpicture}
\end{align}

A tensor $A\in \text{End}(V^{\otimes k})$ will be represented by a box with $k$ lines attached to it at the top and $k$ lines attached to it at the bottom. These lines represent the $k$ tensor factors of $V^{\otimes k}$, and they carry spectral parameters $z_i$. In some cases these parameters acquire shifts as they ``pass'' through $A$. These shifts are necessary features appearing in the matrix shuffle algebras of $\UU$, as explained in \cite{Ng_tale}. Hence we will encounter two types of tensors:
\begin{align}\label{eq:A_graph}
A(z_1\ldots z_k) = ~
\begin{tikzpicture}[scale=0.5,baseline=(current  bounding  box.center)]
\draw[invarrow=0.4] (1,0) -- (1,0.5);
\draw[invarrow=0.4] (2,0) -- (2,0.5);
\draw[invarrow=0.6] (1,1.5) -- (1,2);
\draw[invarrow=0.6] (2,1.5) -- (2,2);
\draw (0.5,0.5) -- (0.5,1.5) -- (2.5,1.5) -- (2.5,0.5) -- (0.5,0.5);
\node at (1.5,1) {$\scriptstyle A$};
%%% spectral parameters
\node[below] at (1,0) {$\scriptstyle z_1$};
\node[below] at (2,0) {$\scriptstyle z_k$};
\node at (1.5,0) {$\scriptstyle \ldots$};
\node[above] at (1,2) {$\scriptstyle qz_1$};
\node[above] at (2,2) {$\scriptstyle q z_k$};
\node at (1.5,2) {$\scriptstyle \ldots$};
\end{tikzpicture}
\quad\text{and}
\quad
A(z_1\ldots z_k) = ~
\begin{tikzpicture}[scale=0.5,baseline=(current  bounding  box.center)]
\draw[invarrow=0.4] (1,0) -- (1,0.5);
\draw[invarrow=0.4] (2,0) -- (2,0.5);
\draw[invarrow=0.6] (1,1.5) -- (1,2);
\draw[invarrow=0.6] (2,1.5) -- (2,2);
\draw (0.5,0.5) -- (0.5,1.5) -- (2.5,1.5) -- (2.5,0.5) -- (0.5,0.5);
\node at (1.5,1) {$\scriptstyle A$};
%%% spectral parameters
\node[below] at (1,0) {$\scriptstyle z_1$};
\node[below] at (2,0) {$\scriptstyle z_k$};
\node at (1.5,0) {$\scriptstyle \ldots$};
\node[above] at (1,2) {$\scriptstyle z_1$};
\node[above] at (2,2) {$\scriptstyle z_k$};
\node at (1.5,2) {$\scriptstyle \ldots$};
\end{tikzpicture}
\end{align}
The changes of the spectral parameters as they pass through such tensors will be made clear in our pictures and from the context. 

Using the diagrammatic representation \eqref{eq:R_i-graph} of $\check R_i$ we can write $\Gamma^+_{A,B}$ from \eqref{eq:Gamma_plus} as follows:
\begin{align}\label{eq:Gamma_plus_graph}
\Gamma^+_{A,B}=\quad 
\begin{tikzpicture}[scale=0.5,baseline=(current  bounding  box.center)]
\node[above] at (0,8) {$\scriptstyle q z_1$};
\node at (1.5,8) {$\scriptstyle \ldots$};
\node[above] at (3,8) {$\scriptstyle q z_{k+l}$};
\draw[invarrow=0.75] (0,7) -- (0,8);
\draw[invarrow=0.75] (3,7) -- (3,8);
\rcheckarrows{1}{7}
\rcheck{0}{6}
\rcheck{2}{6}
\rcheck{1}{5}
\draw (2-0.5,5) -- (2+1+0.5,5) -- (2+1+0.5,4) -- (2-0.5,4) -- (2-0.5,5);
\node at (2.5,4.5) {$\scriptstyle B$};
\draw (0,6) -- (0,4);
\draw (1,4) -- (1,5);
\draw (3,3) -- (3,4);
\draw (3,5) -- (3,6);
\rcheck{1}{3}
\rcheck{0}{2}
\rcheck{2}{2}
\rcheck{1}{1}
\draw (0,3) -- (0,4);
\draw (0,1) -- (0,2);
\draw (2-0.5,1) -- (2+1+0.5,1) -- (2+1+0.5,0) -- (2-0.5,0) -- (2-0.5,1);
\node at (2.5,0.5) {$\scriptstyle A$};
\draw (0,1) -- (0,-0.5);
\draw (1,1) -- (1,-0.5);
\draw (2,0) -- (2,-0.5);
\draw (3,2) -- (3,1);
\draw (3,0) -- (3,-0.5);
\node[below] at (0,-0.5) {$\scriptstyle z_1$};
\node at (1.5,-0.5) {$\scriptstyle \ldots$};
\node[below] at (3,-0.5) {$\scriptstyle z_{k+l}$};
\end{tikzpicture}
\end{align}
We note that $\check R_{\omega^k}(q z_1\ldots q z_{k+l})=\check R_{\omega^k}(z_1\ldots z_{k+l})$, so the crosses at the top in the diagram \eqref{eq:Gamma_plus_graph} correspond to the factor $\check R_{\omega^k}$ in \eqref{eq:Gamma_plus}. We indicated the spectral parameters at the top and at the bottom of this diagram which specifies the shifts of the parameters as they pass through $A$ and $B$. 

Let us return to Example \ref{ex:wheel u=1} where we computed the wheel condition of a tensor $X(z_1,z_2,z_3)$ for the composition $\lambda=(3)$. The form of the residue of $X(z_1,z_2,z_3)$ is given in \eqref{eqn:residue variables example R-check} and in the graphical language we have:
\begin{align}\label{eq:residue (k) graph}
\underset{\left\{z_1 = y,z_2 =q y,z_{3} = q^{2} y\right\}}{\normalfont{\text{Res}}}X(z_1,z_2, z_3)
 = 
(t^{\frac 12} - t^{-\frac 12})^{2} 
f \left(q \right)
\begin{tikzpicture}[scale=0.5,baseline=(current  bounding  box.center)]
\rcheck{0}{0}
\rcheck{1}{1}
% level 2
\draw (1.5-0.1,2) -- (1.5-0.1,3) -- (2.5+0.1,3) -- (2.5+0.1,2) -- (1.5-0.1,2);
\node at (2,2.5) {$\scriptstyle X^{(3)}$};
\draw[arrow=0.5] (2,4) -- (2,3);
\draw (2,1) -- (2,0);
\draw[arrow=0.25] (1,4) -- (1,2);
\draw[arrow=0.166] (0,4) -- (0,1);
\node[above] at (0,4) {$\scriptstyle q y$};
\node[above] at (1,4) {$\scriptstyle q^2 y$};
\node[above] at (2,4) {$\scriptstyle q^3 y$};
\node[below] at (0,0-0.2) {$\scriptstyle y$};
\node[below] at (1,0-0.2) {$\scriptstyle q y$};
\node[below] at (2,0) {$\scriptstyle q^2 y$};
\end{tikzpicture}
\end{align}

\section{An anti-isomorphism of shuffle algebras}\label{sec:Psi}
In this section we introduce an anti-isomorphism $\Psi$ which relates the shuffle algebras $\CA^+$ and an algebra $\CA'$, which is obtained from $\CA^-$ by reversing the order of the tensor product spaces and inverting $q$ and $t$.  This anti-isomorphism can be extended to a map $\widetilde\Psi$ relating shuffle algebras associated to different supersymmetric quantum toroidal algebras. In the next section we will use both $\Psi$ and $\widetilde \Psi$ as tools for producing new commuting elements of matrix shuffle algebras.

\subsection{The shuffle algebra $\CA'$}\label{sec:A-prime}
We start by defining a transform $\Omega$ which acts on tensors in $\CV$ by:
\begin{align}\label{eq:Omega}
    \Omega(X_{1\ldots k}(z_1\ldots z_k)) = X_{k\ldots 1}(z_{k}\ldots z_{1})\Big{|}_{q,t\rightarrow q^{-1},t^{-1}},
    \qquad X(z_1\ldots z_k)\in \CV
\end{align}
It is clear that $\Omega(A * B)=\Omega(A) * \Omega(B)$ and that $\Omega$ is an involution. For any two tensors $A = A_{1\ldots k}(z_1\ldots z_k)$ and $B = B_{1\ldots l}(z_1\ldots z_l)$ in $\CV$ we define: 
\begin{multline}
        \label{eq:Gamma_prime}
    \Gamma'_{A,B}(z_1\ldots z_{k+l}) :=
    A_{1\ldots k}(z_{1}\ldots z_{k})
    \check R^\bullet_{\omega^{k}}(z_{k+1}\ldots  z_{k+l},q z_{1}\ldots q z_{k})\\
    B_{1\ldots l}(z_{k+1}\ldots z_{k+l}) 
    \check R_{\omega^l}(z_1\ldots z_{k+l})
\end{multline}
\begin{lem}
Let $A = A_{1\ldots k}(z_1\ldots z_k)\in\CV$ and $B = B_{1\ldots l}(z_1\ldots z_l)\in\CV$ and $\rho\in \mathfrak S_{k+l}$ be the longest element. 
The transform $\Omega$ satisfies:
\begin{align}\label{eq:Omega_props}
\Omega(\Gamma^-_{A,B}) = \Gamma'_{\Omega(A),\Omega(B)},
\qquad 
\Omega(\check R_{\sigma}) = \check R_{\rho \sigma \rho}
\end{align}
where $\Gamma^-$ is defined in \eqref{eq:Gamma_minus}.
\end{lem}
\begin{proof}
The first identity in \eqref{eq:Omega_props} follows by applying $\Omega$ to the definition of $\Gamma^-$ and using the second identity in \eqref{eq:Omega_props}. For $\sigma$ a simple  transposition the second identity in \eqref{eq:Omega_props} is a property of the $\check R$-matrix which can be verified directly. Using this and \eqref{eq:Rch-sigma} one verifies the second identity for a general $\sigma$.
\end{proof}
\begin{defn}\label{def:A prime}
Let $\CA'$ be the algebra of tensors $\{\Omega(X)\}_{X \in \CA^-}$ with the multiplication $*'$: 
\begin{align}
    \label{eq:A*'B-Rcheck}
    A*'B=\sum_{\sigma\in \mathfrak{S}_{k+l}/\mathfrak{S}_k\times \mathfrak{S}_l}\frac{1}{f_\sigma(z)}
    \check R_\sigma^{-1}
    \Gamma'_{A,B}(z_\sigma) \check R_\sigma
\end{align}
for any $A = A_{1\ldots k}(z_1\ldots z_k)$ and $B = B_{1\ldots l}(z_1\ldots z_l)$ in $\CA'$.
\end{defn}
The shuffle product $*'$ in \eqref{eq:A*'B-Rcheck} is related to the shuffle product $*$ of the algebra $\CA^-$ by the two properties in \eqref{eq:Omega_props}.
The graphical representation of $\Gamma'$ is given by:
\begin{align}
    \label{eq:Gamma_prime-graph}
    \Gamma'_{A,B}
    :=
\begin{tikzpicture}[scale=0.5,baseline=(current  bounding  box.center)]
\node[above] at (0,8) {$\scriptstyle q z_1$};
\node at (1.5,8) {$\scriptstyle \ldots$};
\node[above] at (3,8) {$\scriptstyle q z_{k+l}$};
\draw[invarrow=0.75] (0,7) -- (0,8);
\draw[invarrow=0.75] (3,7) -- (3,8);
\rcheckarrows{1}{7}
\rcheck{0}{6}
\rcheck{2}{6}
\rcheck{1}{5}
\draw (-0.5,5) -- (1+0.5,5) -- (1+0.5,4) -- (-0.5,4) -- (-0.5,5);
\node at (0.5,4.5) {$\scriptstyle B$};
\draw (0,6) -- (0,5);
\draw (2,4) -- (2,5);
\draw (3,3) -- (3,6);
\rcheck{1}{3}
\rcheck{0}{2}
\rcheck{2}{2}
\rcheck{1}{1}
\node at (1.5,3.5) {$\scriptstyle \bullet$};
\node at (0.5,2.5) {$\scriptstyle \bullet$};
\node at (2.5,2.5) {$\scriptstyle \bullet$};
\node at (1.5,1.5) {$\scriptstyle \bullet$};
\draw (0,3) -- (0,4);
\draw (0,1) -- (0,2);
\draw (-0.5,1) -- (1+0.5,1) -- (1+0.5,0) -- (-0.5,0) -- (-0.5,1);
\node at (0.5,0.5) {$\scriptstyle A$};
\draw (0,0) -- (0,-0.5);
\draw (1,0) -- (1,-0.5);
\draw (2,1) -- (2,-0.5);
\draw (3,2) -- (3,-0.5);
\node[below] at (0,-0.5) {$\scriptstyle z_1$};
\node at (1.5,-0.5) {$\scriptstyle \ldots$};
\node[below] at (3,-0.5) {$\scriptstyle z_{k+l}$};
\end{tikzpicture}
\end{align}

\subsection{The maps $\Psi$ and $\Psi'$}
In this section we define two maps $\Psi$ and $\Psi'$ which we will later show to be algebra isomorphisms. A key ingredient in these two maps is the tensor $\check R_{\omega^l}$. For two sets of parameters $z=(z_1\ldots z_k)$ and $w=(w_1\ldots w_l)$ and $\omega^l\in \mathfrak{S}_{k+l}$ we have:
\begin{align}
    \check{R}_{\omega^l}(z,w)= \check{R}_{\omega^l}(z_1\ldots z_k,w_1\ldots w_l)
    =\prod_{j=1}^k\prod_{i=1}^l \check{R}_{l+j-i}\left(\frac{w_{l-i+1}}{z_j}\right)
\end{align}
Graphically we can represent $\check{R}_{\omega^l}$ by: 
\begin{align}\label{eq:Z_graph}
\check{R}_{\omega^l}(z,w)=\quad 
\begin{tikzpicture}[scale=0.5,baseline=(current  bounding  box.center)]
\rcheckarrows{1}{2}
\rcheck{0}{1}
\rcheck{2}{1}
\rcheck{1}{0}
\draw[invarrow=0.75] (3,2) -- (3,3);
\draw (3,0) -- (3,1);
\draw[invarrow=0.75] (0,2) -- (0,3);
\draw (0,0) -- (0,1);
\node[below] at (0,0) {$\scriptstyle w_1$};
\node[below] at (1,0) {$\scriptstyle w_l$};
\node[below] at (2,0) {$\scriptstyle z_{1}$};
\node[below] at (3,0) {$\scriptstyle z_{k}$};
\node at (0.5,0) {$\scriptstyle \ldots $};
\node at (2.5,0) {$\scriptstyle \ldots $};
\node[above] at (0,3) {$\scriptstyle z_1$};
\node[above] at (1,3) {$\scriptstyle z_k$};
\node[above] at (2,3) {$\scriptstyle w_{1}$};
\node[above] at (3,3) {$\scriptstyle w_{l}$};
\node at (0.5,3) {$\scriptstyle \ldots $};
\node at (2.5,3) {$\scriptstyle \ldots $};
\end{tikzpicture}
\end{align}
Thanks to the Yang--Baxter equation, the tensor $\check R_{\omega^l}$ satisfies:
\begin{align}\label{eq:R-omega_exchange1}
    \check{R}_{\omega^l}(z,w)(\ldots  z_{i+1}, z_{i}\ldots,w) \check R_i(z_{i+1}/z_i)&=
    \check R_{l+i}(z_{i+1}/z_i) \check{R}_{\omega^l}(\ldots z_{i},  z_{i+1}\ldots,w) \\
    \label{eq:R-omega_exchange2}
    \check{R}_{\omega^l}(z,w)(z,\ldots  w_{j+1}, w_{j}\ldots) \check R_{k+j}(w_{j+1}/w_j)&=
    \check R_{j}(w_{j+1}/w_j) \check{R}_{\omega^l}(z,\ldots w_{j},  w_{j+1}\ldots)
\end{align}
for $i=1\ldots k-1$ and $j=1\ldots l-1$.
\begin{defn}
    For any $X_{1\ldots k}(z_1\ldots z_k)\in\CV$ we define two  maps:
\begin{align}\label{eq:Psi}
\Psi\left [ X(z)\right]&: =  \Tr_{k+1\ldots 2k}
\left[
\check{R}_{\omega^k}(q z,z)
X_{k+1\ldots 2k}(z)
\right]\\
\label{eq:Psi_prime}
\Psi'\left [ X(z)\right]&: =  \Tr_{1\ldots k}
\left[
\check{R}^\bullet_{\omega^k}(z,q z)
X_{1\ldots k}(z)
\right]
\end{align}
\end{defn}
The graphical representation of \eqref{eq:Psi} follows from \eqref{eq:A_graph} and \eqref{eq:Z_graph}, and it may be drawn in the two ways below due to the cyclic property of the trace:
\begin{align}\label{eq:Psi_graph}
\Psi[X(z)]=\quad 
\begin{tikzpicture}[scale=0.5,baseline=(current  bounding  box.center)]
\rcheck{1}{2}
\rcheck{0}{1}
\rcheck{2}{1}
\rcheck{1}{0}
\draw (3,2) -- (3,3);
\draw (2-0.5,3) -- (3+0.5,3) -- (3+0.5,4) -- (2-0.5,4) -- (2-0.5,3);
\node at (2.5,3.5) {$\scriptstyle X$};
\draw [rounded corners=5pt] (3,1) -- (4,1) -- (4,4.5) -- (3,4.5) -- (3,4);
\draw [rounded corners=5pt] (2,0) -- (5,0) -- (5,5) -- (2,5) -- (2,4);
\draw[invarrow=0.85] (0,2) -- (0,5);
\draw[invarrow=0.78] (1,3) -- (1,5);
\draw (0,1) -- (0,0);
\node[above] at (0,5) {$\scriptstyle q z_1$};
\node at (0.5,5) {$\scriptstyle \ldots$};
\node[above] at (1,5) {$\scriptstyle q z_k$};
\node[below] at (0,0) {$\scriptstyle z_1$};
\node at (0.5,0) {$\scriptstyle \ldots$};
\node[below] at (1,0) {$\scriptstyle z_k$};
\end{tikzpicture}
\qquad
=
\begin{tikzpicture}[scale=0.5,baseline=(current  bounding  box.center)]
\rcheckarrowleft{1}{4}
\rcheck{0}{3}
\rcheck{2}{3}
\rcheck{1}{2}
\draw (2-0.5,2) -- (3+0.5,2) -- (3+0.5,1) -- (2-0.5,1) -- (2-0.5,2);
\node at (2.5,1.5) {$\scriptstyle X$};
\draw (3,3) -- (3,2);
\draw (0,3) -- (0,0);
\draw (1,2) -- (1,0);
\draw[invarrow=0.75] (0,4) -- (0,5);
\node[above] at (0,5) {$\scriptstyle q z_1$};
\node at (0.5,5) {$\scriptstyle \ldots$};
\node[above] at (1,5) {$\scriptstyle q z_k$};
\draw [rounded corners=5pt] (3,4) -- (4,4) -- (4,0.5) -- (3,0.5) -- (3,1);
\draw [rounded corners=5pt] (2,5) -- (5,5) -- (5,0) -- (2,0) -- (2,1);
\node[below] at (0,0) {$\scriptstyle z_1$};
\node at (0.5,0) {$\scriptstyle \ldots$};
\node[below] at (1,0) {$\scriptstyle z_k$};
\end{tikzpicture}
\end{align}
The graphical representation of $\Psi'[X(z)]$ is analogous. These two maps have several important properties which are outlined in the following Lemma.
\begin{lem}\label{lem:prop_psi}
For $X(z_1\ldots z_k)\in\CV$ we have the relations:
\begin{align}\label{eq:Psi-exchange}
    &\check R_{\sigma}^{-1}
    \Psi\left[X(z_\sigma)\right]
    \check R_\sigma=
    \Psi\left[\check R_{\sigma}^{-1}X(z_\sigma)\check R_{\sigma} \right]
    \\
    \label{eq:Psi_prime-exchange}
    &\check R_{\sigma}^{-1}
    \Psi'\left[X(z_\sigma)\right]
    \check R_\sigma=
    \Psi'\left[\check R_{\sigma}^{-1}X(z_\sigma)\check R_{\sigma} \right]
\end{align}
If $X$ is a symmetric tensor then $\Psi[X(z)]$ and $\Psi'[X(z)]$ are also symmetric tensors:
\begin{align}
    \label{eq:Psi-symmetric}
    \check R_\sigma 
\Psi[X(z)] &= 
\Psi[X(z_{\sigma})]
\check R_\sigma\\
    \label{eq:Psi_prime-symmetric}
    \check R_\sigma 
\Psi'[X(z)] &= 
\Psi'[X(z_{\sigma})]
\check R_\sigma
\end{align}
\end{lem}
\begin{proof}
It is enough to prove \eqref{eq:Psi-exchange} for all simple transpositions, i.e.:
\begin{align}\label{eq:Psi-exchange_i}
\check R_{i}(z_{i+1}/z_i)^{-1}
    \Psi\left[X(\ldots z_{i+1},z_i\ldots)\right]
    \check R_i(z_{i+1}/z_i)=
    \Psi\left[\check R_{i}(z_{i+1}/z_i)^{-1}X(\ldots z_{i+1},z_i\ldots)\check R_{i}(z_{i+1}/z_i) \right]    
\end{align}
Consider \eqref{eq:R-omega_exchange1}  and \eqref{eq:R-omega_exchange2} with $l=k$, $j=i$ and $z_a\rightarrow qz_a$ and $w_a\rightarrow  z_a$, for $a=1\ldots k$. These two equations can be combined to produce the following identity:
\begin{multline}\label{eq:R_omega_exchange_3}
\check R_i(z_{i+1}/z_i)^{-1} \check R_{\omega^k}(\ldots qz_{i+1},q z_i\ldots,\ldots z_{i+1},z_i\ldots) \check R_i(z_{i+1}/z_i)\\
=
\check R_{k+i}(z_{i+1}/z_i) \check R_{\omega^k}(qz,z) \check R_{k+i}(z_{i+1}/z_i)^{-1}
\end{multline}
Plugging the formula above in the left hand side of \eqref{eq:Psi-exchange_i} (and using the definition of $\Psi$ in \eqref{eq:Psi}, the identity \eqref{eq:R_omega_exchange_3} and the cyclicity of the trace) gives:
\begin{align*}
&\check R_{i}(z_{i+1}/z_i)^{-1}
    \Psi\left[X(\ldots z_{i+1},z_i\ldots)\right]
    \check R_i(z_{i+1}/z_i)\\
    =
&
   \Tr_{k+1\ldots 2k}
\left[
\check R_{k+i}(z_{i+1}/z_i) \check R_{\omega^k}(qz,z) \check R_{k+i}(z_{i+1}/z_i)^{-1}
X_{k+1\ldots 2k}(\ldots z_{i+1},z_i\ldots)
\right]\\
   =
&
   \Tr_{k+1\ldots 2k}
\left[ \check R_{\omega^k}(qz,z) \check R_{k+i}(z_{i+1}/z_i)^{-1}
X_{k+1\ldots 2k}(\ldots z_{i+1},z_i\ldots)
\check R_{k+i}(z_{i+1}/z_i)
\right]\\
=
    &\Psi\left[\check R_{i}(z_{i+1}/z_i)^{-1}X(\ldots z_{i+1},z_i\ldots)\check R_{i}(z_{i+1}/z_i) \right]    
\end{align*}
This proves \eqref{eq:Psi-exchange_i}, which implies \eqref{eq:Psi-exchange}. If $X(z)$ is a symmetric tensor then \eqref{eq:Psi-symmetric} holds by \eqref{eq:Psi-exchange} and \eqref{eqn:symmetric tensor}. The proofs of \eqref{eq:Psi_prime-exchange} and \eqref{eq:Psi_prime-symmetric} are analogous.
\end{proof}

\subsection{A shuffle algebras anti-isomorphism}
Here we show that $\Psi$ and $\Psi'$ are anti-isomorphisms and are inverses of each other.
\begin{prop}
For any tensors $A(z_1\ldots z_l), B(z_1\ldots z_k)\in \CV$ we have:
\begin{align}\label{eq:Gamma-Psi}
&\Gamma^+_{\Psi[A],\Psi[B]}   = \Psi[\Gamma'_{B,A}]\\
\label{eq:Gamma-Psi-prime}
&\Gamma'_{\Psi'[A],\Psi'[B]}   = \Psi'[\Gamma^+_{B,A}]
\end{align}
Additionally we have:
\begin{align}
    \label{eq:Psi-inv}
    \Psi[\Psi'[A(z)]] = \Psi'[\Psi[A(z)]] = A(z)
\end{align}
\end{prop}
\begin{proof}
We give a graphical proof. 
Using \eqref{eq:Gamma_plus_graph} and \eqref{eq:Psi_graph} we can represent $\Gamma^+_{\Psi[A],\Psi[B]}$ as:
\begin{align*}
\Gamma^+_{\Psi[A],\Psi[B]}
=
\begin{tikzpicture}[scale=0.4,baseline=(current  bounding  box.center)]
% Level 4
\rcheckarrows{1}{6}
\rcheck{0}{5}
\rcheck{2}{5}
\rcheck{1}{4}
\draw[invarrow=0.75] (0,6) -- (0,7);
\draw[invarrow=0.75] (3,6) -- (3,7);
\draw (0,4) -- (0,5);
\node[above] at (0,7) {$\scriptstyle q z_1$};
\node at (0.5,7) {$\scriptstyle \ldots$};
\node at (2.5,7) {$\scriptstyle \ldots$};
\node[above] at (3,6.93) {$\scriptstyle q z_{l+k}$};
% Level 3
\rcheck{3}{4}
\rcheck{2}{3}
\rcheck{4}{3}
\rcheck{3}{2}
\draw (5,4) -- (5,5);
\draw (4-0.5,5) -- (5+0.5,5) -- (5+0.5,6) -- (4-0.5,6) -- (4-0.5,5);
\node at (4.5,5.5) {$\scriptstyle B$};
\draw [rounded corners=5pt] (5,3) -- (6,3) -- (6,6.5) -- (5,6.5) -- (5,6);
\draw [rounded corners=5pt] (4,2) -- (7,2) -- (7,7) -- (4,7) -- (4,6);
% Level 2
\draw (0,4) -- (0,2);
\draw (1,4) -- (1,3);
\rcheck{1}{2}
\rcheck{0}{1}
\rcheck{2}{1}
\rcheck{1}{0}
% Level 1
\rcheck{3}{0}
\rcheck{2}{-1}
\rcheck{4}{-1}
\rcheck{3}{-2}
\draw (4-0.5,-2) -- (5+0.5,-2) -- (5+0.5,-3) -- (4-0.5,-3) -- (4-0.5,-2);
\node at (4.5,-2.5) {$\scriptstyle A$};
\draw (5,-1) -- (5,-2);
\draw (0,1) -- (0,-4);
\draw (1,0) -- (1,-4);
\draw (2,-1) -- (2,-4);
\draw (3,-2) -- (3,-4);
\draw [rounded corners=5pt] (5,0) -- (6,0) -- (6,-3.5) -- (5,-3.5) -- (5,-3);
\draw [rounded corners=5pt] (4,1) -- (7,1) -- (7,-4) -- (4,-4) -- (4,-3);
\node[below] at (0,-4) {$\scriptstyle z_1$};
\node at (0.5,-4) {$\scriptstyle \ldots$};
\node at (2.5,-4) {$\scriptstyle \ldots$};
\node[below] at (3,-4) {$\scriptstyle z_{l+k}$};
\end{tikzpicture}
\end{align*}
In $\Gamma^+_{\Psi[A],\Psi[B]}$ we have traces over $l$ vector spaces corresponding to $\Psi[A]$ and traces over $k$ vector spaces corresponding to $\Psi[B]$. In Lemma \ref{lem:traces} (specifically formula \eqref{eq:Y-graph}) we will show that the picture above equals:
\begin{align*}
\Gamma^+_{\Psi[A],\Psi[B]}
=
\begin{tikzpicture}[scale=0.4,baseline=(current  bounding  box.center)]
% Level 4
\rcheckarrows{1}{6}
\rcheck{0}{5}
\rcheck{2}{5}
\rcheck{1}{4}
\draw[invarrow=0.75] (0,6) -- (0,7);
\draw[invarrow=0.75] (3,6) -- (3,7);
\draw (0,4) -- (0,5);
% Level 3
\rcheck{3}{4}
\rcheck{2}{3}
\rcheck{4}{3}
\rcheck{3}{2}
\draw (5,4) -- (5,5);
\draw (4-0.5,5) -- (5+0.5,5) -- (5+0.5,6) -- (4-0.5,6) -- (4-0.5,5);
\node at (4.5,5.5) {$\scriptstyle B$};
% Level 2
\draw (0,4) -- (0,2);
\draw (1,4) -- (1,3);
\rcheck{1}{2}
\rcheck{0}{1}
\rcheck{2}{1}
\rcheck{1}{0}
% r-matrices due to crossing unitarity
\rcheck{5}{2}
\rcheck{4}{1}
\rcheck{6}{1}
\rcheck{5}{0}
% Level 1
\rcheck{3}{0}
\rcheck{2}{-1}
\rcheck{4}{-1}
\rcheck{3}{-2}
\draw (4-0.5,-2) -- (5+0.5,-2) -- (5+0.5,-3) -- (4-0.5,-3) -- (4-0.5,-2);
\node at (4.5,-2.5) {$\scriptstyle A$};
\draw (5,-1) -- (5,-2);
\draw (0,1) -- (0,-8);
\draw (1,0) -- (1,-8);
\draw (2,-1) -- (2,-8);
\draw (3,-2) -- (3,-8);
\draw (6,-3) -- (6,0);
\draw (7,-4) -- (7,1);
% Level -1
\rcheck{5}{-4}
\rcheck{4}{-5}
\rcheck{6}{-5}
\rcheck{5}{-6}
\node at (5.5,-3.5) {$\scriptstyle \bullet$};
\node at (4.5,-4.5) {$\scriptstyle \bullet$};
\node at (6.5,-4.5) {$\scriptstyle \bullet$};
\node at (5.5,-5.5) {$\scriptstyle \bullet$};
\draw (4,-4) -- (4,-3);
\draw [rounded corners=5pt] (7,2) -- (8,2) -- (8,-5) -- (7,-5);
\draw [rounded corners=5pt] (6,3) -- (9,3) -- (9,-6) -- (6,-6);
\draw [rounded corners=5pt] (5,6) -- (5,6.5) -- (10,6.5) -- (10,-7) -- (5,-7) -- (5,-6);
\draw [rounded corners=5pt] (4,6) -- (4,7.5) -- (11,7.5) -- (11,-8) -- (4,-8) -- (4,-5);
\end{tikzpicture}
=
~~
%%%%%%%%%%%%%%%%%%%%%%%%%
%%%%%%%%%%%%%%%%%%%%%%%%%
\begin{tikzpicture}[scale=0.4,baseline=(current  bounding  box.center)]
% Level 4
% \rcheckarrows{1}{6}
% \rcheck{0}{5}
% \rcheck{2}{5}
% \rcheck{1}{4}
\draw[invarrow=0.75] (0,4) -- (0,5);
\draw[invarrow=0.75] (1,4) -- (1,5);
\draw[invarrow=0.75] (2,4) -- (2,5);
\draw[invarrow=0.75] (4,4) -- (3,5);
\draw (0,2) -- (0,4);
\draw (1,3) -- (1,4);
% Level 3
\rcheck{3}{4}
\rcheck{2}{3}
\rcheck{4}{3}
\rcheck{3}{2}
% Level 2
\rcheck{1}{2}
\rcheck{0}{1}
\rcheck{2}{1}
\rcheck{1}{0}
% r-matrices due to crossing unitarity
\rcheck{5}{2}
\rcheck{4}{1}
\rcheck{6}{1}
\rcheck{5}{0}
% Level 1
\rcheck{3}{0}
\rcheck{2}{-1}
\rcheck{4}{-1}
\rcheck{3}{-2}
\rcheck{5}{-2}
\rcheck{4}{-3}
\rcheck{6}{-3}
\rcheck{5}{-4}
\draw (6,0) -- (6,-1);
\draw (7,1) -- (7,-2);
\draw (4-0.5,-4) -- (5+0.5,-4) -- (5+0.5,-5) -- (4-0.5,-5) -- (4-0.5,-4);
\node at (4.5,-4.5) {$\scriptstyle A$};
\draw (4,-3) -- (4,-4);
\draw (0,1) -- (0,-10.5);
\draw (1,0) -- (1,-10.5);
\draw (2,-1) -- (2,-10.5);
\draw (3,-2) -- (3,-10.5);
\draw (6,-5) -- (6,-4);
\draw (7,-6) -- (7,-3);
% Level -1
\rcheck{5}{-6}
\rcheck{4}{-7}
\rcheck{6}{-7}
\rcheck{5}{-8}
\node at (5.5,-5.5) {$\scriptstyle \bullet$};
\node at (4.5,-6.5) {$\scriptstyle \bullet$};
\node at (6.5,-6.5) {$\scriptstyle \bullet$};
\node at (5.5,-7.5) {$\scriptstyle \bullet$};
\draw (4,-6) -- (4,-5);
\draw (4,-8) -- (4,-7);
\draw [rounded corners=5pt] (7,2) -- (8,2) -- (8,-7) -- (7,-7);
\draw [rounded corners=5pt] (6,3) -- (9,3) -- (9,-8) -- (6,-8);
\draw [rounded corners=5pt] (5,4) -- (10,4) -- (10,-9.5) -- (5,-9.5) -- (5,-9);
\draw [rounded corners=5pt] (4,5) -- (11,5) -- (11,-10.5) -- (4,-10.5) -- (4,-9);
\draw (4-0.5,-8) -- (5+0.5,-8) -- (5+0.5,-9) -- (4-0.5,-9) -- (4-0.5,-8);
\node at (4.5,-8.5) {$\scriptstyle B$};
\end{tikzpicture}
~~~ \stackrel{\eqref{eq:Gamma_prime-graph},\eqref{eq:Psi_graph}}=
\Psi[\Gamma'_{B,A}]
\end{align*}
In the second equality we used the Yang--Baxter equation repeatedly. This concludes the proof of \eqref{eq:Gamma-Psi}. The proof of the second equation \eqref{eq:Gamma-Psi-prime} is analogous. The proof of \eqref{eq:Psi-inv} uses the same ideas as the proof of Lemma \ref{lem:traces}, we leave it to the interested reader. 
\end{proof}
\begin{cor}\label{cor:psi iso}
There exists an anti-isomorphism of shuffle algebras $\Psi : \CA' \xrightarrow{\sim} \CA^+$ given by:
\begin{align}\label{eq:psi-iso}
     X \mapsto \Psi[X]
\end{align}
with inverse $\Psi'$.
\end{cor}
\begin{proof}
Let $A(z_1\ldots z_l), B(z_1\ldots z_k)\in \CA'$. We apply $\Psi$ to the shuffle product $*'$:
\begin{multline}\label{eq:proof-psi-iso}
        \Psi[A*'B]=\sum_{\sigma\in \mathfrak S_{k+l}/\mathfrak S_l\times \mathfrak S_k}
    \Psi[\check R_\sigma^{-1}
    \Gamma'_{A,B}(z_\sigma) \check R_\sigma]
    =\sum_{\sigma\in \mathfrak S_{k+l}/\mathfrak S_l\times \mathfrak S_k}
    \check R_\sigma^{-1}\Psi[
    \Gamma'_{A,B}(z_\sigma) ]\check R_\sigma\\
    =\sum_{\sigma\in \mathfrak S_{k+l}/\mathfrak S_l\times \mathfrak S_k}
    \check R_\sigma^{-1}
    \Gamma^+_{\Psi[B],\Psi[A]}(z_\sigma)
    \check R_\sigma
    =\Psi[B]*\Psi[A]
\end{multline}
where we used \eqref{eq:A*'B-Rcheck}, \eqref{eq:Psi-exchange}, \eqref{eq:Gamma-Psi} and \eqref{eq:A*B-Rcheck}, respectively. This implies that $\Psi$ and $\Psi'$ give mutually inverse algebra anti-homomorphisms between $\CV$ and $\CV'$. The map $\Psi$ sends elements of $\text{End}(V)(z_1^{\pm 1})$ to elements of $\text{End}(V)(z_1^{\pm 1})$ and therefore by Proposition \ref{prop:gen} this implies that $\Psi$ sends $\CA'$ to $\CA^+$.
\end{proof}

\subsection{Shuffle algebras' anti-homomorphisms}\label{sec:Psi tilde}
In this section we introduce linear maps $\widetilde \Psi$ which arise from composing $\Psi$ with projections. These linear maps become anti-homomorphisms under certain conditions.

Recall the vector space $V \cong \mathbb C^{n+m}$ and consider two vectors spaces $V'\cong \mathbb C^{n'+m'}$ and $V''\cong \mathbb C^{n''+m''}$ such that $V$ and $V'$ are two coordinate subspaces of $V''$. Let $\mathcal{I}'$ and $\mathcal{I}''$ denote the label sets of the basis vectors of $V'$ and $V''$, respectively. Fix two projections $\pi$ and $\pi'$ and embeddings $\iota$ and $\iota'$:
\begin{align}\label{eq:proj-emb}
\pi:V''\to V, 
\qquad \pi':V''\to V',
\qquad
\iota:V\to V'', 
\qquad 
\iota':V'\to V''
\end{align}
such that $\pi$ removes $n''-n$ standard basis vectors with bosonic grading and $m''-m$ standard basis vectors with fermionic grading and similarly $\pi'$ removes $n''-n'$ standard basis vectors with bosonic grading and $m''-m'$ standard basis vectors with fermionic grading. The embeddings $\iota$ and $\iota'$ act as inverses on the corresponding subspaces. The maps $\pi,\pi',\iota$ and $\iota'$ can be reused to denote the maps of the labeling sets by replacing $V$ with $\mathcal{I}$ in \eqref{eq:proj-emb}. In particular, the maps $\pi(\mathcal{I}'')$ and $\pi'(\mathcal{I}'')$ are order preserving and we have $\iota(\mathcal{I})\cup \iota'(\mathcal{I}')\subseteq\mathcal{I}''$.

Let $R^{(V')}$ and $R^{(V'')}$ be the $R$-matrices of $U_t(\dot{\fgl}_{n'|m'})$ and $U_t(\dot{\fgl}_{n''|m''})$ acting on $V'\otimes V'$ and $V''\otimes V''$, respectively. Let $\epsilon'_i$, with $i\in\mathcal{I}'$, and $\epsilon_i''$, with $i\in \mathcal{I}''$, be the gradings of the standard basis vectors of $V'$ and $V''$, respectively. We can write:
\begin{align}
    \label{eq:projections}
    (\pi\otimes \pi)R^{(V'')}(z)(\iota\otimes \iota) = R(z),
    \qquad
    (\pi'\otimes \pi')R^{(V'')}(z)(\iota'\otimes \iota') = R^{(V')}(z)
\end{align}
where we require that $\epsilon''_{\iota(i)}=\epsilon_i$ and $\epsilon''_{\iota'(i)}=\epsilon'_i$.
The same relations hold for $\check R$. Let: 
\begin{equation}
\label{eqn:big space primed}
\CV' = \bigoplus_{k=0}^{\infty} \text{End}_{\BC(q,t)}(V'^{\otimes k}) (z_1 \ldots z_k)
,
\qquad
\CV'' = \bigoplus_{k=0}^{\infty} \text{End}_{\BC(q,t)}(V''^{\otimes k}) (z_1 \ldots z_k)
\end{equation}
Consider two tensors $A = A_{1\ldots k}(z_1\ldots z_k)\in \CV''$ and $B = B_{1\ldots l}(z_1\ldots z_l)\in \CV''$ which preserve the coordinate subspaces $V^{\otimes k}$ and $V^{\otimes l}$, respectively. Using \eqref{eq:projections} one can show that:
\begin{align}
    \label{eq:proj hom}
    \pi^{\otimes (k+l)}\left( A*''B \right)\iota^{\otimes (k+l)}=
    \left(\pi^{\otimes k} A \iota^{\otimes k}\right) *\left(\pi^{\otimes l} B \iota^{\otimes l}\right)
\end{align}
where $*''$ denotes the shuffle product \eqref{eqn:shuffle product} with $R$ replaced by $R^{(V'')}$. Let us stress that the space $V''$ is used as an auxiliary space and the gradation of the standard basis vectors of $V''$ does not need to be ordered as $\{1\ldots 1,-1\ldots -1\}$. 
\begin{defn}\label{defn:tilde Psi}
For $X(z)=X_{1\ldots k}(z_1\ldots z_k)\in \CV'$ we define the map:
\begin{align}\label{eq:Psi tilde}
    \widetilde \Psi\left [ X(z)\right]=  \pi^{\otimes k}\left(
    \Tr_{k+1\ldots 2k}
\left[
\left(\check{R}^{(V'')}_{\omega^k}(q z,z)\right)
\left( \id_{V''}^{\otimes k} \otimes \left(
\iota'^{\otimes k} X(z)\pi'^{\otimes k}\right)\right)
\right]\right)\iota^{\otimes k}
\end{align}
where the traces are taken over $k$ copies of $V''$.
\end{defn}
In \eqref{eq:Psi tilde}  the non-trivial contribution from the trace only comes from the subspaces $V'$. An equivalent definition can be given by writing:
\begin{align}\label{eq:Psi tilde x}
    \widetilde \Psi\left [ X(z)\right]: =  \Tr_{k+1\ldots 2k}
\left[
\left(\pi^{\otimes k}\otimes \pi'^{\otimes k}\right)\left(\check{R}^{(V'')}_{\omega^k}(q z,z)\right)
\left(\iota^{\otimes k}\otimes \iota'^{\otimes k}\right)
X_{k+1\ldots 2k}(z)
\right]\in \CV
\end{align}
where the traces are taken over $k$ copies of  $V'$. The graphical formula of \eqref{eq:Psi tilde} is:
\begin{align}\label{eq:Psi a graphics}
\widetilde\Psi\left [ X(z)\right]
=
\sum_{\alpha,\beta\in \iota(\mathcal{I})^k}
\ket{\pi(\alpha)}\bra{\pi(\beta)}
\sum_{\gamma,\delta\in \iota'(\mathcal{I}')^k}
\begin{tikzpicture}[scale=0.5,baseline=(current  bounding  box.center)]
% crosses
\rcheck{3}{4}
\rcheck{2}{3}
\rcheck{4}{3}
\rcheck{3}{2}
\rcheck{1}{2}
\rcheck{0}{1}
\rcheck{2}{1}
\rcheck{1}{0}
\rcheck{5}{2}
\rcheck{4}{1}
\rcheck{6}{1}
\rcheck{5}{0}
\rcheck{3}{0}
\rcheck{2}{-1}
\rcheck{4}{-1}
\rcheck{3}{-2}
% vertical lines
\draw[arrow=0.166] (0,5) -- (0,2);
\draw[arrow=0.25] (1,5) -- (1,3);
\draw[arrow=0.5] (2,5) -- (2,4);
\draw[arrow=0.3] (3,5) -- (4,4);
\draw[arrow=0.834] (0,1) -- (0,-2);
\draw[arrow=0.75] (1,0) -- (1,-2);
\draw[arrow=0.5] (2,-1) -- (2,-2);
\draw[invarrow=0.3] (3,-2) -- (4,-1);
% indices
\node[below] at (0,-2) {$\scriptstyle{\alpha_1}$};
\node[below] at (1,-2) {$\scriptstyle{\cdots}$};
\node[below] at (2,-2) {$\scriptstyle{\cdots}$};
\node[below] at (3,-2) {$\scriptstyle{\alpha_k}$};
\node[above] at (0,5) {$\scriptstyle{\beta_1}$};
\node[above] at (1,5) {$\scriptstyle{\cdots}$};
\node[above] at (2,5) {$\scriptstyle{\cdots}$};
\node[above] at (3,5) {$\scriptstyle{\beta_k}$};
% parameters
\node[below] at (0,-2-0.6) {$\scriptstyle{z_1}$};
\node[below] at (1,-2-0.6) {$\scriptstyle{\cdots}$};
\node[below] at (2,-2-0.6) {$\scriptstyle{\cdots}$};
\node[below] at (3,-2-0.6) {$\scriptstyle{z_k}$};
\node[above] at (0,5+0.6) {$\scriptstyle{q z_1}$};
\node[above] at (1,5+0.6) {$\scriptstyle{\cdots}$};
\node[above] at (2,5+0.6) {$\scriptstyle{\cdots}$};
\node[above] at (3,5+0.6) {$\scriptstyle{q z_k}$};
% X
\draw (8,1) -- (12,1) -- (12,2) -- (8,2) -- (8,1);
\node at (10,1.5) {$\scriptstyle \iota'^{\otimes k} X(z)\pi'^{\otimes k}$};
\draw (8.5,2) -- (8.5,2.5) node[above] {$\scriptstyle \gamma_1$};
\draw (9.5,2) -- (9.5,2.5) node[above] {$\scriptstyle \ldots$};
\draw (10.5,2) -- (10.5,2.5) node[above] {$\scriptstyle \ldots$};
\draw (11.5,2) -- (11.5,2.5) node[above] {$\scriptstyle \gamma_k$};
\draw (8.5,1) -- (8.5,0.5) node[below] {$\scriptstyle \delta_1$};
\draw (9.5,1) -- (9.5,0.5) node[below] {$\scriptstyle \ldots$};
\draw (10.5,1) -- (10.5,0.5) node[below] {$\scriptstyle \ldots$};
\draw (11.5,1) -- (11.5,0.5) node[below] {$\scriptstyle \delta_k$};
% labels of trace sums
\node[above, right] at (3.5, 5.5) {$\scriptstyle \gamma_k$};
\node[above, right] at (4.5, 4.5) {$\scriptstyle \ldots$};
\node[above, right] at (5.5, 3.5) {$\scriptstyle \ldots$};
\node[above, right] at (6.5, 2.5) {$\scriptstyle \gamma_1$};
\node[below, right] at (6.5, 0.5) {$\scriptstyle \delta_1$};
\node[below, right] at (5.5, -0.5) {$\scriptstyle \ldots$};
\node[below, right] at (4.5, -1.5) {$\scriptstyle \ldots$};
\node[below, right] at (3.5, -2.5) {$\scriptstyle \delta_k$};
\end{tikzpicture}
\end{align}
where each cross is interpreted as $\check R^{(V'')}(z_i/(qz_j))$. In this diagram we removed the trace lines and instead wrote the summations explicitly in order to emphasize that the trace is taken over the non-trivial subspace of $\iota'(V')$. When we choose projections in \eqref{eq:proj-emb}, we will always require that $\iota(\mathcal{I})\cup \iota'(\mathcal{I}')=\mathcal{I}''$, otherwise the space $V''$ will contain extra dimensions which will play no role in $\widetilde\Psi\left [ X(z)\right]$.

\begin{rmk}\label{rmk:Psi pi map}
Consider the map $\widetilde \Psi$ from Definition \ref{defn:tilde Psi} applied to $X\in\CV''$ (i.e. take $V'= V''$ in \eqref{eq:proj-emb}). Let $\CA'$ be the shuffle algebra from Definition \ref{def:A prime} (but associated to $\fgl_{n''|m''}$) and $\CA^+$ be the shuffle algebra from Definition \ref{def:wheel} (associated to $\fgl_{n|m}$). There exists an anti-homomorphism of shuffle algebras $\CA' \xrightarrow{\sim} \CA^+$ given  by:
\begin{align}\label{eq:psi a hom}
     X \mapsto \widetilde\Psi[X]
\end{align}
\end{rmk}
The statement holds because under the conditions of the Remark the map
\begin{align}\label{eq:Psi tilde hom}
    \widetilde \Psi\left [ X(z)\right]=  \pi^{\otimes k}\left(
    \Tr_{k+1\ldots 2k}
\left[
\left(\check{R}^{(V'')}_{\omega^k}(q z,z)\right)
\left( \id_{V''}^{\otimes k} \otimes 
X(z)\right)
\right]\right)\iota^{\otimes k}
\end{align}
is a composition of the anti-isomorphism $\Psi$ (acting in $\CV''$) and the projection which is a homomorphism due to \eqref{eq:proj hom}. However, if we take $\CV'$ different than $\CV''$ in Remark \ref{rmk:Psi pi map} then $\widetilde\Psi$ fails to be an anti-homomorphism. This happens because the inner part of the map $\widetilde \Psi$ in \eqref{eq:Psi tilde} is given by $\iota'^{\otimes^k}\bullet \pi'^{\otimes^k}$ which does not respect the shuffle algebras' multiplication. As will be explained in the next section there exist exceptions $X,Y\in \CV'$ for which:
\[
\widetilde \Psi[X*'Y] =\widetilde \Psi[X]*\widetilde \Psi[Y]
\]

%%% ------ %%% ------ %%% ------ %%% ------ %%% ------ 
%%% ------ %%% ------ %%% ------ %%% ------ %%% ------ 
%%% ------ %%% ------ %%% ------ %%% ------ %%% ------ 

\section{The commutative subalgebra}\label{sec:Commutative}

This section is devoted to the commutative subalgebra of the shuffle algebra $\CA^+$. Our main goal is to derive trace formulas for several families of commuting elements using the anti-isomorphism $\Psi$ and the map $\widetilde \Psi$. As a consequence we prove Theorem \ref{thm:Z_intro} and equation \eqref{eq:S-trace intro} from the Introduction.

\subsection{Slope 0}
\label{sub:slope 0}

The shuffle algebra $\CA^+$ comes endowed with two gradings, which we will refer to as vertical and horizontal. In order to define them, we recall the fact that any element of $\CA^+$ is a linear combination of tensors:
\begin{equation}
\label{eqn:a tensor}
f(z_1\ldots z_k) E_{i_1j_1} \otimes \dots \otimes E_{i_kj_k}
\end{equation}
for various rational functions $f$ and indices $i_1\ldots i_k,j_1\ldots j_k \in \{1\ldots n\}$. We then define the \emph{vertical} and \emph{horizontal} gradings by the following formulas:
\begin{align}
&\vdeg X = k \in \BN \label{eqn:vertical} \\
&\hdeg X = (\text{hom deg }f)\bde - \sum_{s=1}^k [i_s;j_s) \label{eqn:horizontal} \in \zz 
\end{align}
where hom deg denotes the homogeneous degree of a rational function, $\bde = (1\ldots 1) \in \zz$, and
$$
[i;j) = \bs^i+\dots+\bs^{j-1} \in \zz
$$
with $\bs^k$ being the $n$-tuple of integers with a single 1 on position $\equiv k$ modulo $n$, and 0 everywhere else. While the definition above makes sense only for $i \leq j$, we extend it to all integers $i$ and $j$ by setting $[i;j) + [j;i) = 0$. We will write
$$
\CA^+ = \bigoplus_{k=0}^{\infty} \CA_k, \qquad \CA_k = \bigoplus_{\bd \in \zz} \CA_{\bd,k}
$$
for the graded pieces of the shuffle algebra, and we henceforth focus on the $\bd = 0$ direct summand.

\begin{defn}
\label{def:slope}

An element of $\CA_{0,k}$ is said to have slope 0 if it is a linear combination of tensors of the form \eqref{eqn:a tensor} such that
\begin{equation}
\label{eqn:inequality}
\deg_{\xi} f(z_1\ldots z_{\ell}, \xi z_{\ell+1}, \dots, \xi z_k) \cdot n - \sum_{a=\ell+1}^k (j_a-i_a) \leq 0
\end{equation}
for all $0 \leq \ell \leq k$, where $\deg_{\xi}$ refers to the order of a rational function in $\xi$ as $\xi \rightarrow \infty$. 
    
\end{defn}

Let $\CB_{k} \subset \CA_{0,k}$ be the subspace of tensors satisfying the property in Definition \ref{def:slope}. For all $0 \leq \ell \leq k$, we may define the operation
\begin{equation}
\label{eqn:lead}
\text{lead}_{\ell} : \CB_{k} \rightarrow \CB_{\ell} \otimes \CB_{k-\ell}
\end{equation}
by the following procedure: first write any element $\CB_{k}$ as a linear combination of tensors \eqref{eqn:a tensor}. Then change the variables of $f$ according to $z_{\ell+1}\mapsto \xi z_{\ell+1}, \dots, z_k \mapsto \xi z_k$ and retain only the coefficient of
$$
\xi^{\frac 1n\sum_{a=l+1}^k (j_a-i_a)}
$$
i.e. the greatest possible power of $\xi$ allowed by inequality \eqref{eqn:inequality}. The resulting expression in variables $z_1\ldots z_k$ can be interpreted as an element of $\CB_{\ell} \otimes \CB_{k-\ell}$, where the variables $z_1\ldots z_{\ell}$ (respectively $z_{\ell+1}\ldots z_k$) correspond to the first (respectively second) tensor factor. 

It was shown in \cite[Proposition 5.8]{Ng_tale} that 
\begin{equation}
\label{eqn:slope subalgebra}
\CB^+ = \bigoplus_{k=0}^{\infty} \CB_{k}
\end{equation}
is preserved by the shuffle product $*$. The following result shows that $\CB^+$ is commutative.

\begin{prop}
\label{prop:slope}

(\!\!\cite[Proposition 5.23]{Ng_tale}) There exists an algebra isomorphism 
\begin{equation}
\label{eqn:upsilon}
\Xi : \Lambda^{\otimes n} \xrightarrow{\sim} \CB^+
\end{equation}
where $\Lambda = \BC(q,t)[p_1,p_2,\dots]$ is the ring of symmetric polynomials in infinitely many variables.

\end{prop}

\subsection{Explicitly unraveling $\Xi$}

The isomorphism \eqref{eqn:upsilon} was defined in \cite{Ng_tale} by placing bialgebra structures on (enlargements of) both $\Lambda^{\otimes n}$ and $\CB^+$ and then constructing explicit elements of the latter algebra which correspond under $\Xi$ to complete symmetric functions. However, it is more natural to use power sum functions instead. To this end, consider the element
\begin{equation}
\label{eqn:power sum}
P_k^{(i)} \in \CB_{k}
\end{equation}
that corresponds to the power-sum function $p_k$ in the $i$-th tensor factor under the isomorphism $\Xi$. Explicitly, the way these elements were constructed in \cite{Ng_tale} is by requiring that
\begin{equation}
\label{eqn:property 1}
P_k^{(i)} \text{ has the property that the inequality \eqref{eqn:inequality} is strict}, \ \forall \ell \in \{1\ldots k-1\}
\end{equation}
The vector space of elements of $\CB_{k}$ with the property in \eqref{eqn:property 1} is $n$-dimensional, and thus one has the freedom to transform the vector 
$$
\left( P_k^{(1)} \dots P_k^{(n)} \right) \in \CB_{k}^{\oplus n}
$$
since it is only determined up to a linear transformation in $GL_n$. To completely determine the vector above, we consider the following. 

\begin{defn}
\label{def:evaluation}

For any $i \in \{1\ldots n\}$, consider the linear map
$$
\alpha_i : \CB^+ \rightarrow \BC(q,t)
$$
$$
\alpha_i(X) = \Big( \text{coefficient of } E_{ii} \text{ in } X^{(k)}(y) \text{ of \eqref{eqn:residue (k)}} \Big)
$$
for all $X \in \CB_{k}$ (note that $\alpha_i(X) = \alpha_{[i;i)}(X) \cdot (1-t^{-2})^k q^{\frac {k(2i-1)}n}$ in the notation of \cite{Ng_tale}).

\end{defn}

As shown in \cite[Corollary 5.13]{Ng_tale}, the map $\alpha_i$ is an algebra homomorphism for all $i \in \{1\ldots n\}$. We completely determine the elements \eqref{eqn:power sum} by property \eqref{eqn:property 1} together with the following normalization
\begin{equation}
\label{eqn:property 2}
\alpha_j \left( P_k^{(i)} \right) = \delta_{ij}, \quad \forall  i,j \in \{1,\dots,n\}
\end{equation}
The following result is key to constructing interesting elements in the commutative subalgebra $\CB^+$.

\begin{prop}
\label{prop:exp}

Suppose we have elements $\{H_k \in \CB_{k}\}_{k \geq 0}$ such that $H_0 = 1$ and
$$
\emph{lead}_{\ell} \left( H_k  \right) = H_{\ell}\otimes H_{k-\ell}
$$
(recall the notation \eqref{eqn:lead}) for all $\ell \in \{0\ldots k\}$. Then we have the following equality of generating series
\begin{equation}
\label{eqn:exponential}
\sum_{k=0}^{\infty} H_{k} x^k = \exp \left( \sum_{k=1}^{\infty} \sum_{i =1}^n \frac {\gamma^k_{(i)}}k \cdot P_k^{(i)}x^k \right)
\end{equation}
for some constants $\{\gamma^k_{(i)}\}_{k \in \BN, i \in \{1\ldots n\}} \in \BC(q,t)$.

\end{prop}

Proposition \ref{prop:slope} allows us to reduce Proposition \ref{prop:exp} to the analogous statement for $\Lambda^{\otimes n}$ (with the analogue of $\sum_{\ell=0}^k \text{lead}_{\ell}$ being the Hall coproduct on $\Lambda$) in which case the result is well known. Moreover, because the maps $\alpha_i$ are algebra homomorphisms, the constants $\gamma^k_{(i)}$ can be determined from the equalities
\begin{equation}
\label{eqn:exponential evaluation}
\sum_{k=0}^{\infty} \alpha_i(H_{k}) x^k = \exp \left( \sum_{k=1}^{\infty}  \frac {\gamma^k_{(i)}}k  x^k \right), \quad \forall i \in \{1\ldots n\}
\end{equation}
We make the convention that $\alpha_i(H_0) = 1$ for all $i$.

\subsection{Elements of $\CB^+$ part I}\label{sec:basic elements} In the present paper, a key role is played by the following elements:
\begin{equation}
\label{eqn:def h}
H_k^{(i)} = \underbrace{E_{ii} \otimes \dots \otimes E_{ii}}_{k\text{ factors}} \in \CV
\end{equation}
Define their generating series:
\begin{align}\label{eq:H-generating}
    H^{(i)}(x) := \sum_{k=0}^{\infty} H_{k}^{(i)} x^k 
\end{align}

\begin{prop}
\label{prop:computation}

We have the following equality of generating series:
\begin{equation}
\label{eqn:H generator}
H^{(i)}(x) = \exp \left( \sum_{k=1}^{\infty} \frac {P_k^{(i)} (-1)^{k-1}}k x^k \right)
\end{equation}
for all $i \in \{1,\dots,n\}$.

\end{prop}

\begin{proof} Since it is easy to show that $H_k^{(i)}$ commutes with $\check R_{i,j}$ for all $1 \leq i\neq j \leq k$, then $H_k^{(i)} \in \CV_{\text{sym}}$. Since $E_{ii} \otimes \dots \otimes E_{ii}$ has no denominators, its residue at $z_i = z_jq$ is equal to 0 for all $k \geq 2$ for trivial reasons, so we have
$$
H_k^{(i)} \in \CA^+
$$
Clearly, $H_k^{(i)}$ has vertical degree $k$ and horizontal degree 0. Moreover, it has slope 0 because the left-hand side of \eqref{eqn:inequality} for $E_{ii} \otimes \dots \otimes E_{ii}$ is equal to 0 for all $ \ell \in \{0\ldots k\}$. This implies not only that
$$
H_k^{(i)} \in \CB^+
$$
but that the elements $\{H_k^{(i)}\}_{k \geq 0}$ satisfy the assumptions of Proposition \ref{prop:exp}. Therefore, we have a power series equality of the form \eqref{eqn:exponential}, and it remains to compute its coefficients. Because 
$$
\left(H_1^{(i)}\right)^{(1)} = E_{ii}, \qquad \left(H_k^{(i)}\right)^{(k)} = 0, \ \forall  k \geq 2
$$
(the latter property follows because the residues \eqref{eqn:residue} are 0 for $E_{ii} \otimes \dots \otimes E_{ii}$) then we have 
$$
\alpha_j(H_k^{(i)}) = \delta_{k0} + \delta_{k1}\delta_{ij}
$$
Plugging this into formula \eqref{eqn:exponential evaluation} yields
$$
1+x \delta_{ij} = \exp \left( \sum_{k=1}^{\infty}  \frac {\gamma^k_{(j)}}k x^k \right), \quad \forall j \in \{1,\dots,n\}
$$
Taking the logarithm implies that $\gamma^k_{(j)} = \delta_{ij} (-1)^{k-1}$, which implies formula \eqref{eqn:H generator}.

\end{proof}
Let $\lambda=(\lambda_1 \ldots \lambda_N)\in \{1\ldots n\}^N$ be any sequence and $m(\lambda)=(m_1(\lambda)\ldots m_n(\lambda))$ be the multiplicity vector of $\lambda$, i.e. $m_i(\lambda)$ denotes the multiplicity of $i$ in $\lambda$.
\begin{lem}\label{lem:H shuffle product}
Let $k_1\ldots k_n\geq 0$ such that $k_1+\cdots +k_n=N$. We have the following identity:
\begin{align}
\label{eqn:align}
H_{k_1}^{(1)}* \cdots * H_{k_n}^{(n)} = \sum_{\substack{\lambda\in \{1\ldots n\}^N\\ m(\lambda)=(k_1 \ldots k_n)}} E_{\lambda_1\lambda_1}\otimes \cdots \otimes E_{\lambda_N\lambda_N}
\end{align}
\end{lem}

\begin{proof}

We will follow the proof of Proposition \ref{prop:computation} in order to prove \eqref{eqn:align} by induction on $k_1+\dots+k_n = N$. To this end, we must first show that the left and right-hand sides of \eqref{eqn:align} have the same value of $\text{lead}_{M}$ for any $M \in \{1,\dots,N-1\}$. Indeed, $\text{lead}_{M}(\text{LHS})$ is equal to
$$
\sum_{\ell_1+\dots+\ell_n = M} \Big( H_{\ell_1}^{(1)} * \dots * H_{\ell_n}^{(n)} \Big) \otimes \Big( H_{k_1 - \ell_1}^{(1)} * \dots * H_{k_n - \ell_n}^{(n)} \Big)
$$
(this uses the multiplicativity of $\sum_\ell \text{lead}_{M}$, which was established in \cite{Ng_tale}), while $\text{lead}_{\ell}(\text{RHS})$ is equal to
$$
\sum_{\substack{\lambda\in \{1\ldots n\}^N\\ m(\lambda)=(k_1 \ldots k_n)}} \Big( E_{\lambda_1\lambda_1}\otimes \cdots \otimes E_{\lambda_M\lambda_M} \Big) \otimes \Big( E_{\lambda_{M+1}\lambda_{M+1}}\otimes \cdots \otimes E_{\lambda_N\lambda_N} \Big)
$$
The two displays above are equal for all $M \in \{1,\dots,N-1\}$ by the induction hypothesis. Finally, we must show that the LHS and the RHS have the same values under the linear maps $\alpha_i$ of Definition \ref{def:evaluation}. Since these linear maps are multiplicative, we have
$$
\alpha_i(\text{LHS}) = \alpha_i(H_{k_1}^{(1)}) \cdots \alpha_i(H_{k_n}^{(n)}) = \begin{cases} 1 & \text{if }(k_1,\dots, k_i,\dots,k_n) = (0,\dots,0,\dots,0) \\ 1 & \text{if }(k_1,\dots,k_i,\dots,k_n) = (0,\dots,1,\dots,0) \\ 0 &\text{otherwise } \end{cases}
$$
Meanwhile, because the RHS of \eqref{eqn:align} has no poles at $z_i - z_jq$, then $\alpha_i(\text{RHS})$ is non-zero only if $N=0$ or if $N=1$ and $\lambda_1 = i$. In either case, we see that the left and right-hand sides of equation \eqref{eqn:align} have the same value under the linear maps $\alpha_i$, thus implying that they are equal. 

\end{proof}

\subsection{Elements of $\CB^+$ part II} We will now give a new description of the generators $P_k^{(i)}$ of $\CB^+$.

\begin{prop}
\label{prop:comm}

The elements $P_k^{(i)} \in \CB_k \subset \CA_k$ are completely determined by the linear equations
\begin{equation}
\label{eqn:linear}
\begin{cases} P_k^{(1)} = q^k S_k^{(n)}-S_k^{(1)} \\ 
P_k^{(2)} = S_k^{(1)}-S_k^{(2)} \\
\dots \\ 
P_k^{(n)} = S_k^{(n-1)}-S_k^{(n)}\end{cases}
\end{equation}
where $\{S_k^{(i)} \in \CA_k \subset \emph{End}(V^{\otimes k})(z_1\ldots z_k)\}_{k \in \BN, i \in \{1\ldots n\}}$ are recursively determined by the properties
\begin{equation}
\label{eqn:power comm}
(z_1+\dots+z_k) S_k^{(i)} = \left[S_{k-\ell}^{(i)},(z_1+\dots+z_{\ell}) S_{\ell}^{(i)} \right]
\end{equation}
for all $i \in \{1\ldots n\}$ and $k > \ell \geq 1$, as well as the initial conditions 
\begin{align}\label{eqn: power sum 1}
    S_{1}^{(i)} =\frac{-1}{1-q}\sum_{j=1}^{n} q^{\delta_{j>i}}E_{jj}
\end{align}
(which is simply restating \eqref{eqn:linear} for $k=1$, given that $P_1^{(i)} = E_{ii}$).
\end{prop}

\begin{proof}

First of all, we observe that the operation
$$
\delta : \CA^+ \rightarrow \CA^+, \quad \delta(X(z_1\ldots z_k)) = (z_1+\dots+z_k) X(z_1\ldots z_k)
$$
is a derivation with respect to the shuffle product, i.e.
$$
\delta(X*X') = \delta(X) * X' + X * \delta(X')
$$
However, commutators are also derivations, so we claim that for all $X \in \CA^+$ we have
\begin{equation}
\label{eqn:delta is comm}
\delta(X) = [X,p_1]
\end{equation}
where $p_1$ is an element of $\uui \subset \uu$, embedded in $\CA$ by the middle tensor factor in \eqref{eqn:double a}. To prove \eqref{eqn:delta is comm}, one only needs to check this formula on the generators $E_{ij} \in \CA^+$ (see Proposition \ref{prop:gen}), in which case it is a direct application of \cite[formulas (3.77) and (6.36)]{Ng_tale}. Note that for the latter statement to be true on the nose, we need to renormalize our generators according to
$$
\Big( p_1 \text{ of \eqref{eqn:gl 1}} \Big) = \left( \frac {p_{\bde,1}^{(0)}}{q^{-1}-q} \text{ of \emph{loc. cit.}} \right)
$$
With formula \eqref{eqn:delta is comm} in mind, formula \eqref{eqn:power comm} becomes equivalent to the following equality
\begin{equation}
\label{eqn:power comm proof}
\left[ S_k^{(i)}, p_1 \right] = \left[S_{k-\ell}^{(i)}, [S_\ell^{(i)}, p_1] \right]
\end{equation} 
in $\CA$. To prove \eqref{eqn:power comm proof}, it is enough to prove it under the isomorphism
$$
\Upsilon \circ \Phi^{-1} : \CA \rightarrow \CS
$$
with respect to \eqref{eqn:Phi} and \eqref{eqn:Upsilon}. It was shown in \cite{Ng_tor_sh} that, under the isomorphism $\Upsilon$, $p_1$ corresponds to the rational function
\begin{equation}
\label{eqn:shuffle p1}
f(x_{11}\ldots x_{n1}) = \frac {\gamma \cdot x_{11}\dots x_{n1}}{(x_{11}t^{-\frac 12} - x_{21}t^{\frac 12}) \dots (x_{n-1,1}t^{-\frac 12} - x_{n1}t^{\frac 12})(x_{n1}t^{-\frac 12} - x_{11}t^{\frac 12}q^{-1})}
\end{equation}
in the $\bde = (1\ldots 1)$ direct summand of \eqref{eqn:classic shuffle}. The constant $\gamma \in \BC(q,t)^*$ is not important for us, as one can always compose the isomorphism $\Upsilon$ by an appropriate renormalization of the generators of $\UU$. Moreover, in $\CS$ the following equalities were proved in \cite[formulas (3.23) and (3.105)]{Ng_pbw}
\begin{equation}
\label{eqn:comm old 1}
\begin{split}
&\left[P_k^{(1)}, R \right] = \left( q^k\Big( x_{n1}^k +\dots+q^k x_{nd_n}^k \Big)- \Big(x_{11}^k+\dots+x_{1d_1}^k\Big) \right) R \\
&\left[P_k^{(2)}, R \right] = \left(\Big(x_{11}^k +\dots+x_{1d_1}^k \Big) - \Big( x_{21}^k+\dots+x_{2d_2}^k \Big)\right)R \\
&\dots \\
&\left[P_k^{(n)}, R \right] = \left(\Big(x_{n-1,1}^k +\dots+x_{n-1,d_{n-1}}^k \Big) - \Big( x_{d1}^k+\dots+x_{nd_n}^k \Big)\right)R
\end{split}
\end{equation}
for all $R = R(x_{i1}\ldots x_{id_i})_{i \in \{1,\dots,n\}} \in \CS^+$. The linear combinations \eqref{eqn:linear} were chosen so that 
\begin{equation}
\label{eqn:comm old 2}
\left[S_k^{(i)}, R \right] = \left( x_{i1}^k+\dots+x_{id_i}^k\right)R
\end{equation}
for all $i \in \{1\ldots n\}$. Thus, formula \eqref{eqn:power comm proof} is a consequence of the fact that the function $f = f(x_{11}\ldots x_{n1})$ of \eqref{eqn:shuffle p1} (as does any function in the $(1\ldots 1)$ direct summand of \eqref{eqn:classic shuffle}) satisfies
\begin{equation}
\label{eqn:trivial}
x_{i1}^k \cdot f = x_{i1}^{k-\ell} \cdot x_{i1}^{\ell} \cdot f
\end{equation}
Thus, the trivial equality \eqref{eqn:trivial} in $\CS$ gives rise to the quite non-trivial equality \eqref{eqn:power comm proof} in $\CA$.

\end{proof}

\begin{rmk}\label{rmk:H shuffle product in A prime}
    Let $\mathcal{B}'\subset \mathcal{A}'$ be the slope $0$ subalgebra of $\mathcal{A}'$. 
    We can define the elements: 
    \begin{equation}
    \label{eqn:def h prime}
    H_k'^{(i)} = \underbrace{E_{ii} \otimes \dots \otimes E_{ii}}_{k\text{ factors}} \in \mathcal{B}'
    \end{equation}
Let $k_1\ldots k_n\geq 0$ such that $k_1+\cdots +k_n=N$. 
By a similar logic to that of Lemma \ref{lem:H shuffle product}, we have 
    \begin{align}
        \label{eq:h in A prime}
        H_{k_1}'^{(1)}*' \cdots *' H_{k_n}'^{(n)} = \sum_{\substack{\lambda\in \{1\ldots n\}^N\\ m(\lambda)=(k_1 \ldots k_n)}} E_{\lambda_1\lambda_1}\otimes \cdots \otimes E_{\lambda_N\lambda_N} \in \mathcal{B}'
    \end{align}
\end{rmk}

\subsection{The super case}
\label{sub:commutative super}

As we explained in Subsection \ref{sub:super}, the $R$-matrix \eqref{eq:R-matrix super} is so similar to \eqref{eq:R-matrix} that many of the definitions and basic properties of shuffle algebras carry through. In fact, the only place where the $\fgl_{n|m}$ case exhibits a difference is in the fact that the $R$-matrix \eqref{eq:R-matrix super} has the property that
\begin{align}
&\lim_{z \rightarrow \infty}  R(z) = \sum_{1\leq i,j\leq n+m}
\left( \epsilon_i t^{\epsilon_i /2} \right)^{\delta_{i=j}}E_{ii}\otimes E_{jj}  
  +
  \sum_{1\leq i\neq j\leq n+m}
   t^{1/2} E_{ij}\otimes E_{ji} \label{eq:limit infinity} \\
&\lim_{z \rightarrow 0}  R(z) = \sum_{1\leq i,j\leq n+m}
\left( \epsilon_i t^{-\epsilon_i /2} \right)^{\delta_{i=j}}E_{ii}\otimes E_{jj}  
  +
  \sum_{1\leq i\neq j\leq n+m}
 t^{-1/2} E_{ij}\otimes E_{ji} \label{eq:limit infinity}
\end{align}
whereas the respective limits for the $R$-matrices \eqref{eq:R-matrix} only involved those $i$'s with $\epsilon_i = 1$. The fact that the ``diagonal'' terms in the limits above depend on $\epsilon_i$ is precisely accounted for in the commutation properties of the Cartan subalgebra of $\fgl_{n|m}$ (specifically, formula (5.8) of \cite{Ng_tale} would need to be adapted in order to capture the dependence on $\epsilon_i$). Therefore, the definition of the slope 0 subalgebra $\CB^+$ in Subsection \ref{sub:slope 0} runs through without modification, and the following analogue of Proposition \ref{prop:slope} holds:
\begin{equation}
\label{eqn:upsilon super}
\Xi : \Lambda^{\otimes n+m} \xrightarrow{\sim} \CB^+
\end{equation}
Moreover, the natural analogues of Proposition \ref{prop:exp} and \ref{prop:computation} also hold: the elements
$$
H_k^{(i)} = \underbrace{E_{ii} \otimes \dots \otimes E_{ii}}_{k \text{ factors}}, \qquad \forall i \in \{1,\dots,n+m\}, k > 0
$$
lie in $\CB^+$, and have the property that 
$$
H^{(i)}(x) = \sum_{k=0}^{\infty} H_k^{(i)} x^k = \exp \left( \sum_{k=1}^{\infty} \frac {P_k^{(i)} (-1)^{k-1}}k x^k \right)
$$
for certain ``power-sum'' generators $\{P_k^{(i)}\}^{1 \leq i \leq n+m}_{k > 0}$. However, we do not have a proof of Proposition \ref{prop:comm} in the super case, since it relies on the conjectural isomorphism \eqref{eqn:conj super}. Thus, while the features of the commutative subalgebra $\CB^+ \cong \Lambda^{\otimes (n+m)}$ run completely parallel in the $\fgl_{n|m}$ as in the $\fgl_n$ case, those properties which require us to understand the entire matrix shuffle algebra (such as \eqref{eqn:conj super}) are still conjectural in the $\fgl_{n|m}$ case.

\subsection{Computing new elements of $\mathcal{B}^+$  using the anti-isomorphism $\Psi$}
The basic elements of $\mathcal{B}^+$, which were discussed in Section \ref{sec:basic elements}, have counterparts in $\mathcal{B}'$. One such example, given in Remark \ref{rmk:H shuffle product in A prime}, is the family $H'^{(j)}_k$. In this section we apply the isomorphism $\Psi : \CA' \rightarrow \CA^+$ from Section \ref{sec:Psi} to the elements $H'^{(j)}_k$ in order to describe a new family of commuting elements of $\mathcal{B}^+$. As a consequence we prove Theorem \ref{thm:Z_intro}.

\begin{thm}\label{thm:Psi H}
Fix $j\in\{1\ldots n+m\}$. The action of $\Psi$ on the generating function $H'^{(j)}(x)$, defined by:
\begin{align}
\Psi\left[H'^{(j)}(x)\right]=
\sum_{k=0}^{\infty} \Psi\left[H'^{(j)}_k\right]x^k
\end{align}
is given by the formula:
\begin{align}\label{eq:Psi exp P}
    \Psi\left[H'^{(j)}(x)\right] =
    \exp \left(\sum_{k=1}^{\infty}\sum_{i=1}^{n+m} \frac{\epsilon_j^k}{k}\frac{t^{\epsilon_jk/2}-q^{\delta_{i=j}k}t^{-\epsilon_jk/2}}{1-q^k}q^{k\delta_{i>j}} P_k^{(i)}x^k \right)
\end{align}
\end{thm}
\begin{proof}
Since $\Psi$ is an isomorphism, it satisfies the conditions of Proposition \ref{prop:exp} and we must have:
\begin{equation}
\label{eqn:Psi exponential}
\Psi\left[H'^{(j)}(x)\right] =
\sum_{k=0}^{\infty} 
\Psi\left[E_{jj}^{\otimes k}\right]x^k =
\exp \left( \sum_{k=1}^{\infty} \sum_{i =1}^{n+m} \frac {\gamma^k_{(i)}}k \cdot P_k^{(i)}x^k \right)
\end{equation}
for some coefficients $\gamma^k_{(i)}$. We can compute these coefficients using evaluations as in \eqref{eqn:exponential evaluation}:
\begin{equation}
\label{eqn:Psi exponential evaluation}
\sum_{k=0}^{\infty} 
\alpha_i\left(\Psi\left[E_{jj}^{\otimes k}\right]\right) x^k = \exp \left( \sum_{k=1}^{\infty}  \frac {\gamma^k_{(i)}}k  x^k \right), \quad \forall i \in \{1\ldots n+m\}
\end{equation}
The evaluations $\alpha_i \left( \Psi\left[E_{jj}^{\otimes k}\right]\right)$ are computed in Lemma \ref{lem:alpha Psi} in Appendix \ref{app:evaluations Psi}:
\begin{align*}
    \alpha_i \left( \Psi\left[E_{jj}^{\otimes k}\right]\right) = 
     \begin{cases}
      \prod_{l=1}^{k}\epsilon_j\frac{ t^{\epsilon_j /2}-q^{l}t^{-\epsilon_j /2}}{1-q^{l}}  
      & i=j\\
q^{k\delta_{i>j}}\prod_{l=1}^{k}\epsilon_j\frac{ t^{\epsilon_j /2}-q^{l-1}t^{-\epsilon_j /2}}{1-q^{l}}
    & i\neq j
     \end{cases}
\end{align*}
For $i=j$ \eqref{eqn:Psi exponential evaluation} becomes:
\begin{align}\label{eq:alpha psi ii}
    \sum_{k=0}^{\infty} 
\alpha_j\left(\Psi\left[E_{jj}^{\otimes k}\right]\right) x^k &=
 \sum_{k=0}^{\infty} x^k
            \prod_{l=1}^{k}\epsilon_j\frac{ t^{\epsilon_j /2}-q^{l}t^{-\epsilon_j /2}}{1-q^{l}}   \\&=
      \exp\left(\sum_{k=1}^\infty \frac{\epsilon_j^k}{k}\frac{t^{\epsilon_j k/2}-q^{k}t^{-\epsilon_j k/2}}{1-q^{k}}x^k\right)\nonumber
\end{align}
and for $i\neq j$:
\begin{align}\label{eq:alpha psi jj}
    \sum_{k=0}^{\infty} 
\alpha_i\left(\Psi\left[E_{jj}^{\otimes k}\right]\right) x^k &=
\sum_{k=0}^{\infty} x^k q^{k\delta_{i>j}}
            \prod_{l=1}^{k}\epsilon_j\frac{ t^{\epsilon_j /2}-q^{l-1}t^{-\epsilon_j /2}}{1-q^{l}}     \\
            &=
      \exp\left(\sum_{k=1}^\infty \frac{\epsilon_j^k}{k}\frac{t^{\epsilon_jk/2}-t^{-\epsilon_jk/2}}{1-q^{k}}q^{k\delta_{i>j}}x^k\right)      \nonumber
\end{align}
where we computed the infinite sums in \eqref{eq:alpha psi ii} and \eqref{eq:alpha psi jj} using the $q$-binomial theorem. 
With these two formulas we find all $\gamma_{(i)}^k$  from \eqref{eqn:Psi exponential} which gives us \eqref{eq:Psi exp P}. 
\end{proof}
\begin{rmk}\label{rmk:Psi H in S}
    We can express the exponential formula \eqref{eq:Psi exp P} in terms of the generators $S_k^{(j)}$:
\begin{align}
    \label{eq:Psi exp S}
    \Psi\left[H'^{(j)}(x)\right] = \exp \left(\sum_{k=1}^\infty \frac{x^k}{k}\epsilon_j^k\left(t^{-\epsilon_jk/2}S_{k}^{(j-1)}-t^{\epsilon_jk/2}S_{k}^{(j)}\right)\right)
\end{align}
where $S_k^{(0)}:=q^k S_k^{(n+m)}$.
\begin{proof}
    The expression in \eqref{eq:Psi exp S} follows from the substitution \eqref{eqn:linear} into \eqref{eq:Psi exp P} and a further simplification.
\end{proof}
\end{rmk}
\begin{cor}[Theorem \ref{thm:Z_intro}]\label{cor:Zexp}
The generating function $Z(v)$ can be expressed as:
    \begin{align}
    \label{eq:Zexp text}
        Z(v) = \exp \left(\sum_{k>0}\frac{v^k}{k}
        \sum_{i=1}^{n+m}  \left(y_{i+1}^k - t^{\epsilon_{i} k}y_{i}^k \right)S_k^{(i)} \right)
    \end{align}    
where $y_i:=t^{-\epsilon_i/2}\epsilon_i u_i$ and $u_{n+m+1}:=q u_1$.
\end{cor}
\begin{proof}
Let us recall the expression \eqref{eq:Z-R introduction} for $Z_N$ from the introduction:
\begin{align}\label{eq:Z-R}
   Z_N(u_1\ldots u_{n+m}) = \Tr_{N+1\ldots 2N} 
   \left[
   \prod_{j=1}^N\prod_{i=1}^N \check{R}_{N+j-i}\left(\frac{z_{N-i+1}}{q z_j}\right)\cdot
\text{id}^{\otimes N}\otimes \hat u^{\otimes N}
\right]
\end{align}
where we made explicit the dependence of $Z_N$ on the parameters $u_i$ and we recall $\hat u$:
\begin{align}\label{eq:u-hat}
    \hat u := \sum_{i=1}^{n+m}u_i E_{ii}
\end{align}
The expression $\hat u^{\otimes N}$ can be rewritten as follows: 
\begin{align}\label{eq:u-sum}
    \hat u^{\otimes N} = \sum_{\substack{\kappa=(\kappa_1\ldots \kappa_{n+m}):\\|\kappa|=N}}u_1^{\kappa_1}\cdots u_{n+m}^{\kappa_{n+m}}
    \sum_{\lambda:m(\lambda)=\kappa}  E_{\lambda_1\lambda_1}\otimes\cdots \otimes E_{\lambda_N\lambda_N}
\end{align}
By Remark \ref{rmk:H shuffle product in A prime}, we can view the coefficients of the monomials in $u_i$ in the above formula as elements of $\mathcal{B}'$. 
Hence we can apply the anti-isomorphism $\Psi$ to each of these coefficients. By comparing the definition of $\Psi$ in \eqref{eq:Psi} with $Z_N$ in \eqref{eq:Z-R} we conclude that:
\begin{align}\label{eq:Psi of u}
Z_N(u_1\ldots u_{n+m})=\Psi\left [\hat u^{\otimes N}\right] \in \mathcal{B}_0[[u_1\ldots u_{n+m}]]
\end{align}
Let us calculate the generating function $Z(v)$:
\begin{align}
    \label{eq:Z(v)}
    Z(v;u_1\ldots u_{n+m})= \sum_{N= 0}^\infty  Z_{N}(u_1\ldots u_{n+m})v^N
\end{align}
using \eqref{eq:Psi of u}. In doing so we will use the expansion \eqref{eq:u-sum}, rewrite the tensor products using \eqref{eq:h in A prime} viewing $H_k'^{(i)}\in \mathcal{B}'$, and using the fact that $\Psi$ is a homomorphism of commuting shuffle algebras:
\begin{align}\label{eq:Z(v) computation}
    Z(v;u_1\ldots u_{n+m})=\sum_{N=0}^{\infty}v^N
    \Psi[\hat u^{\otimes N}] &= \sum_{N=0}^{\infty}
    v^N \sum_{\substack{\kappa=(\kappa_1\ldots \kappa_{n+m}):\\|\kappa|=N}}
    u_1^{\kappa_1}\cdots u_{n+m}^{\kappa_{n+m}}
    \Psi\left[H_{\kappa_1}'^{(1)}*' \cdots *' H_{\kappa_n}'^{(n)}\right] \nonumber \\
    &= \sum_{N=0}^{\infty}v^N\sum_{\substack{\kappa=(\kappa_1\ldots \kappa_{n+m}):\\|\kappa|=N}}
    u_1^{\kappa_1}\cdots u_{n+m}^{\kappa_{n+m}}
    \Psi\left[H_{\kappa_1}'^{(1)}\right] * \cdots * \Psi \left[H_{\kappa_n}'^{(n)}\right]\nonumber \\
    & = \Psi\left[H'^{(1)}(v u_1)\right] * \cdots * \Psi\left[H'^{(n+m)}(v u_{n+m})\right]
\end{align}
With the help of Theorem \ref{thm:Psi H} we obtain the formula:
\begin{align}\label{eq:Z exp P}
    Z(v;u_1\ldots u_{n+m}) = 
    \exp \left(\sum_{k=1}
    ^{\infty}\frac{v^k}{k}\sum_{i,j=1}^{n+m} \epsilon_j^k\frac{t^{\epsilon_jk/2}-q^{\delta_{i=j}k}t^{-\epsilon_jk/2}}{1-q^k}q^{k\delta_{i>j}} P_k^{(i)}u_j^k \right)
\end{align}
The exponent simplifies if we express it in terms of the elements $S_k^{(j)}$ instead of $P_k^{(i)}$, see \eqref{eq:Psi exp S}. After introducing the variables $y_i:=t^{-\epsilon_i/2}\epsilon_i u_i$ and $y_{n+m+1}:=q y_1$ we obtain the formula \eqref{eq:Zexp text}.
\end{proof}

\subsection{Computing elements of $\mathcal{B}^+$ using the linear map $\widetilde\Psi$}\label{sec:Psi-tilde computation}
In this section we demonstrate how the linear map $\widetilde\Psi$ works in specific examples. We will write:
\begin{align}
    \mathcal{B}^+(\fgl_{n|m}), \quad \mathcal{B}'(\fgl_{n|m})
\end{align}
when referring to the commutative subalgebras $\CB^+ \subset \CA^+$ and $\CB' \subset \CA'$ built out from the $R$-matrix \eqref{eq:R-matrix super}. The maps $\widetilde \Psi$ will be restricted to the commuting subalgebras and thus we will write:
\begin{align}
    \widetilde \Psi: \mathcal{B}'(\fgl_{n'|m'})\to \mathcal{B}^+(\fgl_{n|m})
\end{align}
In what follows, we will use the notations from Subsection \ref{sec:Psi tilde}.

\subsubsection{Reproducing $\Psi\left[H'^{(j)}_k\right]$ using $\widetilde\Psi$}
In the first example we set:
\begin{align}
    V''=V\simeq\mathbb C^{n+m},
     \quad V'\simeq\mathbb C, 
     \quad
     \iota = \text{identity map},
     \quad
     \iota'(1)=j
\end{align}
where we fix $j\in \mathcal{I}''=\mathcal{I}$. 
In this case we have:
\begin{align}
    \widetilde \Psi: \mathcal{B}'(\fgl_1)\to \mathcal{B}^+(\fgl_{n|m})
\end{align} 
The role of $H'^{(j)}_k$ is played by the element $H'_k$ which is a $1\times 1$ matrix with the matrix element being the symmetric rational function in $(z_1\ldots z_k)$ which is equal to $1$. Applying $\widetilde\Psi$ to $H'_k$ and using the definition \eqref{eq:Psi tilde} gives:
\begin{align}\label{eq:Psi tilde H'k}
    \widetilde \Psi\left [ H'_k\right]=  \Tr_{k+1\ldots 2k}
\left[
\check{R}_{\omega^k}(q z,z)
\left(\id^{\otimes k}\otimes \left(\iota'^{\otimes k}
H'_k\pi'^{\otimes k}\right)\right)
\right]
\end{align}
The expression $\iota'^{\otimes k}
H'_k$ is equal to $E_{jj}^{\otimes k}$ hence we may write:
\begin{align}\label{eq:Psi tilde H'k X}
    \widetilde \Psi\left [ H'_k\right]=  \Tr_{k+1\ldots 2k}
\left[
\check{R}_{\omega^k}(q z,z)
\left(\id^{\otimes k}\otimes E_{jj}^{\otimes k} \right)
\right]
\end{align}
The right hand side of \eqref{eq:Psi tilde H'k X} coincides with $\Psi\left[H'^{(j)}_k\right]$.  Let $H'(x)$ be the generating function of $H'_k$, then  by \eqref{eq:Psi exp S} we have:
\begin{align}
    \label{eq:Psi tilde exp S}
    \widetilde\Psi\left[H'(x)\right] = \exp \left(\sum_{k=1}^\infty \frac{x^k}{k}\epsilon_j^k\left(t^{-\epsilon_jk/2}S_{k}^{(j-1)}-t^{\epsilon_jk/2}S_{k}^{(j)}\right)\right)
\end{align}

\subsubsection{Computing elements exponentially generated by $S_k^{(j)}$}
In this example we fix $a\in \mathcal{I}''$ and set:
\begin{align}
    V''\simeq\mathbb C^{n+m+1},
     \quad V'\simeq\mathbb C, 
     \quad
         \iota(i) = \begin{cases}
        i & \text{if}\quad i<a\\
        i+1 & \text{if}\quad i\geq a
    \end{cases},
    \qquad
    \iota'(1) = a
\end{align}
The projection $\pi:V''\to V$ removes the $a$-th coordinate. Recall that $V$ has gradation given by $\epsilon_1\ldots \epsilon_{n+m}$ and let $\epsilon'_1\in\{\pm 1\}$ be the grading of the single vector of $V'$. Define the gradation of $V''$ as:
\begin{align}
    \label{eq:epsilon''}
    \epsilon''_i :=\begin{cases}
        \epsilon_i &\text{if}\quad  i<a\\
        \epsilon'_1 & \text{if}\quad i=a\\
        \epsilon_{i-1} & \text{if}\quad i>a
    \end{cases},
    \qquad i=1\ldots n+m+1
\end{align}
The two possible values of $\epsilon'_1$ give us two maps of the form:
\begin{align}
    \widetilde \Psi: \mathcal{B}'(\fgl_1)\to \mathcal{B}^+(\fgl_{n|m})
\end{align}
\begin{prop}
    With the map $\widetilde \Psi$ as constructed above we have:
\begin{align}
    \label{eq:Psi a exp S}
    \widetilde \Psi\left [ H'(x)\right]= \exp\left(\sum_{k=1}^\infty \frac{x^k}{k}\epsilon_1'^{k+1}\left(t^{-k/2}-t^{k/2}\right) S_k^{(a-1)}\right)
\end{align}
\end{prop}
\begin{proof}
Let us compute $\widetilde\Psi\left [ H_k'\right]$ using \eqref{eq:Psi tilde}. We note that $\iota'^{\otimes k}
H_k'=E_{aa}^{\otimes k}$ hence:
\begin{align}\label{eq:Psi tilde H'k - a}
    \widetilde \Psi\left [ H_k'\right]=  \pi^{\otimes k}
    \Tr_{k+1\ldots 2k}
\left[
\left(\check{R}^{(V'')}_{\omega^k}(q z,z)\right)
\left(\id^{\otimes k}\otimes 
E_{aa}^{\otimes k}\right)
\right]\iota^{\otimes k}
\end{align}
The expression between $\pi^{\otimes k}$ and $\iota^{\otimes k}$ is of the form \eqref{eq:Psi tilde H'k X} where $j=a$ and $\widetilde\Psi$ is taken as $\widetilde\Psi: \mathcal{B}'(\fgl_1)\to \mathcal{B}^+(\fgl_{n+1|m})$ or $\widetilde\Psi: \mathcal{B}'(\fgl_1)\to \mathcal{B}^+(\fgl_{n|m+1})$ depending on the sign of $\epsilon_1'$. Let $\widetilde S_k^{(j)}\in \mathcal{B}^+(\fgl_{n+1|m})\subset \CV''$ (or $\widetilde S_k^{(j)}\in \mathcal{B}^+(\fgl_{n|m+1})\subset \CV''$) denote the counterpart of  $S_k^{(j)}\in \mathcal{B}^+(\fgl_{n|m})$. With this notation and \eqref{eq:Psi tilde exp S} we can write the generating function of $\widetilde \Psi[H_k']$ from \eqref{eq:Psi tilde H'k - a} as:
\begin{align}\label{eq:Psi tilde H' cont'd}
    \widetilde\Psi\left[H'(x)\right] = \exp \left(\sum_{k=1}^\infty \frac{x^k}{k}\epsilon_1'^k\left(t^{-\epsilon_1'k/2}\left(\pi^{\otimes k}\widetilde S_{k}^{(a-1)}\iota^{\otimes k} \right)-t^{\epsilon_1'k/2}\left(\pi^{\otimes k}\widetilde S_{k}^{(a)}\iota^{\otimes k} \right)\right)\right)
\end{align}
Using \eqref{eqn: power sum 1} we compute:
\begin{align}\label{eq:S tilde to S at k=1}
\pi\widetilde S_{1}^{(j)}\iota 
=\begin{cases}
    S_{1}^{(j)}  &\text{if}\quad  j< a\\
    S_{1}^{(j-1)}  &\text{if}\quad  j\geq a
\end{cases}
\end{align}
From the recursive formula \eqref{eqn:power comm} with $\ell=1$ we find:
\begin{align}\label{eq:S tilde to S at k=1}
\pi^{\otimes k}\widetilde S_{k}^{(j)}\iota^{\otimes k}
=\begin{cases}
    S_{k}^{(j)}  &\text{if}\quad  j< a\\
    S_{k}^{(j-1)}  &\text{if}\quad  j\geq a
\end{cases}
\end{align}
For $j=a-1,a$ we have:
\begin{align}\label{eq:proj S tilde}
\pi^{\otimes k}\widetilde S_{k}^{(a-1)}\iota^{\otimes k}= S_k^{(a-1)},
\qquad 
\pi^{\otimes k}\widetilde S_{k}^{(a)}\iota^{\otimes k}= S_k^{(a-1)}
\end{align}
After inserting this into \eqref{eq:Psi tilde H' cont'd} we obtain:
\begin{align}
    \label{eq:Psi a exp S}
    \widetilde \Psi\left [ H'(x)\right]= \exp\left(\sum_{k=1}^\infty \frac{x^k }{k}\epsilon_1'^{k+1}\left(t^{-k/2}-t^{k/2}\right) S_k^{(a-1)}\right)
\end{align}
where we used the identity:
\[
\epsilon'_1(t^{-k/2}-t^{k/2})=(t^{-\epsilon'_1k/2}-t^{\epsilon'_1k/2}),
\qquad
\epsilon'_1 \in \{\pm 1\}
\]
\end{proof}
\subsubsection{Computing the elements $S_k^{(i)}\in\mathcal{B}^+$.}
In this example we derive a trace formula for the elements $S_k^{(i)}$. The derivation is based on the following. Suppose we have an identity:
\begin{align}\label{eq:exp S}
\sum_{k=0}^\infty G^{(i)}_k(u) v^k = 
\exp\left(\sum_{k=1}^\infty \frac{v^k}{k}(u^k-1) S^{(i)}_k\right)
\end{align}
where $G^{(i)}_k(u)$ is given, then we can compute $S^{(i)}_k$ using:
\begin{align}\label{eq:S-G}
S^{(i)}_k=\frac{d}{du}G^{(i)}_k(u)\Big{|}_{u=1},
\qquad k>0
\end{align}

Let us construct a map $\widetilde\Psi$ which will produce such elements $G^{(i)}_k(u)$. We fix $a\in \mathcal{I}''$ and set:
\begin{align}\label{eq:Psi tilde setup}
    V''\simeq\mathbb C^{n+m+2},
     \quad V'\simeq\mathbb C^2, 
     \quad
         \iota(i) = \begin{cases}
        i & \text{if}\quad i<a\\
        i+2 & \text{if}\quad i\geq a
    \end{cases},
    \qquad
    \iota'(1) = a, \quad
    \iota'(2) = a+1
\end{align}
The projection $\pi:V''\to V$ removes the $a$-th and $a+1$-st coordinates and $\pi':V''\to V'$ is the complementary projection. The gradation of $V'$ is fixed to be:
\begin{align}
    \label{eq:epsilons12}
\epsilon'_1=1, \quad \epsilon'_2=-1
\end{align}
then the gradation of $V''$ is:
\begin{align}
    \label{eq:epsilon''}
    \epsilon''_i :=\begin{cases}
        \epsilon_i &\text{if}\quad  i<a\\
        \epsilon'_1=1 & \text{if}\quad i=a\\
        \epsilon'_2=-1 & \text{if}\quad i=a+1\\
        \epsilon_{i-2} & \text{if}\quad i>a+1
    \end{cases}
\end{align}
With this data we have a map $\widetilde \Psi$ of the form:
\begin{align}
    \widetilde \Psi: \mathcal{B}'(\fgl_{1|1})\to \mathcal{B}^+(\fgl_{n|m})
\end{align}
\begin{prop}
    With the map $\widetilde \Psi$ as constructed above we have:
    \begin{align}
        \label{eq:Psi tilde maps HH}
    S_k^{(a-1)}=    \frac{1}{t^{-k/2}-t^{k/2}}
    \widetilde\Psi \left(\sum_{l=1}^k (-1)^{l+1} lH_{k-l}'^{(1)}*'H_{l}'^{(2)}\right)
    \end{align}
\end{prop}
\begin{proof}
Let us apply the map $\widetilde\Psi$ to $H'^{(1)}_k$ and $H'^{(2)}_k$. Using \eqref{eq:Psi tilde} and $\iota'$ from \eqref{eq:Psi tilde setup} we can write:
\begin{align}\label{eq:Psi tilde H'k1}
    &\widetilde \Psi\left [ H'^{(1)}_k\right]=  \pi^{\otimes k}
    \Tr_{k+1\ldots 2k}
\left[
\left(\check{R}^{(V'')}_{\omega^k}(q z,z)\right)
\left(\id^{\otimes k}\otimes 
E_{aa}^{\otimes k}\right)
\right]\iota^{\otimes k}\\
\label{eq:Psi tilde H'k2}
    &\widetilde \Psi\left [ H'^{(2)}_k\right]=  \pi^{\otimes k}
    \Tr_{k+1\ldots 2k}
\left[
\left(\check{R}^{(V'')}_{\omega^k}(q z,z)\right)
\left(\id^{\otimes k}\otimes 
E_{a+1a+1}^{\otimes k}\right)
\right]\iota^{\otimes k}
\end{align}
Next we compute the generating functions of \eqref{eq:Psi tilde H'k1} and \eqref{eq:Psi tilde H'k2} similarly to \eqref{eq:Psi tilde H' cont'd}:
\begin{align}
\label{eq:Psi tilde H'1 }
    \widetilde\Psi\left[H'^{(1)}(x)\right] &= \exp \left(\sum_{k=1}^\infty \frac{x^k}{k}\epsilon_1'^k\left(t^{-\epsilon_1'k/2}\left(\pi^{\otimes k}\widetilde S_{k}^{(a-1)}\iota^{\otimes k} \right)-t^{\epsilon_1'k/2}\left(\pi^{\otimes k}\widetilde S_{k}^{(a)}\iota^{\otimes k} \right)\right)\right)\\
\label{eq:Psi tilde H'2}
    \widetilde\Psi\left[H'^{(2)}(x)\right] &= \exp \left(\sum_{k=1}^\infty \frac{x^k}{k}\epsilon_2'^k\left(t^{-\epsilon_2'k/2}\left(\pi^{\otimes k}\widetilde S_{k}^{(a)}\iota^{\otimes k} \right)-t^{\epsilon_2'k/2}\left(\pi^{\otimes k}\widetilde S_{k}^{(a+1)}\iota^{\otimes k} \right)\right)\right)
\end{align}
where $\widetilde S_k^{(j)}\in \mathcal{B}^+(\fgl_{n+1|m+1})\subset \CV''$ denotes the counterpart of  $S_k^{(j)}\in \mathcal{B}^+(\fgl_{n|m})$. By the same logic leading to \eqref{eq:proj S tilde} we have:
\begin{align}\label{eq:proj S tilde again}
\pi^{\otimes k}\widetilde S_{k}^{(a-1)}\iota^{\otimes k}= S_k^{(a-1)},
\qquad 
\pi^{\otimes k}\widetilde S_{k}^{(a)}\iota^{\otimes k}= S_k^{(a-1)}
,
\qquad 
\pi^{\otimes k}\widetilde S_{k}^{(a+1)}\iota^{\otimes k}= S_k^{(a-1)}
\end{align}
This gives us the following formulas:
\begin{align}
\label{eq:Psi tilde H'1 exp}
    \widetilde\Psi\left[H'^{(1)}(x)\right] &= \exp \left(\sum_{k=1}^\infty \frac{x^k}{k}\left(t^{-k/2}-t^{k/2}\right)S_k^{(a-1)}\right)\\
\label{eq:Psi tilde H'2 exp}
    \widetilde\Psi\left[H'^{(2)}(x)\right] &= 
      \exp \left(\sum_{k=1}^\infty \frac{(-1)^{k+1}x^k}{k}\left(t^{-k/2}-t^{k/2}\right)S_k^{(a-1)}\right)
\end{align}
where we substituted $\epsilon'_1=1$ and $\epsilon'_2=-1$. For our final step we need to show that the map $\widetilde \Psi$ satisfies: 
\begin{align}
    \label{eq:Psi tilde Hk Hl}
    \widetilde\Psi\left[H'^{(1)}_k*'H'^{(2)}_l\right]=
    \widetilde\Psi\left[H'^{(1)}_k\right]*\widetilde\Psi\left[H'^{(2)}_l\right]
\end{align}

Let $V_{(a^k,(a+1)^l)}\subset V''^{\otimes (k+l)}$ be the coordinate subspace which is spanned by the vectors:
$$
\ket{\sigma(a\ldots a,a+1\ldots a+1},
\qquad \forall\sigma\in \mathfrak S_{k+l}/\mathfrak S_l\times \mathfrak S_k
$$
where $a$'s appear $k$ times and $a+1$'s appear $l$ times. Consider the expression:
\begin{align}\label{eq:Ea^k Ea+1^l}
    E^{\otimes k}_{aa}*'E^{\otimes l}_{a+1a+1} \in \mathcal{B}'(\fgl_{n+1|m+1})    
\end{align}
where $*'$ is the shuffle product of the shuffle algebra $\mathcal{A}'(\fgl_{n+1|m+1})$. Because of the presence of the projectors $E^{\otimes k}_{aa}$ and $E^{\otimes l}_{a+1a+1}$ the tensor \eqref{eq:Ea^k Ea+1^l} has non-zero matrix elements only in the subspace $V_{(a^k,(a+1)^l)}$. This can be seen by looking at the definitions \eqref{eq:A*'B-Rcheck} and \eqref{eq:Gamma_prime} and noting that the tensor:
\[
  (E^{\otimes k}_{aa}\otimes \id^{\otimes l})
    \check R^\bullet_{\omega^{(k+l)}}
    (E^{\otimes l}_{a+1a+1}\otimes \id^{\otimes k}) 
\]
has non-zero matrix elements only in the subspace $V_{(a^k,(a+1)^l)}$ and left and right multiplication of this tensor by $\check R_\sigma$ preserves this property. 

This argument leads us to the following equation:
\begin{align}\label{eq:E11E22}
    \iota'^{\otimes (k+l)}\left(E^{\otimes k}_{11}*'E^{\otimes l}_{22} \right)\pi'^{\otimes (k+l)}= E^{\otimes k}_{aa}*'E^{\otimes l}_{a+1a+1} \in \mathcal{B}'(\fgl_{n+1|m+1})
\end{align}
where on the left hand side $*'$ is the shuffle product of the shuffle algebra $\mathcal{A}'(\fgl_{1|1})$ and on the right hand side $*'$ is the shuffle product of the shuffle algebra $\mathcal{A}'(\fgl_{n+1|m+1})$. Using \eqref{eq:E11E22} we can write the left hand side of \eqref{eq:Psi tilde Hk Hl} as:
\begin{align}\label{eq:Psi tilde HH}
        \widetilde\Psi\left[H'^{(1)}_k*'H'^{(2)}_l\right]=  \pi^{\otimes N}\left(
    \Tr_{N+1\ldots 2N}
\left[
\left(\check{R}^{(V'')}_{\omega^N}(q z,z)\right)
\left( \id_{V''}^{\otimes N} \otimes \left(
H'^{(a)}_k*'H'^{(a+1)}_l\right)\right)
\right]\right)\iota^{\otimes N}
\end{align}
where $N=k+l$. The expression on the right hand side of \eqref{eq:Psi tilde HH} respects the shuffle algebras' multiplication (see Remark \ref{rmk:Psi pi map}). Hence the right hand side of \eqref{eq:Psi tilde HH} factorizes and we obtain \eqref{eq:Psi tilde Hk Hl}. 

With \eqref{eq:Psi tilde H'1 exp}, \eqref{eq:Psi tilde H'2 exp} and \eqref{eq:Psi tilde Hk Hl} we can compute $\widetilde\Psi$ of a product of two generating functions:
\begin{align}\label{eq:Psi tilde H1H2}
    \widetilde\Psi\left[H'^{(1)}(x)H'^{(2)}(-ux)\right]
    =    \exp\left(\sum_{k=1}^\infty \frac{-x^k}{k}\left(t^{-k/2}-t^{k/2}\right)\left(u^k-1\right) S_k^{(a-1)}\right)
\end{align}
This expression is of the form \eqref{eq:exp S} with $x$ playing the role of $v$. Therefore we can use \eqref{eq:S-G}:
\begin{align}\label{eq:Psi tilde = S}
    \widetilde\Psi\left[\sum_{l=0}^k (-1)^l l H_{k-l}'^{(1)}*'H_{l}'^{(2)}\right]
    =  - \left(t^{-k/2}-t^{k/2}\right)S_k^{(a-1)}
\end{align}
After a simple rearrangement we get \eqref{eq:Psi tilde maps HH}.
\end{proof}
Using Remark \ref{rmk:H shuffle product in A prime} compute $H_{k-l}'^{(1)}*'H_{l}'^{(2)}$ in terms of matrix units. After inserting the result into \eqref{eq:Psi tilde maps HH} and shifting the index $a\rightarrow a+1$ we obtain a trace formula for $S_k^{(a)}$:
\begin{align}\label{eq:S trace}
    S_k^{(a)}= &\frac{-1}{t^{-k/2}-t^{k/2}} \nonumber\\
    \times &
  \pi^{\otimes k}
    \Tr_{k+1\ldots 2k}
\left[
\check{R}^{(V'')}_{\omega^k}(q z,z)
\left( \id^{\otimes k} \otimes\sum_{\lambda\in\{a+1,a+2\}^k}(-1)^{m_{a+2}(\lambda)}m_{a+2}(\lambda)\bigotimes_{i\in \lambda} E_{ii}\right)
\right]\iota^{\otimes k}
\end{align}
where the trace is taken over $k$ copies of $V''$. In the graphical form this equation corresponds to \eqref{eq:S-trace intro} from the Introduction.

%%% ------ %%% ------ %%% ------ %%% ------ %%% ------ 
%%% ------ %%% ------ %%% ------ %%% ------ %%% ------ 
%%% ------ %%% ------ %%% ------ %%% ------ %%% ------ 
%%% ------ %%% ------ %%% ------ %%% ------ %%% ------ 
%%% ------ %%% ------ %%% ------ %%% ------ %%% ------ 

\appendix
\section{A trace identity from crossing unitarity}\label{app:trace}
In this section we derive an identity related to $R$-matrices of $\uu$ and their crossing unitarity property \eqref{eq:crossing_unitarity}. This identity involves tensors $F\in \normalfont{\text{End}}(V^{\otimes l})$ and $G\in \normalfont{\text{End}}(V^{\otimes k})$, for $l,k>0$, which are not necessarily elements of shuffle algebras. The spectral parameters will not change as they pass through these tensors, hence their diagrams look as follows:
\begin{align}\label{eq:A_graph_no_q}
F(z_1\ldots z_l) = ~
\begin{tikzpicture}[scale=0.5,baseline=(current  bounding  box.center)]
\draw[invarrow=0.4] (1,0) -- (1,0.5);
\draw[invarrow=0.4] (2,0) -- (2,0.5);
\draw[invarrow=0.6] (1,1.5) -- (1,2);
\draw[invarrow=0.6] (2,1.5) -- (2,2);
\draw (0.5,0.5) -- (0.5,1.5) -- (2.5,1.5) -- (2.5,0.5) -- (0.5,0.5);
\node at (1.5,1) {$\scriptstyle F$};
%%% spectral parameters
\node[below] at (1,0) {$\scriptstyle z_1$};
\node[below] at (2,0) {$\scriptstyle z_k$};
\node at (1.5,0) {$\scriptstyle \ldots$};
\node[above] at (1,2) {$\scriptstyle z_1$};
\node[above] at (2,2) {$\scriptstyle z_k$};
\node at (1.5,2) {$\scriptstyle \ldots$};
\end{tikzpicture}
\end{align}

\begin{lem}\label{lem:traces}
Let $k,l$ be two positive integers.
Let $z=(z_1\ldots z_k)$, $w=(w_1\ldots w_l)$ and $\omega^j\in \mathfrak{S}_{k+l}$ a $j-$cycle. 
For two tensors $F\in \normalfont{\text{End}}(V^{\otimes l})$ and $G\in \normalfont{\text{End}}(V^{\otimes k})$ we have:
\begin{align}
    \label{eq:trace_identity}
    \Tr \left[ F\otimes G\right ]
    &= 
    \Tr \left[
    \check R^\bullet_{\omega^k}(w,z)
    F_{1\ldots l}
    \check R_{\omega^l}(z,w)
    G_{1\ldots k}
    \right]
    % \\
        % \label{eq:trace_identity_1}
    % \Tr \left[ A\otimes B\right ]
    % &= 
    % \Tr \left[
    % A_{1\ldots l}
    % \check R^\bullet_{\omega^l}(z,w)
    % B_{1\ldots k}
    % \check R_{\omega^k}(w,z)
    % \right]
\end{align}
with the trace taken over $V^{\otimes (k+l)}$.
\end{lem}
\begin{proof}
We provide a recursive proof of this identity starting from the right hand side and showing how the $\check R$ matrices can be removed with the help of the crossing unitarity. The sequence of algebraic manipulations is guided by the graphical representation. 

We note that $\omega^k=\omega^{-l}$. Let $\tilde\omega^{\pm (l-1)}\in \mathfrak{S}_{k+l-1}$ and $\tilde w=(w_1\ldots w_{l-1})$ denote the counterparts of $\omega^{\pm l}$ and the alphabet $w$ in which $l$ is decreased by $1$. Denote the right hand side of \eqref{eq:trace_identity} by $Y$
\begin{align}
    \label{eq:Y}
    Y
    = 
    \Tr \left[
    \check R^\bullet_{\omega^k}(w,z)
    F_{1\ldots l}
    \check R_{\omega^l}(z,w)
    G_{1\ldots k}
    \right]
\end{align}
Graphically we have:
\begin{align}\label{eq:Y-graph}
Y 
=\quad
\begin{tikzpicture}[scale=0.5,baseline=(current  bounding  box.center)]
% level 1
\rcheck{1}{0}
\rcheck{2}{1}
\rcheck{0}{1}
\rcheck{1}{2}
\node at (1.5,0.5) {$\scriptstyle\bullet$};
\node at (2.5,1.5) {$\scriptstyle\bullet$};
\node at (0.5,1.5) {$\scriptstyle\bullet$};
\node at (1.5,2.5) {$\scriptstyle\bullet$};
% level 2
\draw (-0.5,3) -- (-0.5,4) -- (1.5,4) -- (1.5,3) -- (-0.5,3);
\draw (0,2) -- (0,3);
\draw (0,0) -- (0,1);
\node at (0.5,3.5) {$\scriptstyle F$};
% level 3
\rcheck{1}{4}
\rcheck{2}{5}
\rcheck{0}{5}
\rcheckarrows{1}{6}
\node[left] at (0.2,8.3) {$\scriptstyle z_1$};
\node at (0.27,8.15) {$\scriptstyle \ldots$};
\node[left] at (1.3,8.3) {$\scriptstyle z_k$};
\node[left] at (0.2,4.3) {$\scriptstyle w_1$};
\node at (0.27,4.16) {$\scriptstyle \ldots$};
\node[left] at (1.33,4.3) {$\scriptstyle w_l$};
% level 4
\draw (-0.5,7) -- (-0.5,8) -- (1.5,8) -- (1.5,7) -- (-0.5,7);
\node at (0.5,7.5) {$\scriptstyle G$};
\draw (0,6) -- (0,7);
\draw (0,4) -- (0,5);
\draw (2,3) -- (2,4);
\draw (3,2) -- (3,5);
\draw[arrow=0.25] (0,6) -- (1,5);
\draw[arrow=0.25] (3,6) -- (2,5);
%trace
\draw [rounded corners=5pt] (3,6) -- (4,6) -- (4,1) -- (3,1);
\draw [rounded corners=5pt] (2,7) -- (5,7) -- (5,0) -- (2,0);
\draw [rounded corners=5pt] (1,8) -- (1,8.5) -- (6,8.5) -- (6,-0.5) -- (1,-0.5) -- (1,0);
\draw [rounded corners=5pt] (0,8) -- (0,9) -- (7,9) -- (7,-1) -- (0,-1) -- (0,0);
\end{tikzpicture}
\quad =\quad
\begin{tikzpicture}[scale=0.5,baseline=(current  bounding  box.center)]
% G
\draw (0,4) -- (0,5) -- (2,5) -- (2,4) -- (0,4);
\node at (1,4.5) {$\scriptstyle G$};
\draw [rounded corners=5pt] (1.5,4) -- (1.5,3.5) -- (2.5,3.5) -- (2.5,5.5) -- (1.5,5.5) -- (1.5,5);
\draw [rounded corners=5pt] (0.5,4) -- (0.5,3) -- (3,3) -- (3,6) -- (0.5,6) -- (0.5,5);
% F
\draw (0,0) -- (0,1) -- (2,1) -- (2,0) -- (0,0);
\node at (1,0.5) {$\scriptstyle F$};
\draw [rounded corners=5pt] (1.5,0) -- (1.5,-0.5) -- (2.5,-0.5) -- (2.5,1.5) -- (1.5,1.5) -- (1.5,1);
\draw [rounded corners=5pt] (0.5,0) -- (0.5,-1) -- (3,-1) -- (3,2) -- (0.5,2) -- (0.5,1);
\end{tikzpicture}
\end{align}
where the first equality represents \eqref{eq:Y} and the second equality represents the equation which we need to prove. 
Suppose \eqref{eq:trace_identity} holds when $l$ is replaced by $l-1$ and the tensor $F$ is replaced by any tensor in $\normalfont{\text{End}}(V^{\otimes {(l-1)}})$. Let us show that:
\begin{align}\label{eq:Y-induction}
        Y
    = 
    \Tr \left[
    \check R^\bullet_{\tilde\omega^{-(l-1)}}(\tilde w,z)
    \left(\Tr_{k+l} 
    F_{1\ldots l-1,k+l}\right)
    \check R_{\tilde \omega^{(l-1)}}(z,\tilde w)
    G_{1\ldots k}
    \right]
\end{align}
with the outer trace taken over $V^{\otimes (k+l-1)}$. This will give us a proof by induction.

We can write the two blocks $\check R_{\omega^l}(z,w)$ and $\check  R^\bullet_{\omega^{-l}}(w,z)$ in \eqref{eq:Y}
as follows:
\begin{align}\label{eq:R-rec}
&\check  R_{\omega^{l}}(z,w) =   
\check R_{l}(w_l/z_1)\cdots \check  R_{k+l-1}(w_l/z_k)\times
\check R_{\tilde\omega^{l-1}}(z,\tilde w)
\\
\label{eq:Rb-rec}
&\check  R^{\bullet}_{\omega^{-l}}(w,z) = 
\check  R^{\bullet}_{\tilde\omega^{-(l-1)}}(\tilde w,z)
\times
\check R^{\bullet}_{k+l-1}(z_k/w_l) \cdots    
\check  R^{\bullet}_{l}(z_1/w_l) 
\end{align}
The individual matrices $\check R_{l+i-1}(w_l/z_{i-l+1})$ and $\check R_{k+l-i}^\bullet(z_{k-i+1}/w_l)$, $i=1\ldots k$, in these two expressions correspond to the crossings associated with the ``innermost'' vector space carrying the parameter $w_l$ in the middle term in \eqref{eq:Y-graph}. Insert \eqref{eq:R-rec} and \eqref{eq:Rb-rec} into \eqref{eq:Y}:
\begin{align}\label{eq:trace_proof_a}
    Y
    =
    \Tr \left[\cdots 
    \check R^{\bullet}_{k+l-1}(z_k/w_l)
    X^{(k,l)}  
    \check  R_{k+l-1}(w_l/z_k) 
    \cdots   \right]
\end{align}
where $\cdots$ contain tensors acting on the spaces $1\ldots k+l-1$ and $X^{(k,l)}$ is given by: 
\begin{align}\label{eq:X}
    X^{(k,l)}=X^{(k,l)}_{1\ldots k+l-1} := 
    \check R^{\bullet}_{k+l-2}(z_{k-1}/w_l)\cdots \check  R^{\bullet}_{l}(z_1/w_l)
F_{1\ldots l}
    \check  R_{l}(w_l/z_1) \cdots    \check  R_{k+l-2}(w_l/z_{k-1})
\end{align}
We can replace $X^{(k,l)}_{1\ldots k+l-1}$ with $X^{(k,l)}_{1\ldots k+l-2,k+l}$ using permutation matrices $P_{k+l-1}$ as follows:
\begin{align}\label{eq:trace_proof_b}
    Y
    =
    \Tr \left[\cdots 
    \check R^{\bullet}_{k+l-1}(z_k/w_l) 
    P_{k+l-1} X^{(k,l)}_{1\ldots k+l-2,k+l}P_{k+l-1}
    \check  R_{k+l-1}(w_l/z_k) 
    \cdots   \right]
\end{align}
Next we apply the transposition $T_{k+l}$ inside the trace and compute\footnote{In this computation we used the identity $\Tr\left[F^{T_i} G^{T_i}\right]=\Tr\left[F G\right]$ which holds for any two matrices $F,G\in \text{End}(V^{\otimes N})$ and $i=1\ldots N$.}:
\begin{align}
    Y
    &=
    \Tr \left[\cdots \left(
    \check R^{\bullet}_{k+l-1}(z_k/w_l)
    P_{k+l-1}\right)^{T_{k+l}}\left( X_{1\ldots k+l-2,k+l}P_{k+l-1}\check  R_{k+l-1}(w_l/z_k) \right)^{T_{k+l}} \cdots   \right]
    \nonumber\\
    &=
    \Tr \left[\cdots \left(
    \check R^{\bullet}_{k+l-1}(z_k/w_l)
    P_{k+l-1}\right)^{T_{k+l}}
    \left(P_{k+l-1}\check  R_{k+l-1}(w_l/z_k) \right)^{T_{k+l}}
    \left( X_{1\ldots k+l-2,k+l}\right)^{T_{k+l}} \cdots   \right]
    \nonumber\\
    &=
    \Tr \left[\cdots 
    \left( X_{1\ldots k+l-2,k+l}\right)^{T_{k+l}} \cdots   \right]=
    \Tr \left[\cdots 
     X_{1\ldots k+l-2,k+l} \cdots   \right]
     \label{eq:trace_proof_c}
\end{align}
To go from the second line to the third we used the crossing unitarity \eqref{eq:crossing_unitarity}. 
This computation brings 
\eqref{eq:trace_proof_a} to the form:
\begin{align}\label{eq:trace_proof_d}
    Y
    =
    \Tr \left[\cdots \check R^{\bullet}_{k+l-2,k+l}(z_{k-1}/w_l) X^{(k-1,l)} \check  R_{k+l-2,k+l}(w_l/z_{k-1}) \cdots   \right]
\end{align}
where we used the explicit form of $X^{(k,l)}$ in \eqref{eq:X}. This expression for $Y$ has the same structure as \eqref{eq:trace_proof_a} therefore we can reuse  \eqref{eq:trace_proof_b} and \eqref{eq:trace_proof_c}. Performing this computation in total $k$ times gives us the following expression for $Y$:
\begin{align}\label{eq:Y2}
        Y
    = 
    \Tr \left[
    \check R^\bullet_{\tilde\omega^{-(l-1)}}(\tilde w,z)
    \tilde F_{1\ldots l-1}
    \check R_{\tilde \omega^{(l-1)}}(z,\tilde w)
    G_{1\ldots k}
    \right]
\end{align}
where $\tilde F_{1\ldots l-1}$ equals to a partially traced matrix $F$:
\begin{align*}
    \tilde F_{1\ldots l-1}:= \Tr_{k+l} 
    X^{(1,l)}_{1\ldots l-1,k+l}=\Tr_{k+l} 
    F_{1\ldots l-1,k+l}
\end{align*}
From the graphical point of view this computation corresponds to pulling the innermost vector spaces and removing all its crosses in the middle term of \eqref{eq:Y-graph}:
\begin{align}\label{eq:Y-graph_x}
Y = 
\quad
\begin{tikzpicture}[scale=0.5,baseline=(current  bounding  box.center)]
% level 1
\rcheck{1}{0}
% \rcheck{2}{1}
\rcheck{0}{1}
% \rcheck{1}{2}
\node at (1.5,0.5) {$\scriptstyle\bullet$};
% \node at (2.5,1.5) {$\scriptstyle\bullet$};
\node at (0.5,1.5) {$\scriptstyle\bullet$};
% \node at (1.5,2.5) {$\scriptstyle\bullet$};
% level 2
\draw (-0.5,3) -- (-0.5,4) -- (0.5,4) -- (0.5,3) -- (-0.5,3);
\draw (0,2) -- (0,3);
\draw (0,0) -- (0,1);
\node at (0,3.5) {$\scriptstyle \tilde F$};
% level 3
% \rcheck{1}{4}
% \rcheck{2}{5}
\rcheck{0}{5}
\rcheckarrows{1}{6}
\node[left] at (0.2,8.3) {$\scriptstyle z_1$};
\node at (0.27,8.16) {$\scriptstyle \ldots$};
\node[left] at (1.28,8.3) {$\scriptstyle z_k$};
\node[left] at (0.2,4.3) {$\scriptstyle \tilde w$};
% level 4
\draw (-0.5,7) -- (-0.5,8) -- (1.5,8) -- (1.5,7) -- (-0.5,7);
\node at (0.5,7.5) {$\scriptstyle G$};
\draw (0,6) -- (0,7);
\draw (0,4) -- (0,5);
\draw (1,2) -- (1,5);
\draw (2,1) -- (2,6);
\draw[arrow=0.25] (0,6) -- (1,5);
%trace
% \draw [rounded corners=5pt] (3,6) -- (4,6) -- (4,1) -- (3,1);
\draw [rounded corners=5pt] (2,7) -- (3,7) -- (3,0) -- (2,0);
\draw [rounded corners=5pt] (1,8) -- (1,8.5) -- (4,8.5) -- (4,-0.5) -- (1,-0.5) -- (1,0);
\draw [rounded corners=5pt] (0,8) -- (0,9) -- (5,9) -- (5,-1) -- (0,-1) -- (0,0);
\end{tikzpicture}
\end{align}
In this computation we derived \eqref{eq:Y-induction} which completes the proof.
\end{proof}

\section{Computing evaluations $\alpha_i$}\label{app:evaluations Psi}
In this Appendix we compute the evaluation $\alpha_i \left( \Psi\left[E_{jj}^{\otimes k}\right]\right)$. The result is summarized in the following Lemma.
\begin{lem}\label{lem:alpha Psi}
We have the following identity:
\begin{align}\label{eq:alpha Psi}
    \alpha_i \left( \Psi\left[E_{jj}^{\otimes k}\right]\right) = 
     \begin{cases}
      \prod_{l=1}^{k}\epsilon_j\frac{ t^{\epsilon_j /2}-q^{l}t^{-\epsilon_j /2}}{1-q^{l}}  
      & i=j\\
q^{k\delta_{i>j}}\prod_{l=1}^{k}\epsilon_j\frac{ t^{\epsilon_j /2}-q^{l-1}t^{-\epsilon_j /2}}{1-q^{l}}
    & i\neq j
     \end{cases}
\end{align}
\end{lem}
\begin{proof}
According to Definition \ref{def:evaluation}, in order to compute $\alpha_i$, we need to evaluate the residue of $\Psi\left[E_{jj}^{\otimes k}\right]$ at:
$$
\left\{z_1 = y, \ldots,  z_{k} = q^{k-1} y\right\}
$$
and express it in the form 
\eqref{eqn:residue (k)}, i.e.:
\begin{equation}\label{eqn:residue (k) Psi}
\underset{\left\{z_1 = y, \ldots,  z_{k} = q^{k-1} y\right\}}{\normalfont{\text{Res}}}\Psi\left[E_{jj}^{\otimes k}\right] =
(t^{\frac 12} - t^{-\frac 12})^{k-1} \prod_{l=1}^{k-2} f \left(q^l \right)^{k-1-l}
 \times \check R_{1} \left(q^{-1}\right) \cdots \check R_{k-1} \left(q^{-k+1}\right) X_k^{(k)}(y)
\end{equation}
This will give us an explicit expression for $X_k^{(k)}(y)$ in which we can take the coefficient of $E_{ii}$ and derive \eqref{eq:alpha Psi}.

For clarity we demonstrate the logic of the computation by considering the case $k=3$. The residue of $\Psi\left[E_{jj}^{\otimes 3}\right]$ is:
\begin{align}\label{eq:res Psi A}
\underset{\left\{z_1 = y, z_2 = q y,  z_{3} = q^2 y\right\}}{\normalfont{\text{Res}}}\Psi\left[E_{jj}^{\otimes 3}\right] 
=\quad 
\underset{\left\{z_1 = y, z_2 = q y,  z_{3} = q^2 y\right\}}{\normalfont{\text{Res}}}
\begin{tikzpicture}[scale=0.5,baseline=(current  bounding  box.center)]
% R checks
\rcheck{1}{0}
\rcheck{2}{1}
\rcheck{3}{2}
\rcheck{0}{1}
\rcheck{1}{2}
\rcheck{2}{3}
\rcheck{-1}{2}
\rcheck{0}{3}
\rcheck{1}{4}
% Box of E's
\draw (2-0.5-0.15,5) -- (4+0.5+0.15,5) -- (4+0.5+0.15,6) -- (2-0.5-0.15,6) -- (2-0.5-0.15,5);
\node at (3,5.5) {$\scriptstyle E_{jj}\otimes \cdots \otimes E_{jj}$};
% Straight lines
\draw (3,4) -- (3,5);
\draw (4,3) -- (4,5);
\draw (0,1) -- (0,0);
\draw (-1,2) -- (-1,0);
\draw[invarrow=0.8] (1,5) -- (1,7);
\draw[invarrow=0.85] (0,4) -- (0,7);
\draw[invarrow=0.89] (-1,3) -- (-1,7);
% trace lines
\draw [rounded corners=5pt] (4,2) -- (5,2) -- (5,6.5) -- (4,6.5) -- (4,6);
\draw [rounded corners=5pt] (3,1) -- (6,1) -- (6,7) -- (3,7) -- (3,6);
\draw [rounded corners=5pt] (2,0) -- (7,0) -- (7,7.5) -- (2,7.5) -- (2,6);
% spectral top
\node[above] at (-1,7) {$\scriptstyle q z_1$};
\node[above] at (0,7) {$\scriptstyle q z_{2}$};
\node[above] at (1,7) {$\scriptstyle q z_3$};
% spectral bottom
\node[below] at (-1,0) {$\scriptstyle z_1$};
\node[below] at (0,0) {$\scriptstyle z_{2}$};
\node[below] at (1,0) {$\scriptstyle z_3$};
\end{tikzpicture}
\end{align}
The trace can be immediately computed because $E_{jj}$ select the $j$-th matrix element in each traced tensor factor. 
The poles in the expression on the right hand side in \eqref{eq:res Psi A} can only arise in the matrices $\check R(z)$ when $z\to 1$. We can compute the corresponding residues using \eqref{eq:R residue graph} which results into replacing some $\check R$ matrices with identity matrices and an overall power of $t^{1/2}-t^{-1/2}$. In the computation below we mark these $\check R$ matrices with circles:
\begin{align}\label{eq:res Psi B}
\underset{\left\{z_1 = y, z_2 = q y,  z_{3} = q^2 y\right\}}{\normalfont{\text{Res}}}
\begin{tikzpicture}[scale=0.5,baseline=(current  bounding  box.center)]
% R checks
\rcheck{1}{0}
\rcheck{2}{1}
\rcheck{3}{2}
\rcheck{0}{1}
\rcheck{1}{2}
\rcheck{2}{3}
\rcheck{-1}{2}
\rcheck{0}{3}
\rcheckarrows{1}{4}
% Straight lines
\draw[invarrow=0.75] (3,4) -- (3,5);
\draw[invarrow=0.87] (4,3) -- (4,5);
\draw (3,0) -- (3,1);
\draw (4,0) -- (4,2);
\draw (0,1) -- (0,0);
\draw (-1,2) -- (-1,0);
\draw (1,5) -- (1,5);
\draw[invarrow=0.75] (0,4) -- (0,5);
\draw[invarrow=0.87] (-1,3) -- (-1,5);
% spectral top left
\node[above] at (-1,5) {$\scriptstyle q z_1$};
\node[above] at (0,5) {$\scriptstyle q z_{2}$};
\node[above] at (1,5) {$\scriptstyle q z_3$};
% spectral top right
\node[above] at (2,5.7) {$\scriptstyle z_1$};
\node[above] at (3,5.7) {$\scriptstyle z_{2}$};
\node[above] at (4,5.7) {$\scriptstyle z_3$};
% indices 
\node[above] at (2,5) {$\scriptstyle j$};
\node[above] at (3,5) {$\scriptstyle j$};
\node[above] at (4,5) {$\scriptstyle j$};
\node[below] at (2,0) {$\scriptstyle j$};
\node[below] at (3,0) {$\scriptstyle j$};
\node[below] at (4,0) {$\scriptstyle j$};
% circles
\draw (0.5,1.5) circle (0.2);
\draw (2.5,1.5) circle (0.2);
\end{tikzpicture}
%%%%%%
%%%%%%
=
%%%%%%
%%%%%%
(t^{1/2}-t^{-1/2})^2
\begin{tikzpicture}[scale=0.5,baseline=(current  bounding  box.center)]
% R checks
\rcheck{1}{0}
% \rcheck{2}{1}
\rcheck{3}{2}
% \rcheck{0}{1}
\rcheck{1}{2}
\rcheck{2}{3}
\rcheck{-1}{2}
\rcheck{0}{3}
\rcheckarrows{1}{4}
% Straight lines
\draw[invarrow=0.75] (3,4) -- (3,5);
\draw[invarrow=0.87] (4,3) -- (4,5);
\draw (3,0) -- (3,2);
\draw (4,0) -- (4,2);
\draw (0,2) -- (0,0);
\draw (1,2) -- (1,1);
\draw (2,2) -- (2,1);
\draw (-1,2) -- (-1,0);
\draw[invarrow=0.75] (0,4) -- (0,5);
\draw[invarrow=0.87] (-1,3) -- (-1,5);
% spectral top left
\node[above] at (-1,5) {$\scriptstyle q y$};
\node[above] at (0,5) {$\scriptstyle q^2 y$};
\node[above] at (1,5) {$\scriptstyle q^3 y$};
% spectral top right
\node[above] at (2,5.7) {$\scriptstyle y$};
\node[above] at (3,5.7) {$\scriptstyle q y$};
\node[above] at (4,5.7) {$\scriptstyle q^2 y$};
% indices 
\node[above] at (2,5) {$\scriptstyle j$};
\node[above] at (3,5) {$\scriptstyle j$};
\node[above] at (4,5) {$\scriptstyle j$};
\node[below] at (2,0) {$\scriptstyle j$};
\node[below] at (3,0) {$\scriptstyle j$};
\node[below] at (4,0) {$\scriptstyle j$};
\end{tikzpicture}
\end{align}
On the right hand side we can use the unitarity relation \eqref{eq:unitarity graph}. This leads us to the following equation:
\begin{align}\label{eq:res Psi C}
\underset{\left\{z_1 = y, z_2 = q y,  z_{3} = q^2 y\right\}}{\normalfont{\text{Res}}}\Psi\left[E_{jj}^{\otimes 3}\right] 
=
%%%%%%
%%%%%%
(t^{1/2}-t^{-1/2})^2f(q)
\begin{tikzpicture}[scale=0.5,baseline=(current  bounding  box.center)]
% R checks
% \rcheck{1}{0}
% \rcheck{2}{1}
\rcheck{3}{2}
% \rcheck{0}{1}
% \rcheck{1}{2}
\rcheck{2}{3}
\rcheck{-1}{2}
\rcheck{0}{3}
\rcheckarrows{1}{4}
% Straight lines
\draw[invarrow=0.75] (3,4) -- (3,5);
\draw[invarrow=0.87] (4,3) -- (4,5);
\draw (1,2) -- (1,3);
\draw (2,2) -- (2,3);
\draw[invarrow=0.75] (0,4) -- (0,5);
\draw[invarrow=0.87] (-1,3) -- (-1,5);
% spectral top left
\node[above] at (-1,5) {$\scriptstyle q y$};
\node[above] at (0,5) {$\scriptstyle q^2 y$};
\node[above] at (1,5) {$\scriptstyle q^3 y$};
% spectral top right
\node[above] at (2,5.7) {$\scriptstyle y$};
\node[above] at (3,5.7) {$\scriptstyle q y$};
\node[above] at (4,5.7) {$\scriptstyle q^2 y$};
% indices 
\node[above] at (2,5) {$\scriptstyle j$};
\node[above] at (3,5) {$\scriptstyle j$};
\node[above] at (4,5) {$\scriptstyle j$};
\node[below] at (2,2) {$\scriptstyle j$};
\node[below] at (3,2) {$\scriptstyle j$};
\node[below] at (4,2) {$\scriptstyle j$};
\end{tikzpicture}
\end{align}
The diagram further simplifies. Namely, the only possible non-zero matrix elements of the $\check R$-matrices corresponding to the two crosses on the right are of the form:
\begin{align}\label{eq:Rjjjj}
\begin{tikzpicture}[baseline=(current  bounding  box.center)]
    \rcheckarrows{0}{0}
    \node[above] at (0,1.4) {$\scriptstyle q^3 y$};
    \node[above] at (1,1.4) {$\scriptstyle q^i y$};
    \node[above] at (0,1) {$\scriptstyle j$};
    \node[above] at (1,1) {$\scriptstyle j$};
    \node[below] at (0,0) {$\scriptstyle j$};
    \node[below] at (1,0) {$\scriptstyle j$};
\end{tikzpicture}
=\epsilon_j\frac{t^{-\epsilon_j /2}-q^{i-3} t^{\epsilon_j /2}}{1-q^{i-3}},
\qquad i=1,2
\end{align}
Replacing these two crosses by the corresponding matrix elements of $\check R$ gives us:
\begin{align}\label{eq:res Psi D}
\underset{\left\{z_1 = y, z_2 = q y,  z_{3} = q^2 y\right\}}{\normalfont{\text{Res}}}\Psi\left[E_{jj}^{\otimes 3}\right] 
=
%%%%%%
%%%%%%
(t^{1/2}-t^{-1/2})^2f(q)
\epsilon_j^2
\frac{t^{-\epsilon_j /2}-q^{-1} t^{\epsilon_j /2}}{1-q^{-1}}
\frac{t^{-\epsilon_j /2}-q^{-2} t^{\epsilon_j /2}}{1-q^{-2}}
\begin{tikzpicture}[scale=0.5,baseline=(current  bounding  box.center)]
% R checks
% \rcheck{1}{0}
% \rcheck{2}{1}
% \rcheck{0}{1}
% \rcheck{1}{2}
\rcheck{-1}{2}
\rcheck{0}{3}
\rcheckarrows{1}{4}
% Straight lines
\draw (1,2) -- (1,3);
\draw[invarrow=0.75] (0,4) -- (0,5);
\draw[invarrow=0.87] (-1,3) -- (-1,5);
% spectral top left
\node[above] at (-1,5) {$\scriptstyle q y$};
\node[above] at (0,5) {$\scriptstyle q^2 y$};
\node[above] at (1,5) {$\scriptstyle q^3 y$};
% spectral top right
\node[above] at (2,5) {$\scriptstyle y$};
% indices 
\node[right] at (2,5) {$\scriptstyle j$};
\node[below] at (2,4) {$\scriptstyle j$};
\end{tikzpicture}
\end{align}
\begin{equation}\label{eq:res Psi E}
= (t^{1/2}-t^{-1/2})^2f(q)
\epsilon_j^2
\frac{t^{-\epsilon_j /2}-q^{-1} t^{\epsilon_j /2}}{1-q^{-1}}
\frac{t^{-\epsilon_j /2}-q^{-2} t^{\epsilon_j /2}}{1-q^{-2}} \check R_1(q^{-1}) \check R_2(q^{-2}) \Tr_4 \left(\check R_3(q^{-3}) \left(\id^{\otimes 3}\otimes E_{jj}\right)\right)
\end{equation}

This computation generalizes to the case of $\Psi\left[E_{jj}^{\otimes k}\right]$ as follows. 
\begin{itemize}
    \item The circled $\check R$ matrices as in \eqref{eq:res Psi B} will appear in the row of crosses which is immediately below the middle row. In this row there are $k-1$ such matrices. Using \eqref{eq:R residue graph} for each one of them will give rise to the factor:
\[ \left( t^{1/2}-t^{-1/2}\right)^{k-1}\]    
    \item After the above step we will see $k-2$ instances where the unitarity equation \eqref{eq:unitarity graph} can be applied. This will result in the factor:
\[
f(q)^{k-2}
\]    
    and in further $k-3$ instances where the unitarity equation \eqref{eq:unitarity graph} can be applied again. This will produce the factor:
\[
f(q^2)^{k-3}
\]        
    This pattern will repeat itself. The outcome of this step of the computation produces the net factor:
\[
\prod_{l=1}^{k-2} f \left(q^l \right)^{k-1-l}
\]    
and an expression analogous to \eqref{eq:res Psi C}. 
    \item In order to arrive at the expression \eqref{eq:res Psi D} for general $k$ we use the same logic to remove $k-1$ crosses on the right and introduce the corresponding matrix elements of the $R$-matrix. By collecting the factors which were obtained so far in this computation we obtain:
\begin{align}\label{eq:res (k) Psi A}
\underset{\left\{z_1 = y, \ldots,  z_{k} = q^{k-1} y\right\}}{\normalfont{\text{Res}}}\Psi\left[E_{jj}^{\otimes k}\right]
=
(t^{1/2}-t^{-1/2})^{k-1}
\prod_{l=1}^{k-2} f \left(q^l \right)^{k-1-l}
\prod_{l=1}^{k-1}\epsilon_j\frac{t^{-\epsilon_j /2}-q^{-l} t^{\epsilon_j /2}}{1-q^{-l}}
\begin{tikzpicture}[scale=0.5,baseline=(current  bounding  box.center)]
% R checks
% \rcheck{1}{0}
% \rcheck{2}{1}
% \rcheck{0}{1}
% \rcheck{1}{2}
\rcheck{-1}{2}
\rcheck{0}{3}
\rcheckarrows{1}{4}
% Straight lines
\draw (1,2) -- (1,3);
\draw[invarrow=0.75] (0,4) -- (0,5);
\draw[invarrow=0.87] (-1,3) -- (-1,5);
% spectral top left
\node[above] at (-1,5) {$\scriptstyle q y$};
\node[above] at (0,5) {$\scriptstyle \cdots$};
\node[above] at (1,5) {$\scriptstyle q^k y$};
% spectral top right
\node[above] at (2,5) {$\scriptstyle y$};
% indices 
\node[right] at (2,5) {$\scriptstyle j$};
\node[below] at (2,4) {$\scriptstyle j$};
\end{tikzpicture}
\end{align}    
\begin{align}\label{eq:res (k) Psi B}
=&
(t^{1/2}-t^{-1/2})^{k-1}
\prod_{l=1}^{k-2} f \left(q^l \right)^{k-1-l}
\prod_{l=1}^{k-1}\epsilon_j\frac{t^{-\epsilon_j /2}-q^{-l} t^{\epsilon_j /2}}{1-q^{-l}}
\\
\times &
\check R_1(q^{-1})\cdots \check R_{k-1}(q^{-(k-1)}) 
\Tr_{k+1} \left(\check R_k(q^{-k}) \left(\id^{\otimes k}\otimes E_{jj}\right)\right)
\nonumber
\end{align}
\end{itemize}
The trace in \eqref{eq:res (k) Psi B} can be calculated using the explicit form of $\check R$ from \eqref{eq:R-matrix super}-\eqref{eq:R-check super}:
\begin{align*}
\Tr_{2} \left(\check R(q^{-k}) \left(\id\otimes E_{jj}\right)\right)=    
\epsilon_j \frac{t^{-\epsilon_j /2}-q^{-k} t^{\epsilon_j /2}}{1-q^{-k}}E_{jj}
  +
\frac{\epsilon_j(t^{-\epsilon_j/2}-t^{\epsilon_j/2})}{1-q^{-k}}
  \sum_{1\leq i\leq n+m}\delta_{i\neq j}q^{-k\delta_{i<j}}E_{ii}  
\end{align*}
where we used the identity:
\[
t^{-1/2}-t^{1/2}=\epsilon_j(t^{-\epsilon_j/2}-t^{\epsilon_j/2})
\]
Comparing the final answer for the residue with \eqref{eqn:residue (k)} gives us a formula for $X_k^{(k)}$:
\begin{align}
    \label{eq:Psi Xk}
    X^{(k)} = \prod_{l=1}^{k}\epsilon_j\frac{t^{\epsilon_j /2}-q^{l}t^{-\epsilon_j /2}}{1-q^{l}}
E_{jj}
  +
\prod_{l=1}^{k}\epsilon_j\frac{t^{\epsilon_j /2}-q^{l-1}t^{-\epsilon_j /2}}{1-q^{l}}
  \sum_{\substack{1\leq i\leq n+m\\i\neq j}} q^{k\delta_{i>j}}E_{ii}  
\end{align}
From this we read the coefficient of $E_{ii}$ and find \eqref{eq:alpha Psi}.
\end{proof}

\section*{Acknowledgments}
A.G. would like to thank Ajeeth Gunna and Paul Zinn-Justin for many interesting discussions related to the topic of the paper. A.G. gratefully acknowledges financial support from the Australian Research Council and the hospitality of MIT Department of Mathematics where part of this work was carried.

\bibliographystyle{plain}
\bibliography{literature}{}
\end{document}